\documentclass[psamsfonts]{amsart}

\usepackage{amssymb,amsthm,amsfonts}
\usepackage[all,arc]{xy}
\usepackage{enumerate}
\usepackage{mathrsfs}
\usepackage[linktocpage]{hyperref}
\usepackage{graphicx}
\graphicspath{ {images/} }
\allowdisplaybreaks

\hypersetup{
     colorlinks   = true,
     citecolor    = blue
}
\hypersetup{linkcolor = red}
\usepackage{cite}

	\newcommand{\norm}[1]{\left\|#1\right\|}
	
	\newcommand{\lnorm}[1]{\left\|#1\right\|_{-\sigma,\F}}
	\newcommand{\opnorm}[1]{{\left\vert\kern-0.25ex\left\vert\kern-0.25ex\left\vert #1 
    \right\vert\kern-0.25ex\right\vert\kern-0.25ex\right\vert}}

\newcommand{\N}{\mathbb{N}}
\newcommand{\R}{\mathbb{R}}
\newcommand{\C}{\mathbb{C}}

\newcommand{\Q}{\mathbb{Q}}
\newcommand{\Z}{\mathbb{Z}}
\newcommand{\T}{\mathbb{T}}
\newcommand{\U}{\mathcal{U}}
\newcommand{\LL}{\mathcal{L}}

\newcommand{\I}{\mathcal{I}}
\newcommand{\HH}{\mathcal{H}}
\newcommand{\Ad}{\text{Ad}}
\newcommand{\ad}{\text{ad}}
\newcommand{\ran}{\text{Ran}}
\newcommand{\Id}{\text{Id}}

\newcommand{\II}{\mathfrak{I}}

\newcommand{\Ind}{\text{Ind}}
\newcommand{\F}{\mathcal{F}}

\newcommand{\K}{\mathcal{K}}
\newcommand{\n}{\mathfrak{n}}
\newcommand{\m}{\mathfrak{m}}

\newcommand{\G}{\mathcal{G}}

\newcommand{\uu}{\mathfrak{U}}

\newcommand{\OO}{\mathcal{O}}

\newcommand{\s}{\textbf{s}}

\newcommand{\Lra}{\Longrightarrow}

\newtheorem{thm}{Theorem}[section]
\newtheorem{theorem}[thm]{Theorem}
\newtheorem{cor}[thm]{Corollary}
\newtheorem{corollary}[thm]{Corollary}

\newtheorem{proposition}[thm]{Proposition}
\newtheorem{lemma}[thm]{Lemma}

\newtheorem*{remark}{Remark}

\theoremstyle{definition}
\newtheorem{definition}[thm]{Definition}

\numberwithin{equation}{section}
\numberwithin{figure}{section}	

\usepackage{tikz}
\usepackage{chngcntr}
\counterwithout{equation}{section}

\theoremstyle{remark}

\usepackage{alphalph,cite}

\title{Effective Equidistribution for generalized higher step nilflows}

\author[Minsung Kim]{Minsung Kim*}

				\address{\parbox{\linewidth}{Department of Mathematics, University of Maryland, College Park, MD 20742, USA\\
				\\
			Faculty of Mathematics and Computer Science, Nicolaus Copernicus University, Chopina 12/18, 87-100 Toru\' n, Poland}}			
										\email{minsungdream@gmail.com}

\date{\today}

\begin{document}


\begin{abstract}
The main results of this paper are to prove bounds for ergodic averages for nilflows on general higher step nilmanifolds. Under Diophantine condition on the frequency of a toral projection of the flow, we prove that almost all orbits become equidistributed at the polynomial speed. We analyze the rate of decay which is determined by the number of steps and structure of general nilpotent Lie algebras. Main result follows from the technique over controlling scaling operators in irreducible representations and  measure estimation on close return orbit on general nilmanifolds. 
\end{abstract}

\maketitle


\tableofcontents

\section{Introduction}
By a general results of B. Green and T. Tao \cite{GT12}, all orbits of Diophantine flows on any nilmanifold become equidistributed at polynomial speed. Their approach is an extension of Weyl's method, based on induction on the number of steps, but the rate of decay in their theorem is not explicit and presumably far from optimal. L. Flaminio and G. Forni also established estimates for the \emph{quadratic} polynomial speed of equidistribution of nilflows on \emph{higher step} nilmanifolds \cite{FF14} called Quasi-abelian (Filiform). It is the simplest class of nilmanifolds of arbitrarily higher step structures, and it has an  application in proving the bound of an exponential sum called \emph{Weyl sum}. It is notable that the bound obtained from Flaminio-Forni is established only almost everywhere but comparable with the results by T.D Wooley \cite{T15} from the number theory (see also \cite{BDG15}). 

In this paper, we extend the result for Quasi-abelian to a non-renormalizable class of nilflows on higher step of nilmanifolds under Diophantine conditions  on the frequencies of their toral projections (see Definition \ref{dio1}). 

For any $\alpha = (\alpha_1,\cdots,\alpha_n) \in \R^n$, for any $N \in \N$ and every $\delta>0$, let 
$$R_\alpha(N,\delta) = \{r\in [-N,N] \cap \Z \mid |r\alpha|_1\leq \delta^{1/n},\cdots, |r\alpha|_n\leq \delta^{1/n}\}. $$

For every $\nu >1 $, let $D_n(\nu) \subset (\R\backslash \Q)^n$ be the subset defined as follows: the vector $\alpha \in D_n(\nu)$ if and only if there exists a constant $C(\alpha) >0$ such that, for all $N \in \N$ for all $\delta >0$, 
\begin{equation}\label{rdio}
\# R_\alpha (N,\delta) \leq C(\alpha)\max\{N^{1-\frac{1}{\nu}}, N\delta\}. 
\end{equation}
Diophantine condition $D_n(\nu)$ contains the set of simultaneously Diophantine vectors so that $D_n(\nu)$ has a full measure for sufficiently large $\nu \geq 1$.

For a set of generator $\mathfrak{G}_\alpha$ of $\n$, there exists an element $X_\alpha \in \mathfrak{G}_\alpha$ and codimension 1 ideal $\II$ such that $X_\alpha \notin \II \subset \n$.
We assume that the Lie algebra satisfies  \emph{the transversality condition} if 
$$ \langle \mathfrak{G}_\alpha \rangle + \ran(\ad_{X_{\alpha}})+C_\II(X_{\alpha}) = \n$$
where $C_\II(X_{\alpha}) = \{Y \in \II \mid [Y,X_{\alpha}] = 0\}$ is centralizer.

Under this hypothesis, the speed of ergodic average of nilflows $(\phi_{X_{\alpha}}^t)$ under Diophantine conditions $\alpha \in D_n(\nu)$ is polynomial for \emph{almost all} points, as a function of step size and total number of elements of Lie algebras. 

\begin{thm}\label{main} Let $(\phi_{X_{\alpha}}^t)$ be a nilflow on a k-step nilmanifold M on n+1 generators such that the projected toral flow $(\bar \phi_{X_{\alpha}}^t)$ is a linear flow with frequency vector $\alpha:= (1,\alpha_1, \cdots, \alpha_{n} ) \in \R \times \R^{n}$. Assume that the Lie algebra satisfies the transversality condition and $\alpha \in D_n(\nu)$ for some $1\leq \nu \leq \frac{k}{2}$. Then, there exists a Sobolev norm $\norm{\cdot}$ on the space $C^\infty(M)$ of smooth function on M and for every $\epsilon>0$ there exists a positive measurable function $K_{\epsilon} \in L^p(M) $ for all $p \in [1,2)$, such that the following bound holds. For every smooth zero-average function $f \in C^\infty(M)$, for every $T \geq 1$, for almost all $x \in M$,
$$\left|\frac{1}{T}\int_0^T f\circ \phi_{X_{\alpha}}^t(x)dt \right| \leq K_{\epsilon}(x)T^{-\frac{1}{3S_{\n}(k)}+\epsilon} \norm{f}$$
 where $S_{\n}(k)$ is a higher order polynomial determined by structure of $\n$. Specifically, if $n_i$ is the number of elements in $\n$ with step size $i$,
\begin{equation}
S_{\n}(k): = (n_1-1)(k-1) + n_2 (k-2) + .... + n_{k-1}.
\end{equation}

\end{thm}

In the general higher step nilmanifold, no renormalization for nilflows is known.
Instead, based on the theory of unitary representations for the nilpotent Lie group (Kirillov theory), it is possible to choose a proper scaling operator on the space of invariant distributions. 
Compared to the earlier work on the Quasi-abelian case \cite{FF14}, the main novelty of our results lies in generalization of the scaling method to the general Lie algebra satisfying transversality conditions.

The transversality condition enables the measure estimate (section \ref{sec;AWE}) for the return orbit. This condition is sufficient, and in principle, there is no obstructions to a generalisation to arbitrary nilflows with Diophantine frequencies and all points $x \in M$, except that this would require new approaches to estimation other than a Borel-Cantelli type argument.  On the other hand, the necessity of the condition explains that the total number of elements in the basis cannot grow too fast as the step size gets larger: it grows almost linear in the number of steps and generators.

We can view this phenomena in the following way: if the growth of the number of elements in lower steps (generated by basis) are too large, then it lacks the dimensions to count the measure of return orbit on the transverse manifold. For instance, we observe this phenomenon in free nilpotent Lie algebras. Even a small number of generators creates a large number of elements in the lower level under small steps of commutations, which behave in a completely different way than strictly triangular and Quasi-abelian. We present such an example in the appendix to motivate our condition.

The above theorem is appreciated by its corollary on \emph{strictly triangular nilmanifold} $M_{k}^{(k)}$. Let $N_{k}^{(k)}$ denote a step $k$ nilpotent Lie group on $k$ generators. Up to isomorphism, $N_{k}^{(k)}$ is the group of upper triangular unipotent matrices
\begin{equation}\label{triaaa}
[x_1X_1,\cdots x_nX_n, \cdots y^{(j)}_iY^{(j)}_i \cdots zZ ]
:= \begin{pmatrix}
    1       & x_{1} & \cdots & \cdots & z \\
    0       & 1 & x_{2} & y^{(j)}_i & \vdots \\
    \vdots & \vdots & \ddots & \ddots & \vdots \\
        0       & 0 & 0 & 1 & x_n\\
    0       & 0 & 0 & 0 & 1
\end{pmatrix}, \hspace{10pt} x_i,y^{(j)}_i,z \in \R
\end{equation}
with one dimensional center. Next result states that the rate of equidistribution for nilflows on triangular nilmanifold $M_{k}^{(k)}$ decays  at a polynomial speed with exponent, \textit{cubically} as a function of number of steps. 

\begin{cor}\label{1.2}
Let $(\phi_{X_\alpha}^t)$ be a nilflow on $k$-step strictly triangular nilmanifold M on $k$ generators such that the projected toral flow $(\bar \phi_{X_\alpha}^t)$ is a linear flow with frequency vector $\alpha:= (1,\alpha_1, \cdots, \alpha_{k-1} ) \in \R \times \R^{k-1}$. Under the condition that $\alpha \in D_n(\nu)$ for some $1\leq \nu \leq \frac{k}{2}$, there exists a Sobolev norm $\norm{\cdot}$ on the space $C^\infty(M_k^{(k)})$ of smooth function on $M_k^{(k)}$ and there exists a positive measurable function $K_\epsilon \in L^p(M_k^{(k)}) $ for all $p \in [1,2)$ and for every $\epsilon>0$, such that the following bound holds. For every smooth zero-average function $f \in C^\infty(M_k^{(k)})$, for almost all $x \in M$ and for every $T \geq 1$,
$$\left|\frac{1}{T}\int_0^T f\circ \phi_{X_\alpha}^t(x)dt \right| \leq K_\epsilon(x)T^{-\frac{1}{3(k-1)(k^2+k-3)}+\epsilon} \norm{f}.$$
\end{cor}

We also establish the uniform bound for step-3 strictly triangular nilmanifold case. The result holds for \textit{all points} by estimating the width with counting  close return time directly under Roth-type Diophantine condition. The step-3 case (as well as the filiform case, \cite{F16}) is a good example to derive a simplified proof. 

\begin{thm}\label{step3} Let $(\phi_X^t)$ be a nilflow on 3-step nilmanifold M on 3 generators such that the projected toral flow $(\bar \phi_X^t)$ is a linear flow with frequency vector $v:= (1,\alpha ,\beta )$ of Diophantine condition with exponent $\nu = 1+\epsilon$ for all $\epsilon>0$. For every $s> 26$, there exists a constant $C_s$ such that for every zero-average function $f \in W^s(M)$, for all $(x,T) \in M \times \R$, we have
$$\left|\frac{1}{T}\int_0^T f\circ \phi_X^t(x)dt \right| \leq C_sT^{-1/12+\epsilon} \norm{f}_s.$$
\end{thm}

In the last section, we present exponential mixing of hyperbolic nilautomorphism as a main application.  Exponential mixing of ergodic automorphism and its applications to the Central Limit Theorem on compact nilmanifolds was proven by A. Gorodnik and R. Spatzier  \cite{AR14}. Their approach was based on the result of Green and Tao \cite{GT12}, and mixing follows from the equidistribution of the exponential map called \emph{box map} satisfying certain Diophantine conditions. Our result also shows specific exponent of exponential mixing depending on the structure of nilmanifolds, which follows from equidistribution results and renormalization argument of partial hyperbolic automorphism (hyperbolic on the projected torus). However, they are limited to a special class of nilautomorphisms due to lack of hyperbolicity on the group of automorphisms on general nilpotent Lie algebras. (Cf. triangular step 3 with 3 generators.) \medskip

\textbf{Open problem.} Our work leaves open natural questions.\medskip

\textit{Problem 1.} Can we find the \emph{uniform} bounds of the average width under the transversality condition? That is, prove the Theorem \ref{main} with uniform bound for \emph{all} points. \medskip

\textit{Problem 2.} Give an effective bound of ergodic averages on any nilmanifolds. \medskip

It is still conjectured that the result may hold for \emph{all points on any nilmanifolds}. However, our argument for averaged width can not be improved to control the slow growth of displacement for the return orbit on general higher step cases. \\


This paper is organized as follows. In section \ref{sec;2}, we define structures of nilmanifolds and nilflows. In section  \ref{sec;3}, we carry out Sobolev estimates on solutions of the cohomological equation and on invariant distributions as an application of Kirillov theory of unitary representations of nilpotent groups. In section  \ref{sec;4}, we introduce the notion of average width and prove a Sobolev trace theorem. In section  \ref{sec;AWE}, we prove an effective equidistribution theorem for good points by a Borel-Cantelli argument.  In section \ref{sec;sec6}, we prove bounds on the average width of an orbit segment of nilflows by gluing all the irreducible representations. In section  \ref{7}, we introduce a uniform width estimate under a Roth-type Diophantine condition based on counting return time directly to avoid good points argument. Finally, in section  \ref{sec;8}, as an application, we prove the mixing of nilautomorphism. \\

\textbf{Acknowledgement.} The author deeply appreciates Giovanni Forni for his illuminating suggestions and guidance. He is grateful to Livio Flaminio and Rodrigo Trevi\~no for giving several comments to improve the draft. Part of the paper was written when the author visited the Institut de Mathematiques de Jussieu-Paris Rive Gauche in Paris, France. He also acknowledges Anton Zorich for hospitality during the visit. He is grateful to Xuesen Na, Davi Obata, and Davide Ravotti for helpful discussion and thanks Jacky Jia Chong,  J.T Rustad and Lucia D. Simonelli for their encouragement. Lastly, the author is thankful to the referee for helpful comments and suggestions for improvement in the presentation of this work. This research was partially supported by the NSF grant DMS 1600687 and by the Centre of Excellence “Dynamics, mathematical 
analysis and artificial intelligence” at Nicolaus Copernicus University in Toruń.

\section {Nilflows on higher step nilmanifold}\label{sec;2}
In this section we review nilpotent Lie algebras, groups and basic structures. We recall Kirillov theory and representation theory.

\subsection {Background of nilpotent Lie group and Lie algebras} 
Let $N$ be a connected, simply connected $k$-step nilpotent Lie group with Lie algebra $\n$ with $n+1$ generators. Let $\Gamma$ be a co-compact lattice in $N$. The quotient $N/\Gamma$ is then a compact nilmanifold $M$ which $N$ acts on the left by translations. 
Denote $\mu$ the $N$-invariant measure on $M$.

For $j = 1,\cdots, k$, let $\n_j$ denote the descending central series of $\n$:
\begin{equation}\label{eqn;liealg}
\n_1 = \n, \n_2 = [\n,\n], \cdots, \n_j = [\n_{j-1},\n], \cdots, \n_k \subset Z(\n)
\end{equation}
where $Z(\n)$ is the center of $\n$. In this setting, there exists a strong Malcev basis through the filtration $(\n_{i})_{i=1}^k$ strongly based at the lattice $\Gamma$ (See Theorem 1.1.13 and 5.16 of \cite{CG90}). That is, given a basis 
\[\F = \{\xi^{(1)}, \eta_1^{(1)},\cdots, \eta_{n_1}^{(1)}, \cdots,\eta_1^{(k)},\cdots \eta_{n_k}^{(k)} \}\] 
with $\xi := \xi^{(1)} \in \n_1 \backslash \n_2$ and $\eta_i^{(l)} \in \n_l \backslash \n_{l+1} $ for $i = 1,\cdots, n_l$, we have

 \begin{enumerate}

 \item If we drop the first $l$ elements of the basis, we obtain a basis of a subalgebra $\n$ of codimension $l$ ;
 
 \item For each $j$,  the elements in order $\eta_1^{(j)},\cdots \eta_{n_j}^{(j)},\cdots \eta_1^{(k)},\cdots \eta_{n_k}^{(k)}$ form a basis of an ideal $\n_j$ of $\n$;
 
 \item The lattice $\Gamma$ is generated by $\{x, y_{1}^{(1)}, \cdots, y_{n_k}^{(k)} \}$ with
 $$ x := \exp(\xi),\quad  \ y_{n_i}^{(j)}:= \exp(\eta_{n_i}^{(j)}).$$

 \end{enumerate}

For any nilpotent Lie algebra $\n$, there exists a codimension 1 subalgebra  $\II$ where 
\[\n = \R \xi\oplus \II. \] 
Then $\II$  is an ideal and $[\n,\n] \subseteq \II$. 
(See \cite[Chapter 3, p.12]{H73} and \cite[Lemma 1.1.8]{CG90}). For convenience, we write dimension  $a = \dim (\II) = n_1+\cdots + n_k$ and set $n = n_1$.

\begin{definition} An \emph{adapted basis} of the Lie algebra $\n$ is an ordered basis $(X,Y) := (X,Y_1,\cdots, Y_a) $ of $\n$ such that $X \notin \II$ and $Y:= (Y_1, \cdots, Y_a)$ is an basis of $\II$.

A \emph{strongly adapted} basis $(X,Y) := (X,Y_1,\cdots, Y_a) $ is an adapted basis such that the following holds:

\begin{enumerate}

\item 
the system $(X,Y_1, \cdots, Y_n)$  is a system of generators of $\n$, hence its projection is a basis of the Abelianisation $\n/[\n,\n]$ of the Lie algebra $\n$:
\item 
The system $(Y_{n+1},\cdots, Y_a)$ is a basis of the ideal $[\n,\n]$.
\end{enumerate}

\end{definition}

\subsection {Nilmanifold and nilflows} 
Every nilmanifold $M$ is a fiber bundle over a torus. 
In fact, the group $N^{ab} = N/[N,N]$ is Abelian, connected and simply connected, hence isomorphic to $\R^{n+1}$ and $\Gamma^{ab} = \Gamma/[\Gamma, \Gamma]$ is a lattice in $N^{ab}$. Thus, we have a natural projection $pr_1: M \rightarrow \T^{n+1}$.

We introduce two fibrations of nilmanifold $M$. Let $M_2 \simeq N_2/\Gamma_2$ with $N_2 = \exp(\n_2)$ and its lattice $\Gamma_2$, then there exists an exact sequence
\begin{equation}\label{exact1}
0 \rightarrow M_2 \rightarrow M \xrightarrow {pr_1} \T^{n+1} \rightarrow 0.
\end{equation}

Another fibration arises from the canonical homomorphism $N \rightarrow N/N' \approx \langle \exp \xi \rangle$. For $\theta \in \T^1$, the fiber $M_\theta^a = {pr_2}^{-1}(\theta)$ is local section of the nilflow on $M$. 
\begin{equation}\label{fiber}
0 \rightarrow M_\theta^a \rightarrow M \xrightarrow {pr_2} \T^{1} \rightarrow 0.
\end{equation}


On nilmanifold $M$, the nilflow $\phi^t_X$ generated by $X \in \n$ is the flow obtained by the restriction of this action to the one-parameter subgroup $(\exp tX)_{t \in \R}$ of $N$, with
$$\phi^t_X(x) = x\exp(tX), \quad x \in M, \ t \in \R. $$

The projection $\bar{X}$ of $X$ is the generator of a linear flow $\psi_{\bar{X}} := \{\psi^t_X\}_{t \in \R}$ on $\T^{n+1} \approx \R^{n+1} \backslash \bar{\Gamma}$ defined by 
$$\psi^t_{\bar{X}}(x_1,\cdots,x_{n+1}) = (x_1+tv_1,\cdots ,x_{n+1}+tv_{n+1}).$$
The canonical projections $ {pr_1}:M \rightarrow \T^{n+1}$ intertwines the flows $\phi^t_X$ and $\psi^t_{\bar{X}}$.

%
%

We recall the following:
\begin{theorem}\cite{AGH63} The followings are equivalent. 
\begin{enumerate}
\item The nilflow $(\phi^t_X)$ on $M$ is ergodic.
\item The nilflow $(\phi^t_X)$ on $M$ is uniquely ergodic.
\item The nilflow $(\phi^t_X)$ on $M$ is minimal.
\item The projected flow $(\psi^t_{\bar{X}})$ on $\overline M = N^{ab} / \Gamma^{ab}  \simeq \T^{n+1}$ is an irrational linear flow.
\end{enumerate}
\end{theorem}

Notation. Consider the set of indices 
\begin{align*}
J &:= \{(i,j) \mid 1 \leq i \leq  n_j,  1 \leq j \leq k   \};\\
J^+ &:= \{(i,j) \mid 1 \leq i \leq  n_1, \ j=1   \};\\
 J^{-} &:= \{(i,j) \mid 1 \leq i \leq  n_j ,\  j > 1   \};\\
  J^{-2} &:= \{(i,j) \mid 1 \leq i \leq  n_j ,\  j > 2   \}.
\end{align*}

Let $\alpha = (\alpha_i^{(j)}) \in \R^{J}$ and $X:= X_{\alpha}$ be the vector field on $M$ defined 
\begin{equation}\label{flow1} 
X_{\alpha}:= \log[x^{-1}\exp(\sum_{(i,j) \in J}\alpha^{(j)}_i\eta^{(j)}_i)], \quad x = \exp(\xi).
\end{equation} 
and equivalently we write
\begin{equation}\label{floww} 
X_{\alpha}:= -\xi+\sum_{(i,j) \in J}\alpha^{(j)}_i\eta^{(j)}_i.
\end{equation} 
 For $\theta \in \T^1$ let $M_\theta^a = {pr_2}^{-1}(\theta)$ denote a fiber over $\theta \in \T^1$ of the fibration $pr_2$.   Transverse section $M_\theta^a$ of the nilflow $\{\phi^t_{X_{\alpha}}\}_{t\in \R}$, ${\bf s} = (s_i)_{i=1}^a \in \R^a$ corresponds to
\begin{align*}
 \{ \Gamma \exp(\theta \xi) \exp( \sum_{i=1}^a{s_i \eta_i } ) \mid (s_i) \in \R^a \} = \{ \Gamma  \exp(e^{ad(\theta \xi)} \sum_{i=1}^a{s_i \eta_i } )\exp(\theta \xi) \mid (s_i) \in \R^a \}.
\end{align*}

\begin{lemma}\label{bgg} The flow $(\phi^t_{X_{\alpha}})_{t\in \R}$ on M is isomorphic to the suspension of its first return map $\Phi_{\alpha, \theta}: M_\theta^a \rightarrow M_\theta^a$. For every $(i,j) \in J$, there exists a polynomial $p_{i,N}^{(j)}(\alpha,  {\bf s})$ for ${\bf s} \in \R^a$ such that return map $\Phi_{\alpha, \theta}$ is given by the following: \\

In the  coordinate of ${\bf s} = (s^{(j)}_i)$ for  $\Gamma \exp(\theta \xi)\exp(\sum_{(i,j)\in J} s^{(j)}_i\eta^{(j)}_i) \in M_\theta^a$,

\begin{multline}\label{ret1}
\Phi_{\alpha, \theta}( {\bf s} ) =  \Gamma \exp(\theta \xi)\exp(\sum_{(i,j) \in J} (s^{(j)}_i + \alpha^{(j)}_i)\eta^{(j)}_i \\
+  [\sum_{(i,j) \in J} s^{(j)}_i\eta^{(j)}_i , X_{\alpha}] + \sum_{(i,j) \in J^{2-}} p_i^{(j)}(\alpha, s)\eta^{(j)}_i)
\end{multline} 
and for $r \in \N$,
\begin{multline}\label{retN}  
\Phi^r_{\alpha, \theta}( {\bf s}) =  \Gamma \exp(\theta \xi)\exp(\sum_{(i,j) \in J} (s^{(j)}_i + r\alpha^{(j)}_i)\eta^{(j)}_i   \\
+ [\sum_{(i,j) \in J} s^{(j)}_i\eta^{(j)}_i , rX_{\alpha}] + \sum_{(i,j) \in J^{2-}} p_{i,r}^{(j)}(\alpha, s)\eta^{(j)}_i).
\end{multline} 
\end{lemma}
\proof

By (\ref{flow1}), we have

\begin{align*}
\exp(\sum_{(i,j)\in J } s^{(j)}_i\eta^{(j)}_i)\exp(X_{\alpha}) & =  \exp(\sum_{(i,j) \in J} s^{(j)}_i\eta^{(j)}_i)x^{-1} \exp(\sum_{(i,j) \in J} \alpha^{(j)}_i\eta^{(j)}_i)\\
& = x^{-1}\exp (e^{\ad(\xi)}{\sum_{(i,j) \in J} s^{(j)}_i\eta^{(j)}_i })  \exp(\sum_{(i,j) \in J} \alpha^{(j)}_i\eta^{(j)}_i).
\end{align*}
By Baker-Campbell-Hausdorff formula, there exist polynomial $p_i^{(j)}(\alpha, s)$ with 
\begin{align*}
\exp(\sum_{(i,j)\in J } s^{(j)}_i\eta^{(j)}_i)\exp(X_{\alpha}) &  =  x^{-1}\exp( \sum_{(i,j) \in J} (s^{(j)}_i + \alpha^{(j)}_i)\eta^{(j)}_i   \\
& + [\sum_{(i,j) \in J} s^{(j)}_i\eta^{(j)}_i , X_{\alpha}] + \sum_{(i,j) \in J^{2-}} p_i^{(j)}(\alpha, s)\eta^{(j)}_i).
\end{align*}

Since $x \in \Gamma$, we conclude 
\begin{align*}
 & \Gamma \exp(\theta \xi)\exp(\sum_{(i,j) \in J} s^{(j)}_i\eta^{(j)}_i)\exp(X_{\alpha})  \\
&  = \Gamma\exp(\theta \xi) \exp( \sum_{(i,j) \in J} (s^{(j)}_i + \alpha^{(j)}_i)\eta^{(j)}_i 
   + [\sum_{(i,j) \in J} s^{(j)}_i\eta^{(j)}_i , X_{\alpha}] + \sum_{(i,j) \in J^{2-}} p_i^{(j)}(\alpha, s)\eta^{(j)}_i).
\end{align*}
The formula implies that $t =1$  is a return time of the restriction of the flow to $M_\theta^a \subset M$. The formula for $r \in \N$ follows from induction.
\qed

\subsection{Kirillov theory and Classification}\label{sec;Kirillov}
 Kirillov theory yields the complete classification of irreducible unitary representation of $N$. All the irreducible unitary representation of nilpotent Lie groups are parametrized by the coadjoint orbits $\OO \subset \n^*$. A \emph{polarizing} (or maximal subordinate) subalgebra for $\Lambda \in \n^*$ is a maximal isotropic subspace $\m \subset \n$ which is a subalgebra of $\n$.
 It is well known that for any $\Lambda \in \n^*$ there exists a polarizing subalgebra $\m$ for a nilpotent Lie algebra $\n$. (See \cite[Theorem 1.3.3]{CG90}) Let $\mathfrak{m}$ be a polarizing subalgebra for given linear form $\Lambda \in \mathfrak{n^*}$. Then, the character $\chi_{\Lambda,\m}: \exp{\m}\rightarrow S^1$ is defined  
 $$\chi_{\Lambda,\m}(\exp Y) = e^{2\pi \iota \Lambda(Y) }.$$
Given pair $(\Lambda,\m)$, we associate unitary representation 
 $$\pi_{\Lambda} = \Ind_{\exp \m}^{N}(\chi)$$
where induced representation $\pi_{\Lambda}$ is defined by 
$$\pi_{\Lambda}(x)f(g) = f(g\cdot x), \quad \text{ for } x \in N \text{ and } f \in H_{\pi_{\nu,\m}}. $$

These unitary representations are irreducible up to equivalence, and all unitary irreducible representations are obtained in this way. It is known that $\Lambda$ and $\Lambda'$ belong to the same coadjoint orbit if and only if $\pi_{\Lambda,\m} $ and $\pi_{\Lambda',\m'} $ are unitarily equivalent and $\pi_{\Lambda,\mathfrak{m}}$ is irreducible whenever $\m$ is maximal subordinate for $\Lambda$. 
We write $\pi_\Lambda \simeq \pi_{\Lambda'}$ if $\Lambda$ and $\Lambda'$ are in the same coadjoint orbit. 

Since the action of $N$ on $M$ preserves the measure $\mu$, we obtain a unitary representation $\pi$ of $N$. The regular representation of $L^2(M)$ of $N$ decomposes as a countable direct sum (or direct integral) of irreducible, unitary representation $H_\pi$, which occurs with at most finite multiplicity
\begin{equation}\label{eqn;repL2}
L^2(M,d\mu) = \bigoplus_{\pi}  H_{\pi}.
\end{equation}

  The derived representation $\pi_{*}$ of a unitary representation $\pi$ of $N$ on a Hilbert space $H_\pi$ is a representation of the Lie algebra $\n$ on $H_\pi$ defined as follows. For every $X \in \n$, 
\begin{equation}
\pi_{*}(X)  = \lim_{t\rightarrow 0}(\pi (\exp tX) - I )/t.
\end{equation}
We recall that a vector $v \in H_\pi$ is of \emph{$C^\infty$-vectors} in $H_\pi$ for representation $\pi$ if the function $g \in N \mapsto \pi(g)v \in H_\pi$  is of class $C^\infty$ as a function on $N$ with values in a Hilbert space.

\begin{definition}
The space of  \emph{Schwartz functions} on $\R$ with values of $C^\infty$ vectors for the representation $\pi'$ on $H'$ is denoted 
$\mathcal{S}(\R, C^\infty(H'))$. It is endowed with the Fr\'echet topology induced by the family of semi-norms
$$\{\norm{\cdot}_{i,j,Y_1,Y_2,\cdots,Y_m}  \mid i,j,m\in \N \text{ and } Y_1,\cdots,Y_m \in \n'\}$$
and defined as follows: for all $f \in \mathcal{S}(\R, C^\infty(H'))$,
$$ \norm{f}_{i,j,Y_1,Y_2,\cdots,Y_m} := \sup_{t \in \R} \norm{(1+t^2)^{j/2}\pi_{*}'(Y_1)\cdots\pi_{*}'(Y_m)f^{(i)}(t)}_{H'},\quad t\in \R.$$
\end{definition}

\begin{lemma}\label{lem;schwartz}\cite[Lemma 3.4]{FF07}
As a topological vector space 
$$C^\infty(H_\pi) = \mathcal{S}(\R, C^\infty(H'))$$
where $\mathcal{S}(\R, C^\infty(H'))$ is Schwartz space. 
\end{lemma}
Suppose that $\n = \R X\oplus \II $ with its codimension 1 ideal $\II$, and $N = \R \ltimes N'$ with a normal subgroup $N'$ of $N$. Let $\pi'$ be a unitary irreducible representation of $N'$ on a Hilbert space $H'$. Each irreducible representation $H_\pi$ is unitarily equivalent to $L^2(\R,H',dx)$, and
derived representation of $\pi_*$ of the induced representation $\pi = \Ind_{N'}^N(\pi')$ has the following description.

For $f \in L^2(\R,H',dx)$, the group $\R$ acts by translations  and its representation is polynomial in the variable $x$. For any $Y \in \n'$,
\begin{align}\label{12}
\begin{split}
(\pi_{*}(Y)f)(x) &= \sum_{j=0}^{d_Y} \frac{1}{j!} \pi'_{*}(\ad_X^j(Y)f(x)\\
&= \iota P_Y(x)f(x) = \iota\sum_{j=0}^{d_Y} \frac{1}{j!} (\Lambda\circ\ad_X^jY)x^jf(x).
\end{split}
\end{align}

We define its degree $d_Y \in \N$ with respect to the representation $\pi_{*}(Y)$ to be the degree of polynomial. 
Let $(d_1,\cdots, d_a)$ be the degrees of the elements $(Y_1, \cdots, Y_a)$ respectively. The degree of representation $\pi$ is defined as the maximum of the degrees of the elements of any basis.

\section {The cohomological equation}\label{sec;3}
In this section, we prove a (priori) Sobolev estimate on the Green's operator for the cohomological equation $Xu = f$ of nilflow with generator $X$. We estimate bounds of Green's operator on Sobolev norm and on scaling of invariant distributions.

\subsection{Distributions and Sobolev space}
Let $L^2(M)$ be the space of complex-valued, square integrable functions on $M$. Given ordered basis $\F$ of $\n$, the \emph{transverse Laplace-Beltrami} operator is second-order differential operator defined by 
$$\Delta_{\F} =  -\sum_{i=1}^a Y_i^2,\  Y_i \in \II.$$
For any $\sigma \geq 0$, let $|\cdot|_{\sigma,\F}$ be the transverse Sobolev norm defined as follows: for all functions $f \in C^\infty(M)$, let
$$|f|_{\sigma,\F} := \norm{(I+\Delta_{\F})^\frac{\sigma}{2}f}_{L^2(M)}.$$
Equivalently, 
$$|f|_{\sigma,\F} = (\norm{f}^2_2+ \sum_{1\leq m \leq \sigma}\norm{Y_{j_1}\cdots Y_{j_m}f}_2^2)^\frac{1}{2}, \quad Y_{j_m} \in \II.$$ 

The completion of $C^\infty(M)$ with respect to the norm $|\cdot|_{\sigma,\F}$ is denoted $W^\sigma(M,\F)$ and the distributional dual space (as a space of functional with values in $H'$) to $W^\sigma(M)$ is denoted
$$W^{-\sigma}(M,\F): = (W^{\sigma}(M,\F))'.$$

We denote $C^\infty(H_\pi)$ the space of $C^\infty$ vectors of the irreducible unitary representation $\pi$. 
 Following notation in \eqref{eqn;repL2},
let $W^\sigma(H_\pi) \subset H_\pi$ be the Sobolev space of vectors, endowed with the Hilbert space norm in the maximal domain of the essential self-adjoint operator $(I+\pi_{*}( \Delta_{\F}))^\frac{\sigma}{2}$.  I.e, for every $f \in C^\infty(H_\pi)$ and $\sigma>0$,
$$|f|_{\sigma,\F} := \left( \int_\R \norm{(1+\pi_{*}(\Delta_{\F}))^{\frac{\sigma}{2}}f(x)}^2_{H'}dx \right)^{1/2}$$
where $\pi_{*}(\Delta_{\F})$ is determined by derived representations.

\subsection{A priori estimates }
The distributional obstruction to the existence of solutions of the cohomological equation 
$$Xu = f, \quad f\in C^\infty(H_\pi)$$
in a irreducible unitary representation $H_\pi$ is the normalized $X$-invariant distribution.

\begin{definition} For any $X \in \n$, the space of $X$-invariant distributions for the representation $\pi$ is the space $\I_X(H_\pi)$ of all distributional solutions $D \in D'(H_\pi)$ of the equation $\pi_{*}(X)D = XD = 0$. Let 
$$\I^\sigma_X(H_\pi): = \I_X(H_\pi) \cap W^{-\sigma}(H_\pi) $$
be the subspace of invariant distributions of order at most $\sigma \in \R^+$.

By Lemma 3.5 of \cite{FF07}, each invariant distribution $D$ has a Sobolev order equal to $1/2$, i.e $D \in W^{-\sigma}(H_\pi)$ for any $\sigma > 1/2$.
\end{definition}


For all $\sigma > 1$, let $\K^\sigma(H_\pi) = \{f \in W^\sigma(H_\pi) \mid D(f) = 0 \in C^\infty(H_\pi),  \text{ for any }  D \in W^{-\sigma}(H_\pi)  \}$
be the kernel of the $X$-invariant distribution on the Sobolev space $W^\sigma(H_\pi)$. 
The Green's operator $G_X:C^\infty(H_\pi) \rightarrow C^\infty(H_\pi)$ with
$$G_Xf(t) = \int_{-\infty}^{t}f(s)ds $$
 is well-defined on the kernel of distribution $\K^\infty(H_\pi)$  on $C^\infty(H_\pi)$. In fact, for any $f \in \K^\infty(H_\pi)$, we have $\int_{\R}f(t) dt = 0 \in C^\infty(H')$ and
$$G_Xf(t) = \int_{-\infty}^{t}f(s)ds = -\int_{t}^{\infty}f(s)ds \in C^\infty(\R,H').$$ 

Now we define generalized (complex-valued) invariant distribution on smooth vector $C^\infty(H_\pi)$.
\begin{lemma}\label{lem;ell}
The invariant distribution is generalized  in the following sense. 
For every invariant distribution $D$, there exists  a linear  functional $\ell : C^\infty(H_\pi) \rightarrow \C$  such that for every function $f \in C^\infty (H_\pi) \subset L^2(\R, H')$,
\begin{equation}\label{eqn;invdistri}
D(f) = \int_\R   \ell ( f(t) ) dt.
\end{equation}
Furthermore,  $\ell \in W^{-s}(H_\pi)$ for $s > 1/2$ and
\begin{equation}\label{eqn;commute}
\int_\R   \ell ( f(t) ) dt = \ell \left( \int_\R f(t) dt \right).
\end{equation}
\end{lemma}
\begin{proof}

We construct a linear functional $\ell$ in \eqref{eqn;invdistri} as follows. Let $\chi \in C_0^\infty(\R)$ be a smooth function with compact support with unit integral 
over $\R$.  Given an invariant distribution $D \in \I_X(H_\pi)$, let us define $\ell (v) =   D (f_v)$ for $f_v = \chi v$ and $v\in C^\infty(H')$. 

Firstly, we prove that $\ell$ is well-defined. Let $\chi_1 \neq \chi_2 \in C_0^\infty(\R)$ be functions with compact support such that $\int_\R \chi_1(t)dt =  \int_\R \chi_2(t)dt = 1$. Note that there exists $\psi \in C_0^\infty(\R)$ such that $\chi_1 - \chi_2 = \psi'$ with $\psi(t) = \int_{-\infty}^t (\chi_1(x) - \chi_2(x))dx.$

Then, we have
$$\chi_1(t) v - \chi_2(t) v = \frac{d}{dt}(\psi(t)v) = \pi_*(X)(\psi(t)v) \in C^\infty(H_\pi),$$
and  $\chi_1(t) v - \chi_2(t) v$ is a $X$-coboundary for every $v \in C^\infty(H')$. Hence, $D(\chi_1(t) v - \chi_2(t) v) = 0$, which implies that $D(\chi_1(t)v) = D(\chi_2(t)v)$. Therefore, $\ell(v)$ does not depend on the choice of $\chi$ and the functional $\ell$ is well-defined.

Next, we verify that $\ell$ is a distribution on $C^\infty(H')$. It suffices to prove that $\ell$ is bounded and continuous. 
For $v \in C^\infty(H')$ and $s > 1/2$,
$$|\ell(v)| = |D(\chi(t) {v})| \leq \norm{D}_{-s}\norm{\chi(t)({v})}_{W^s(H')}. $$

In the representation, $\pi_*(Y_i)$ acts as a multiplication of polynomial $p_i(t)$ on $L^2(\R,H')$. By definition of Sobolev norm in representation, there exists a non-zero constant  $C:= C(\chi, p_1,\cdots p_s) = \displaystyle \max_{\substack {t \in \R, j_1 + \cdots j_d = s \\ 0 \leq j_i \leq s, 1\leq i \leq d \leq s}}\{\chi(t)p^{j_1}_1(t)\cdots p^{j_d}_d(t)\}$ such that
$$\norm{\chi(t)({v})}_{W^s(H_{\pi})} \leq C\norm{{v}}_{W^s(H')}.$$
Hence, for $\ell \in W^{-s}(H')$,
$$\norm{\ell}_{-s}: = \sup_{\norm{v} = 1} \dfrac{|\ell(v)|}{\norm{v}_{W^s(H')}} \leq C \norm{D}_{-s}.$$
Therefore, $\ell$ is continuous on $C^\infty(H')$.

To prove equality \eqref{eqn;commute}, for any $f \in C^\infty(H_\pi)$, we observe
$f(t) - \chi(t)\left( \int_\R f(x) dx \right)$ has zero average, hence it is a coboundary with smooth transfer function. Since distribution $D$ is invariant under translation, 
we obtain $$D\left(f - \chi(t)\left( \int_\R f(x) dx \right)\right) = 0,$$ and 

$$D(f) = D \left(\chi(t) \int_\R f(x) dx \right)   = \ell \left( \int_\R f(t) dt \right).$$
Furthermore, 
$$\int_\R \ell (f(t))dt = \int_\R D(f_\chi(t))dt = \int_\R \left(\int_\R \chi(x)f(t)dx\right)dt = \int_\R f(t) \left(\int_\R \chi(x)dx\right)dt.$$
Since $D$ is invariant distriution for translation and $\int_\R \chi = 1$, we conclude
$$\int_\R \ell (f(t))dt = \int_\R f(t) dt = D(f) = \ell \left( \int_\R f(t) dt \right).$$
\end{proof}

Let $\OO$ be any coadjoint orbit of maximal rank.  For all $(X,Y)\in \n \times \n_{k-1}$ and $\Lambda \in \OO$, the skew-symmetric bilinear form 
$$B_\Lambda(X,Y) = \Lambda([X,Y])$$
does not depend on the choice of linear form $\Lambda \in \OO$. 

Let
\begin{equation}
\begin{aligned}
&\delta_{\OO}(X,Y) := |B_\Lambda(X,Y)| \text{ for any $\Lambda \in \OO$}, \\
&\delta_{\OO}(X) := \max\{ \delta_{\OO}(X,Y) \mid Y \in \n_{k-1} \text{ and } \norm{Y}=1 \}.
\end{aligned} 
\end{equation}

Here we recall estimates for Green's operator.
\begin{lemma}\cite[Lemma 2.5]{FF07}\label{niceop} Let $X \in \n$  and $Y \in \n_{k-1}$ be any operator such that $B_\Lambda(X,Y) \neq 0$. 
The derived representation $\pi_*$ of the Lie algebra $\n$ satisfies 
\begin{equation}\label{def;niceop}
 \pi_{*}(X) = \frac{d}{dx}, \quad \pi_{*}(Y) = 2\pi \iota B_\Lambda(X,Y) x \Id_{H'} \text{ on }L^2(\R, H',dx).
 \end{equation}
\end{lemma}

\begin{theorem}\cite[Theorem 3.6]{FF07}\label{green1}
 Let ${\delta_\OO}:=\delta_{\OO}(X)>0$, and let $\pi$ be an irreducible representation of $\n$ on a Hilbert space $H_\pi$.
 If $f \in W^s(H_\pi)$, $s >1$  and $D(f) = 0$ for all $D \in \I_X(H_\pi)$, then $G_Xf \in W^r(H_\pi)$, for all $r < (s-1)/k$ and there exists a constant $C := C(X,k,r,s),$ such that
$$ |{ G_Xf}|_{r,\F} \leq C \max\{1,\delta_\OO^{-(k-1)r-1)} \} |{f}|_{s,\F}.$$
\end{theorem}

\subsection{Rescaling method}
\begin{definition}
The \emph{deformation space} of a $k$-step nilmanifold $M$ is the space $T(M)$ of all adapted bases of the Lie algebra $\n$ of the group $N$.
\end{definition}
The renormalization dynamics is defined as the action of diagonal subgroup of the Lie group on the deformation space.
Let $\rho: =  (\rho_1,\cdots, \rho_a) \in (\R^{+})^a$ be any vector with rescaling condition $\sum_{i=1}^a {\rho_i} = 1$. Then, there exist a one-parameter subgroup $\{A_t^{\rho} \}$ of the Lie group of $SL(a+1,\R)$ defined as follows:
\begin{equation}\label{eqn;reno}
A_t^{\rho}(X,Y_1,\cdots, Y_a) = (e^tX, e^{-\rho_1t}Y_1, \cdots, e^{-\rho_at} Y_a).
\end{equation}
The renormalization group $\{A_t^{\rho} \}$ preserves the set of all adapted basis. However, it is not a group of automorphism of the Lie algebra. Therefore, on higher step nilmanifolds, the dynamics induced by the renormalization group on the deformation space is trivial. (It has no recurrent orbits.)


\begin{definition}
Given any adapted basis $\F = (X,Y)$, \emph{rescaled basis} $\F(t) = (X(t),Y(t))  = \{e^tX, e^{-\rho_1t}Y_1, \cdots, e^{-\rho_at}Y_a\}$ of $\F$ is a  basis of Lie algebra $\n$ satisfying \eqref{eqn;reno}.
\end{definition}

Let $(d_1,\cdots, d_i)$ be the degrees of the elements $(Y_1, \cdots, Y_i)$ respectively. For any $\rho = (\rho_1, \cdots, \rho_a ) \in \R^a$, let 
\begin{equation}\label{67}
\lambda_\F(\rho) := \min_{i : d_i \neq 0}\left(\frac{\rho_i}{d_i} \right).
\end{equation}

\begin{definition}
Scaling factor $\rho = (\rho_1, \cdots,\rho_a)$ is called \textit{Homogeneous}  if the vector $\rho$ is proportional to the vector $d = (d_1,\cdots, d_a) $ of degree of $(Y_1, \cdots ,Y_a)$. That is, under homogeneous scaling, 
$\lambda_\F(\rho) = \frac{\rho_i}{d_i}  $ for all $1 \leq  i \leq a$. 
\end{definition}

For all $i = 1, \cdots, a$, denote
\begin{equation}\label{lambda1}
\Lambda_i^{(j)}(\F) := (\Lambda\circ\ad^j(X))(Y_i).
\end{equation}
It is coefficient appearing in (\ref{12}) and set 
\begin{equation} 
|\Lambda(\F)| :=\sup_{(i,j): 1\leq i \leq a, 0 \leq j \leq d_i}\left|\frac{\Lambda_i^{(j)}(\F)}{j!}\right|.
\end{equation} 

Let $\uu(\n)$ be the enveloping algebras of $\n$. The generator $\delta$ is the derivation on $\uu(\n')$ obtained by extending the derivation $\ad(X)$ of $\n'$ to $\uu(\n')$. 
From nilpotency of $\n$ it follows that for any $L \in \uu(\n')$ there exists a first integer $[L]$ such that $\delta^{[L]+1}L = 0$. 

Recall that $\II$ is a codimension 1 ideal of $\n$ used in \S \ref{sec;Kirillov}.

\begin{lemma}\label{Q_i} For each element $L \in \II$ with degree $[L] = i$, there exists $Q_j \in \uu(\n)$  such that $\pi_{*}(L) = \sum_{j=0}^i\frac{1}{j!}\pi'_{*}(Q_j)x^j$ and  $[Q_j] = [L]+1-j$.
\end{lemma}
\proof 
Firstly, we fix elements $X$ and $Y$ as stated in the Lemma \ref{niceop}. For convenience, we normalize the constant of $\pi_*(Y)$ by 1. That is, there exist $X, Y \in \n$ such that 
$$\pi_{*}(X) = \frac{d}{dx}, \quad \pi_{*}(Y) = x.$$ 
Now, we will replace the expansion of $\pi_{*}(L) :=  \sum_{j=0}^i \frac{1}{j!} \Lambda_L^{(j)}(\F)x^j$ by choosing elements $Q_i$ in enveloping algebra $\uu(\n)$.

For the coefficient of top degree, denote $Q_i = \frac{1}{i!}\ad_X^i(L) \in \n$. For degree $i-1$, we set  
$Q_{i-1} = \ad_X^{i-1}(L) - Q_iY \in \uu(\n)$ such that
$$\pi_{*}(Q_{i-1}) = \pi_{*}(\ad_X^{i-1}(L)) - \pi_{*}(Q_i)\pi_{*}(Y) = {\Lambda_L^{(i-1)}(\F)}.$$

Repeating this process up to degree 0, there exist $\exists Q_{l} \in \uu(\n)$ such that  for  $0 < l <i$,
$$\pi_{*}(Q_{l})  = \pi_{*}(\ad_X^{l}(L)) - \frac{1}{l!}\sum_{j=l+1}^{i} \pi_{*}(Q_{j})\pi_{*}(Y)^{j-l}$$
and
$$\pi_{*}(Q_{0}) = \pi_{*}(L) - \frac{1}{l!}\sum_{l=1}^{i} \pi_{*}(Q_{l})\pi_{*}(Y)^l = {\Lambda_L^{(j)}(\F)}.$$
\qed

\begin{definition} If $A$ is self-adjoint on a Hilbert space $H$ and  $\langle Au,u\rangle \geq 0$ for every $u \in H$, then $A$ is called \emph{positive}, denoted by $A \geq 0$.
\end{definition}

\begin{remark}
For two self-adjoint operators $A$ and $B$, $A \geq B$ if and only if $A - B \geq 0$.
Suppose that $A$ and $B$ are bounded operators and commute. Then, $A \geq 0$ and $B \geq 0$ implies that $AB \geq 0$.\footnote{By spectral theorem, there exists a unique, self-adjoint square root 
$A^{1/2} = \int_{\sigma(A)} x^{1/2}  dE_A(x)$. Since $A$ and $B$ commute, $A^{1/2}$ commutes with $B$. Then,
$$\langle AB u, u\rangle  = \langle A^{1/2}B u, A^{1/2}u\rangle = \langle B A^{1/2}u, A^{1/2}u \rangle  \geq 0.$$}
Also, if $0 \leq A \leq B$, then $A^2 \leq B^2$.\footnote{$B^2-A^2 = B(B-A) + (B-A) A $.}
\end{remark}
Recall that for any positive operators $A$ and $B$,
\begin{equation}\label{eqn;post}
AB + BA \leq A^2+B^2 \text{ and } (A+B)^2 \leq 2(A^2 + B^2)
\end{equation}


%

\begin{lemma}\label{lem;laplace}
 For any $r\geq1$ and $a \geq 1$, there exists constant $C(a,r)>0$ such that 
\begin{equation}\label{eqn;laplace}
\norm{\pi_{*}(\Delta(t)^{2r})u} \leq C(a,r)\norm{\sum_{i=1}^a\pi_{*}(Y_i(t)^{4r})u}.
\end{equation}
\end{lemma}
\begin{proof}
It suffices to prove that there exists $C = C(a,r)>0$ such that 
\begin{equation}\label{eqn;laplace22}
\Delta(t)^{2r} \leq C\left(\sum_{i=1}^aY_i(t)^{4r}\right)
\end{equation}
since this implies $\Delta(t)^{4r} \leq C^2\left(\sum_{i=1}^aY_i(t)^{4r}\right)^2$ by the remark. 

We prove \eqref{eqn;laplace22} by induction. If $r =1$, then
$$\Delta(t)^{2} = \left(\sum_{i=1}^aY_i(t)^2\right)^2 = \left(\sum_{i=1}^aY_i(t)^4 + \sum_{i \neq j}^aY_i(t)^2Y_j(t)^2\right).$$
By \eqref{eqn;post}, for each $i$ and $j$,
$$Y_i(t)^2Y_j(t)^2 + Y_j(t)^2Y_i(t)^2 \leq Y_i(t)^4+Y_j(t)^4.$$ 

Then, there exists $C_0 = (a+1)$ such that  
\begin{equation}\label{eqn;deltac0}
\Delta(t)^{2} \leq  C_0\left(\sum_{i=1}^aY_i(t)^4\right).
\end{equation}

Assume that the statement holds for large $r$. Then, since $\Delta(t)^{2}$ and $\sum_{i=1}^aY_i(t)^4$ are positive, by \eqref{eqn;deltac0},
there exists $C_1(a,r)$ such that
\begin{align*}
\Delta(t)^{2(r+1)} & \leq C_1(a,r)\left(\sum_{i=1}^aY_i(t)^{4r}\right)(\Delta(t)^{2})\\
& \leq C_1(a,r)(a+1)\left(\sum_{i=1}^aY_i(t)^{4r}\right)\left(\sum_{i=1}^aY_i(t)^4\right).
\end{align*}
Note that the following inequality holds: for any $r \geq 1$,
\begin{equation}\label{eqn;y4}
Y_i(t)^{4r}Y_j(t)^4 + Y_i(t)^4Y_j(t)^{4r} \leq Y_i(t)^{4(r+1)}+Y_j(t)^{4(r+1)}.
\end{equation}
This inequality is proved by showing that
\begin{align*}
& \left(Y_i(t)^4 - Y_j(t)^4\right) \left(Y_i(t)^{4r} - Y_j(t)^{4r} \right) \\
& = \left(Y_i(t)^4 - Y_j(t)^4\right)^2 \left(\sum_{l=0}^{r-1}(Y_i(t)^4)^{l}(Y_j(t)^4)^{r-1-l} \right) \geq 0.
\end{align*}
Since $Y_i(t)^4, Y_j(t)^4$ are all positive, the last inequality holds.

Then, by \eqref{eqn;y4},
\begin{align*}
&\left(\sum_{i=1}^aY_i(t)^{4r}\right)\left(\sum_{i=1}^aY_i(t)^4\right) \\
& = \left(\sum_{i=1}^aY_i(t)^{4r+4}+ \sum_{i < j}^a(Y_i(t)^{4r}Y_j(t)^{4}+Y_j(t)^{4r}Y_i(t)^{4})\right) \\
& \leq (a+1)\left(\sum_{i=1}^aY_i(t)^{4r+4}\right).
\end{align*}

Setting $C_2(a,r) = C_1(a,r)(a+1)^2$, 
$$\Delta(t)^{4(r+1)} \leq C_2(a,r)\left(\sum_{i=1}^aY_i(t)^{4(r+1)}\right).$$
Therefore, induction holds and we finish the proof.
\end{proof}


%

For cohomological equation $X(t)u = f$, denote its Green's operator $G_{X(t)}$. The following theorem states an estimate for rescaled version of Theorem \ref{green1}.
\begin{theorem}\label{reestimate}
For $r>1$, let $s > 2r(k+1)+1/2$. For any $f \in \K^s(M) $, there exists $C_{r,k,s} >0$ such that the following holds: for any $t > 0$,
$$ |{G_{X(t)}f}|_{r,\F(t)} \leq C_{r,k,s} e^{-(1-\lambda_\F)t} \max\{1, {\delta_\OO^{-1}}\} |{f}|_{s,\F(t)}. $$ 
\end{theorem}

\begin{proof}

Firstly, we estimate the bound of Green's operator with Sobolev norm. 

By Lemma \ref{niceop}, there exists a rescaled operator $Y(t) \in \n_{k-1}$ with 
$$\pi_{*}(Y(t)) = 2\pi \iota \delta_\OO(t) x \Id_{H'}$$ 
where  $\delta_\OO(t) =\delta_\OO e^{-\rho_Y t}>0$.

By Cauchy-Schwarz inequality, for any $l \in \N$,
\begin{align*}
\begin{split}
&\norm{\pi_{*}(Y(t))^lG_{X(t)}f}^2 \\ 
&\leq \int_{0}^{\infty}\left(|2\pi {\delta_\OO(t)}x|^l \int_{x}^\infty e^{-t}\norm{f(s)}_{H'}ds \right)^2dx   + \int_{-\infty}^{0}\left(|2\pi {\delta_\OO(t)}x|^l \int_{-\infty}^xe^{-t}\norm{f(s)}_{H'}ds \right)^2dx\\
 & \leq \int_{0}^{\infty}|2\pi {\delta_\OO(t)}x|^{2l} \int_{x}^\infty e^{-2t}\norm{f(s)}^2_{H'}ds dx   + \int_{-\infty}^{0}|2\pi {\delta_\OO(t)}x|^{2l} \int_{-\infty}^x e^{-2t}\norm{f(s)}^2_{H'}dsdx.
 \end{split}
\end{align*}

If $\alpha>1$, then by H\"older's inequality, 
\begin{equation}\label{eqn;hold}
\int_{-\infty}^x \norm{f(s)}^2_{H'}ds \leq \left(\int_{-\infty}^x{(1+ 4\pi^2\delta_O(t)^2s^2)^{-\alpha}}ds\right)
\norm{(I-Y(t)^2)^\frac{\alpha}{2}f }^2.
\end{equation}

For all $\alpha >1$, we set 
\begin{align}
\begin{split}
C_{\alpha,l}^2 & = \int_{0}^{\infty} (2\pi x)^{2l}\left(\int_{x}^{\infty}(1+(4\pi^2 s^2))^{-(l+\alpha)}ds\right)dx \\
& + \int_{-\infty}^{0} (2\pi x)^{2l}\left(\int_{-\infty}^{x}(1+(4\pi^2 s^2))^{-(l+\alpha)}ds\right)dx < \infty. 
\end{split}
\end{align}

By H\"older's inequality and change of variables for $  x' = {\delta_\OO(t)}x$ and $s' = {\delta_\OO(t)}s$
\begin{align}\label{holder}
\begin{split}
\norm{\pi_{*}(Y(t))^lG_{X(t)}f} &\leq C_{\alpha,l} \left(\frac{e^{-t}}{\delta_\OO(t)}\right)\norm{\pi_{*}((I-Y(t)^2)^\frac{l+\alpha}{2}f }\\
& = C_{\alpha,l}e^{-(1-\rho_Y)t}\left(\frac{1}{{\delta_\OO}}\right)\norm{\pi_{*}((I-Y(t)^2)^\frac{l+\alpha}{2}f }.
\end{split}
\end{align}

Let $E_x:C^\infty(H_\pi) \rightarrow C^\infty(H')$ be linear operator defined by $E_xf = f(x)$. (Cf. Lemma \ref{lem;schwartz}.)
Then the action of $\pi_{*}(Y(t))$ on $C^\infty(H_\pi)$ can be rewritten as
\begin{equation}
E_x\pi_{*}(Y(t)) = (2\pi \iota x \delta_\OO(t))^jE_x, \quad \text{ for all } j \in \N \text{ and } x \in \R.
\end{equation}

Let $L(t)$ be a rescaled element of $L \in \F$. By definition of representation \eqref{12} and by Lemma \ref{Q_i} for rescaled basis $\F(t)$, for any $L(t)$ there exists $Q_j(t) \in \uu(\n)$ such that
\begin{equation}\label{eqn;p(t)}
E_x\pi_{*}(L(t))  = \sum_{j=0}^{[L]} \frac{(2\pi \iota x \delta_\OO(t) )^j}{j!} \pi'_{*}(Q_{j}(t))E_x. 
\end{equation}

For Green's operator $G_{X(t)}$, since $E_xG_{X(t)} = \int_{-\infty}^{x}E_sds$ for all $x \in \R$,
\begin{align}\label{eqn;PL}
\begin{split}
E_x\pi_{*}(L(t))G_{X(t)} &=  \sum_{j=0}^{[L]} \frac{(2\pi \iota x \delta_\OO(t))^j}{j!} \pi'_{*}(Q_{j}(t))E_xG_{X(t)}\\
& = \sum_{j=0}^{[L]} \frac{(2\pi \iota x \delta_\OO(t))^j}{j!} \int_{-\infty}^{x}\pi'_{*}(Q_{j}(t))E_sds\\
& = \sum_{j=0}^{[L]} \frac{(2\pi \iota x \delta_\OO(t))^j}{j!} \int_{-\infty}^{x}E_s\pi_{*}(Q_{j}(t))ds\\
&= \sum_{j=0}^{[L]} \frac{(2\pi \iota x \delta_\OO(t))^j}{j!} E_xG_{X(t)}(\pi_{*}(Q_{j}(t))\\
& = E_x\sum_{j=0}^{[L]} \frac{1}{j!} \pi_{*}(Y(t))^jG_{X(t)}(\pi_{*}(Q_{j}(t)).
\end{split}
\end{align}

Combining \eqref{holder} and  \eqref{eqn;PL}, there exists a constant $C'_{\alpha,j}>0$ such that
\begin{align*}
 \norm{\pi_{*}(L(t))G_{X(t)}(f)}   & \leq \sum_{j=0}^{[L]} \frac{1}{j!}\norm{\pi_{*}(Y(t))^jG_{X(t)}(\pi_{*}(Q_{j}(t))f)}\\
& \leq \sum_{j=0}^{[L]} C'_{\alpha,j}e^{-(1-\rho_Y)t} \left(\frac{1}{{\delta_\OO}}\right)^{j+1}\norm{\pi_{*}((I-Y(t)^2)^{\frac{j+\alpha}{2}}\pi_{*}((Q_{j}(t)f)}.
\end{align*}

Note that by binomial formula, there exists $R_{j}(t) \in \uu(\n)$ such that 
\begin{equation}\label{eqn;laplace2}
\pi_*(L(t))^{2r} = \left(\sum_{j=0}^{[L]} \frac{1}{j!}\pi_{*}(Y(t)^j)\pi_{*}(Q_{j}(t))\right)^{2r} := \sum_{j=0}^{2r[L]}\pi_{*}(Y(t)^j)\pi_{*}(R_{j}(t)).
\end{equation}

Especially, $R_j(t)$ is product of $Q_i(t)'s$ and $[R_{j}(t)] = 2r([L]+1)-j$.  Specifically, to compute transverse Laplacian, here we assume that $L(t) = (Y_i(t))^{4r}$ for each element $Y_i \in \F$. Then, by \eqref{eqn;laplace2}, for any $\alpha >1$
\begin{align}\label{eqn;P(t)}
\begin{split}
 &\norm{\pi_{*}(Y_i(t)^{4r})G_{X(t)}(f)} \\
&  \leq \sum_{j=0}^{4r[Y_i]}  C({\alpha,i,j,r})e^{-(1-\rho_Y) t} \left(\frac{1}{ {\delta_\OO}}\right)^{j+1}\norm{\pi_{*}((I-Y(t)^2))^{\frac{j+\alpha}{2}}\pi_{*}(R_{j}(t))f }\\
 & \leq C(\alpha,i,r)e^{-(1-\rho_Y) t} \max\{1, {{\delta_\OO}}^{-4r[Y_i]+1}\}|f|_{\alpha + 4r([Y_i] + 1),\F(t)}.
\end{split}
\end{align}

Therefore, combining with Lemma \ref{lem;laplace} 
\begin{flalign*} 
\norm{\pi_*(\Delta(t)^{2r})G_{X(t)}(f)} & \leq C(a,r)\sum_{i=1}^{a} \norm{\pi_*(Y_i(t)^{4r})G_{X(t)}(f)} \\
& \leq C(a,r,\alpha)\sum_{i=1}^{a} e^{-(1-\rho_Y) t} \max\{1, {{\delta_\OO}}^{-4r([Y_i]+1)}\}|f|_{\alpha + 4r([Y_i] + 1),\F(t)} \\
 &\leq C(a,r,\alpha)e^{-(1-\rho_Y) t} \max\{1, {{\delta_\OO}}^{-4r(k+1)}\}|f|_{\alpha + 4r(k + 1),\F(t)}. 
\end{flalign*}

Since $[\Delta^r] \leq 2kr$, there exists $C' = C'(a,\alpha,r,X) > 0$ such that 
\[ |{G_{X(t)}f}|_{2r,\F(t)} \leq C' e^{-(1-\rho_Y)t} \max\{1,{\delta^{-1}_\OO}\} |{f}|_{\alpha+4(k+1)r,\F(t)}. \]

By interpolation, for all $s > 2r(k+1)+1/2$, there exists a constant $C_{r,s}: = C_{r,s}(k,X)>0$ such that
\[ |{G_{X(t)}f}|_{r,\F(t)} \leq C_{r,s} e^{-(1-\rho_Y)t} \max\{1,{\delta^{-1}_\OO} \} |{f}|_{s,\F(t)}. \] 
By the choice of $Y$, we obtain $\rho_Y = \lambda_\F$, which finishes the proof.
\end{proof}

\subsection{Scaling of invariant distribution}

In this section, we introduce the Lyapunov norm and compare bounds between Sobolev dual norm and Sobolev Lyapunov norm of invariant distribution in every irreducible, unitary representation. 


For all $t \in \R$ and $\lambda:= \lambda_\F(\rho)$ defined in (\ref{67}), let the operator $U_t: L^2(\R,H') \rightarrow L^2(\R,H')$ be the unitary operator defined as follows:
\begin{equation}
(U_tf)(x) = e^{-\frac{\lambda}{2}t}f(e^{-\lambda t}x).
\end{equation}

We will compare the norm estimate of invariant distributions with respect to scaled basis
$$|D|_{-r,\F(t)} = \sup_{f \in W^r(H_{\pi})}\{|D(f)|:\norm{f}_{r,\F(t)} =1\}$$
by unscaled norm $|D|_{-r,\F}$.

\begin{theorem} For $r \geq 1$ and $s > r(k+1)$, there exists a constant $C_{r,s}>0$ such that for all $t >0$, the following bound holds:
$${\norm{U_tf}_{r, \F(t)}} \leq C_{r,s}{\norm{f}_{s, \F}}.$$
\end{theorem}
\proof

Assume the same hypothesis for $L \in \F$ and $L(t)$ in the proof of Theorem \ref{reestimate}. By Lemma \ref{Q_i}, there exists $(i-j+1)$th order $Q_j \in \uu(\n)$ with  
\begin{align*}
{U_t^{-1}L(t)U_t } & =  x^iQ_i + e^{-\lambda t}x^{i-1}Q_{i-1} + e^{-2\lambda t}x^{i-2}Q_{i-2}+\cdots + e^{-i \lambda t}Q_o.
\end{align*}

Then there exists $C>0$ such that 
\begin{align*} \norm{U_t^{-1}L(t)U_t f}  
& \leq \sum_{j=0}^i \norm{e^{-(i-j)\lambda t}x^jQ_jf} \\
& \leq C \max_{0 \leq j \leq i} \norm{Y^jQ_jf}\\
& \leq C|f|_{[L]+1,\F}.
\end{align*}
Since $[L] \leq k$, by unitarity
\begin{equation}
|{U_tf}|_{1, \F(t)} \leq C_1 |{f}|_{k+1,\F}.
\end{equation}
Hence, for any $s > r(k+1)$,
\begin{equation}\label{ineq;unitarity}
|{U_tf}|_{r, \F(t)} \leq C_{r,s} |{f}|_{s,\F}.
\end{equation}
\qed

\begin{theorem}\label{DD} For $r \geq 1$ and $s > r(k+1)$, there exists a constant $C_{r,s}>0$ such that for any $\lambda >0$ and $t>0$, the invariant distribution defined in \eqref{eqn;invdistri} satisfies
$$|D|_{-s, \F} \leq C_{r,s}e^{-\frac{\lambda }{2}t}|D|_{-r, \F(t)}.$$
\end{theorem}
\proof 
Recall the functional $\ell$ defined in the Lemma \ref{lem;ell}.
For $f \in C^\infty(H_\pi)$, 
\begin{align*}
\begin{split}
D(U_tf) & = \int_{\R} \ell (e^{-\frac{\lambda}{2}t}f(e^{-\lambda t}x))dx \\
&= \int_{\R} e^{-\frac{\lambda}{2}t}\ell(f(e^{-\lambda t}x))dx \\
& = e^{\frac{\lambda}{2}t} \int_{\R} \ell(f(y))dy \\
& = e^{\frac{\lambda}{2}t}D(f).
\end{split}
\end{align*}

Then by unitarity \eqref{ineq;unitarity},
$$|D|_{-r, \F(t)} = \sup_{f\neq 0} \frac{|D(f)|}{|{f}|_{r, \F(t)}} = \sup_{f\neq 0} \frac{|D(U_tf)|}{|{U_tf}|_{r, \F(t)}} \geq \sup_{f\neq 0} \frac{e^{\frac{\lambda}{2}t}|D(f)|}{C_{r,s}|{f}|_{s, \F}} = C^{-1}_{r,s}e^{\frac{\lambda}{2}t}|D|_{-s, \F},$$
and
$$|D|_{-s, \F} \leq C_{r,s}e^{-\frac{\lambda}{2}t}|D|_{-r, \F(t)}.$$
\qed

\begin{definition}[Lyapunov norm]
For any basis $\F$ and all $\sigma > 1/2$, define {Lyapunov norm} 
\begin{equation} 
\lnorm{D} : = \inf_{\tau \geq 0} e^{-\frac{\lambda_\F(\rho)}{2}\tau}|D|_{-\sigma,\F(\tau)}.
\end{equation}
\end{definition}
The following lemma is immediately from the definition of the norm.
\begin{lemma}\label{lyapp} For all $t \geq 0$, we have
$$\norm{D}_{-\sigma,\F} \leq e^{-\frac{\lambda_\F(\rho)}{2}t}\norm{D}_{-\sigma,\F(t)}.$$
\end{lemma}
\proof By definition of the norm, 
\begin{align*}
\norm{D}_{-\sigma,\F} & = \inf_{\tau \geq 0} e^{-\frac{\lambda_\F(\rho)}{2}\tau}|D|_{-\sigma, \F(\tau)}\\
& = e^{-\frac{\lambda_\F(\rho)}{2}t}\inf_{\tau+t \geq 0} e^{-\frac{\lambda_\F(\rho)}{2}\tau}|D|_{-\sigma, \F(t+\tau)} \leq e^{-\frac{\lambda_\F(\rho)}{2}t}\norm{D}_{-\sigma,\F(t)}.
\end{align*}
\qed

We conclude this section by introducing useful inequality that follows from the Theorem \ref{DD}, 
\begin{equation}\label{lya}
C_{r, s}^{-1}{|D|_{-s,\F}} \leq \norm{D}_{-r,\F} \leq |D|_{-r,\F}. 
 \end{equation}

\section {A Sobolev trace theorem}\label{sec;4}
In this section, the notion of the average width \eqref{width} which is an average measure of close returns along an orbit is introduced and we prove a Sobolev trace theorem for nilpotent orbits. According to this theorem, uniform norm of an ergodic integral is bounded in terms of the average width of the orbit segment times the transverse Sobolev norms of the function, with respect to a given basis of the Lie algebra. 

\subsection{Sobolev a priori bounds} Assume $\F(t)  = (X(t),Y(t))$ is rescaled basis. For any $x \in M$, let $\phi_{x,t} : \R \times \R^a \rightarrow M$ be the local embedding defined by 
$$\phi_{x,t}(\tau,\bold s) = x \exp(\tau X(t))\prod_{i=1}^{a}\exp(s_{i}{Y_{i}(t)}), \ \textbf{s} = (s_i)_{i=1}^a$$

\begin{lemma}\label{4.1} For any $x \in M$, $t \geq 0$, and $f \in C^{\infty}(M)$, we have 
\begin{align*}
&\partial_{s_i}f\circ \phi_{x,t}(\tau,\textbf{s}) = S_i f\circ\phi_{x,t}(\tau,\textbf{s}), \quad S_i = Y_{i}(t)+\sum_{l>i}^aq_l(s,t)Y_{l}(t) \in \n
 \end{align*}
where $q$ is polynomial in s of degree at most $k-1$ and $|q_l(s,t)| \leq |q_l(s,0)|$ for all $t \geq 0$.  
 \end{lemma}
\begin{proof}
Let $\textbf{s}+h_i$ denote sequence with $(\textbf{s}+h)_i = s_i +h$ and $(\textbf{s}+h_i)_j = s_j$ if $i \neq j$. 
By definitions,
$$\partial_{s_i}f\circ \phi_{x,t}(\tau,\textbf{s}) = \lim_{h\rightarrow 0}\frac{f\circ \phi_{x,t}(\tau,\textbf{s}+h_i)- f\circ \phi_{x,t}(\tau,\textbf{s})}{h},$$
and we plan to rewrite $f\circ \phi_{x,t}(\tau,\textbf{s}+h_i)$ in suitable way to differentiate.\\

For fixed $i$ and $j>i$,
\begin{align*}
& \exp((s_{i}+h){Y_{i}(t)})\exp(s_{j}{Y_{j}(t)})  \\
& = \exp(s_{i}Y_{i}(t))\exp(hY_{i}(t))\exp(s_{j}{Y_{j}(t)})\\
& = \exp(s_{i}Y_{i}(t))\exp(e^{ad(hY_{i}(t)) }s_{j}{Y_{j}(t)})\exp(hY_{i}(t))\\
& = \exp(s_{i}Y_{i}(t))\exp(s_{j}Y_{j}(t))\exp(\sum_{n=1}^\infty{\frac{1}{n!}ad^n_{hY_{i}(t)} }s_{j}{Y_{j}(t)})\exp(hY_{i}(t)).
\end{align*}
By Campbell-Hausdorff formula, we set 
$$= \exp(s_{i}Y_{j}(t))\exp(s_{i}Y_{j}(t))\exp(h(Y_{i}(t)+[Y_i(t),s_{j}Y_{j}(t)]) + O(h^2)).$$

Choose $j = i+1$ and observe that all the terms of $h$ are on right side.  Iteratively, we will repeat this process from $j = i+1$ to $a$ until all the terms of $h$ pushed back. That is, we conclude
\begin{align*}
\phi_{x,t}(\tau,\bold s+h_i) &= \phi_{x,t}(\tau,\bold s)\exp(h(Y_{i}(t)+[Y_i(t),s_{i+1}Y_{i+1}(t)] \\
&+[[Y_i(t),s_{i+1}Y_{i+1}(t)],s_{i+2}Y_{i+2}(t)]  + \cdots + [Y_i(t),\cdots] , s_{a}Y_{a}(t)]\cdots ] )\\
& + O(h^2)).
\end{align*}
For convenience,  we write coefficient function $q_l(s,t)$ in polynomial degree at most $k$ for $s$ such that
$$\phi_{x,t}(\tau,\bold s+h_i) = \phi_{x,t}(\tau,\bold s)\exp(h(Y_{i}(t)+\sum_{l>i}^aq_l(s,t)Y_{l}(t)) + o(h^2)).$$

We conclude the proof by choosing $S_i = Y_{i}(t)+\sum_{l>i}^aq_l(s,t)Y_{l}(t)$.
Also, commutation in rescaled elements $[Y_i(t),s_{j}Y_{j}(t)] = s_{j}e^{-\rho t}Y_{k}(t)$ implies that the term $q_l(s,t)$ includes exponential terms with negative exponent so that it decreases for $t \geq 0$.
\end{proof}

Let $\triangle_{\R^a}$ be the Laplacian operator on $\R^a$ given by 
$$ \triangle_{\R^a} = - \sum_{i=1}^a\frac{\partial^2}{\partial s_i^2}.$$

 Given an open set $O \subset \R^a$ containing origin, let $\mathcal{R}_O$ be the family of all $a$-dimensional symmetric rectangles $R \subset [-\frac{1}{2}, \frac{1}{2}]^a \cap O$ that are centered at origin. The \emph{inner width} of the set $O \subset \R^a$ is the positive number 
$$w(O) = \sup \{\text{Leb}(R) \mid R \in \mathcal{R}_O \},$$
where Leb is Lebesgue measure on $R$. The width function of a set  $\Omega \subset \R \times \R^a$ containing the line $\R \times \{0\}$ is the function $w_\Omega: \R \rightarrow [0,1]$ defined as follows:
$$w_\Omega(\tau):= w(\{\bold s\in \R^a \mid (\tau,\bold s) \in \Omega\}), \ \forall \tau \in \R.$$

\begin{definition}\label{def;averagewidth}
Consider the family $\mathcal{O}_{x,t,T}$ of open sets $\Omega \subset \R \times \R^a$ satisfying two conditions
$$[0,T]\times \{0\} \subset \Omega \subset \R \times [-\frac{1}{2}, \frac{1}{2}]^a $$
and $\phi_{x,t}$
 is injective on the open set $\Omega \subset \R^a$. The \emph{average width} of the orbit segment of rescaled nilflow $\{\phi_{x,t}(\tau,0) \mid 0 \leq t \leq T\}$,
\begin{equation}\label{width}
w_{\F(t)}(x, T) := \sup_{\Omega \in \OO_{x,t,T} }\left(\frac{1}{T}\int_{0}^{T}\frac{ds}{w_\Omega(s)}\right)^{-1}.
\end{equation}
is positive number. 
\end{definition}

The following lemma is derived from standard Sobolev embedding theorem under rescaling argument.

\begin{lemma}\cite[Lemma 3.7]{FF14}\label{pre}
Let $I \subset \R$ be an interval, and let $\Omega \subset \R \times \R^a$ be a Borel set containing the segment $I \times \{0\} \subset \R \times \R^a$. For every $\sigma >a/2$, there is a constant $C_s >0$ such that for all functions $F \in C^{\infty}(\Omega)$ and all $\tau \in I$, we have

$$\left(\int_I |F(\tau,0)|d\tau \right)^2 \leq C_\sigma \left(\int_I \frac{d\tau}{w_\Omega(\tau)} \right) \int_{\Omega}|(I-\triangle_{\R^a})^\frac{\sigma}{2}F(\tau,\bold s)|d\tau d\bold s. $$   
\end{lemma}

The following theorem indicates the bound of ergodic average of scaled nilflow $\phi_{X(t)}^\tau$ with width function on general nilmanifolds. (See also \cite[Theorem 5.2]{FFT16} for twisted horocycle flows.) 
\begin{theorem}\label{sobolev} For all $\sigma >  a/2 $, there is a constant $C_{\sigma}>0$ such that the following holds. 
$$\left|\frac{1}{T} \int_{0}^{T}f\circ \phi_{X(t)}^\tau(x)d\tau\right| \leq {C_\sigma}{T^{-\frac{1}{2}}w_{\F(t)}(x,T)^{-\frac{1}{2}}}|f|_{\sigma,\F(t)}$$
\end{theorem}
\proof
Recall that for any self-adjoint operators $A$ and $B$,
$$(A+B)^2 \leq 2(A^2+B^2).$$


Since $|s_i| \leq \frac{1}{2}$ and $t \geq 1$, by Lemma \ref{4.1}, each $q_{j}$ is bounded in $\s$ and $t$. Then, by essentially skew-adjointness of $Y_i(t)$,
there exists a large constant $C>1$ with
\begin{align*}
-S_i^2 &= -(Y_i(t) +\sum_{l>i}^aq_l(s,t)Y_{l}(t))^2\\
& \leq -C\sum_{j=i}^a Y_j({t})^2. 
\end{align*}

Since operators on both sides are essentially self-adjoint,
$$(I-\sum_{i=1}^a S_i^2)^\frac{\sigma}{2} \leq C^{\frac{\sigma}{2}}(I - \sum_{i=1}^{a}Y_i(t)^2 )^\frac{\sigma}{2}.$$

Thus, there is a constant $C_\sigma>0$ such that 
\begin{equation}
\norm{(I-\triangle_{\R^a})^\frac{\sigma}{2}f\circ\phi_{x,t}}_{L^2(\Omega)}^2 \leq C_\sigma \norm{(I-\Delta_{{\F(t)}})^\frac{\sigma}{2}f}_{L^2(M)}^2. 
\end{equation}

By Lemma \ref{pre}, we can see that for $\sigma >a/2$, setting $F(\tau,0) = f\circ \phi_{X(t)}^\tau(x)$
\begin{align*}
\begin{split}
\left|\frac{1}{T} \int_{0}^{T}f\circ \phi_{X(t)}^\tau(x)dt \right|^2 & = \left(\frac{1}{T} \int_{0}^{T}|F(\tau,0)d\tau|^2 \right)^2\\
& \leq C_\sigma \frac{1}{T}\left(\frac{1}{T}\int_0^T \frac{ds}{w_\Omega(s)} \right) \int_{\Omega}|(I-\triangle_{\R^a})^\frac{\sigma}{2}F(\tau,\bold s)|d\tau d\bold s\\
& \leq C_\sigma T^{-1}w_{\F(t)}(x,T)^{-1}\norm{(I-\Delta_{\F(t)})^\frac{\sigma}{2}f}_{L^2(M)}^2. 
\end{split}
\end{align*}
\qed

\section{Average width estimate}\label{sec;AWE}

This section is devoted to the proof of estimates on the averaged width of orbits of nilflows. 
 Compared with Quasi-abelian case (See \cite[Lemma 2.4]{FF14}), there is no explicit expressions for return map on transverse higher step nilmanifolds. Instead, we calculate differential of displacement and estimate the measure of close return orbits with respect to the rescaled vector fields.

\emph{Strategy.}
\begin{enumerate}
\item  In section \ref{sec;5.1}, we introduce basic settings. Since the flow commutes with centralizer, we take quotient map to obtain local diffeomorphism. It is remarkable to see that we set tubular neighborhood to consider close return orbits on the quotient space. This contributes calculating the measure of close return so called almost periodic set. (See \eqref{APset} and Lemma \ref{attract}.)

\item The range of the differential of displacement map coincides with the range of adjoint map $ad_{X_\alpha}$. This is one main reason that it requires a necessity of transverse condition. Without this condition, there could be a direction that return orbit that does not reach on the transverse manifold, which fails the idea that the set of almost periodic point should have a small measure up to rescaling vector $\rho$. (See Lemma \ref{5.4}.)

\item In section \ref{width bounds}, we prove a bound of average width by ergodic averages of cut-off functions. 
By definition \ref{def;cases}, we classify the type of close return orbits by growth of local coordinates. The width of function does not vanish on such set and it is injective under the restricted domain. (See Lemma \ref{lem;85} and \ref{HH}.)

\item
Finally, in section \ref{section;5.3} and \ref{Width estimates}, we follow the known estimate from \cite{FF14}, which are necessary for proving bounds of ergodic averages in the section \ref{sec;sec6}. In particular, Definition \ref{good} of good point means the set of points whose the width along transverse direction to the flow cannot be too small and we prove the complement of the set of good points has a small measure. (See Lemma \ref{L_i}.)
\end{enumerate}

\subsection{Almost periodic points}\label{sec;5.1}
 Let $X_{\alpha}$ be the vector field on $M$ defined in (\ref{flow1}). Recall the formula (\ref{floww})
$$X_{\alpha}:= \xi+\sum_{(i,j) \in J}\alpha^{(j)}_i\eta^{(j)}_i.$$

Let us introduce special type of condition for the Lie algebra $\n$ required for width estimate. 
\begin{definition}\label{1} The nilpotent Lie algebra  $\n$ satisfies \emph{transversality condition} if there exists a basis $(X_\alpha, Y)$ of $\n$ such that
\begin{equation}\label{2}
 \langle \mathfrak{G}_\alpha \rangle + \ran(\ad_{X_{\alpha}})+C_\II(X_{\alpha}) = \n
\end{equation} 
where $\mathfrak{G}_\alpha = (X_\alpha, Y^{(1)}_i)_{1 \leq i \leq n}$ is a set of generator, $\ran(\ad_{X_{\alpha}})=\{Y \in \II \mid Y= \ad_{X_{\alpha}}(W), \ W \in \II \}$ and $C_\II(X_{\alpha}) = \{Y \in \II \mid [Y,X_{\alpha}] = 0\}$ is centralizer.
\end{definition}
It is clear that the set of generators are neither included in the range of $\ad_{X_{\alpha}}$, nor in the centralizer $C_\II(X_{\alpha})$.
We will restrict $\n$ satisfying the condition \eqref{2} in the rest of sections.
\begin{remark}
The transversality condition implies that displacement (or distance between $x$ and $\Phi^r_{\alpha,\theta}(x)$), induced by return map $\Phi_{\alpha,\theta}$, should intersect the set of centralizer transversally. I.e the measure of the set of close return orbit in transverse manifold  $M_\theta^a$ should not be invariant under the action of flow.  This condition is crucial in estimating the almost periodic orbit \eqref{APset} under rescaling of basis in the Lemma \ref{5.4}.
\end{remark}

Recall that $M_\theta^a$ denotes the fiber at $\theta \in \T^1$ of the fibration $pr_2: M \rightarrow \T^1$. $\Phi_{\alpha,\theta}$ denote the first return map of nilflow $\{\phi_{X_{\alpha}}^t\}$ to the transverse section $M_\theta^a$ and $\Phi^r_{\alpha,\theta}$ denote $r$-th iterate of the map $\Phi_{\alpha,\theta}$. Let $G$ denote nilpotent Lie group with its lattice $\Gamma$ defining $M_\theta^a = \Gamma \backslash G$. $G$ acts on $M_\theta^a$ by right action and action of $G$ extends on $M_\theta^a \times M_\theta^a$.

Define a map  $\psi_{\alpha,\theta}^{(r)} : M_\theta^a  \rightarrow M_\theta^a \times M_\theta^a$ given by $\psi_{\alpha,\theta}^{(r)}(x) = (x,\Phi^r_{\alpha,\theta}(x))$. By its definition, the map $\Phi^r_{\alpha,\theta}$ commutes with the action of the centralizer $C_{G} = \exp(C_\II(X_{\alpha})) \subset G$ and its action on product $M_\theta^a \times M_\theta^a$ commutes with $\psi_{\alpha,\theta}^{(r)}$. That is, for $c \in C_{G}$ and $x = \Gamma g$,
\begin{equation}\label{commute}
\psi_{\alpha,\theta}^{(r)}(xc) = (xc,\Phi^r_{\alpha,\theta}(xc)) = (xc, \Phi^r_{\alpha,\theta}(x)c) =  \psi_{\alpha,\theta}^{(r)}(x)c.
\end{equation}

Then quotient map is well-defined on 
\begin{equation}\label{psi}
\Psi_{\alpha,\theta}^{(r)} := M_\theta^a/C_{G} \longrightarrow M_\theta^a \times M_\theta^a/C_{G}.
\end{equation}

\textbf{Setting.} 
$(i)$ In $M_\theta^a \times M_\theta^a$, we set diagonal $\Delta = \{(x,x) \mid x \in M_\theta^a\}$ which is isomorphic to $M_\theta^a $ by identifying $(x,x)$  with $x \in M_\theta^a$.
Given $(x,x) \in \Delta$, tangent space of diagonal is $ T_{(x,x)}\Delta :=  \{(v,v) \mid v \in T_{x} M_\theta^a\}$ and its normal space is defined as $(T_{(x,x)}\Delta)^\perp = \{(v,-v) \mid v \in T_x M_\theta^a \} = T_{(x,x)}\Delta^\perp$.
 On tangent space at $(x,x) \in M_\theta^a\times M_\theta^a$, it splits by
$$T_{(x,x)}(M_\theta^a \times M_\theta^a) = T_{(x,x)}\Delta \oplus (T_{(x,x)}\Delta)^\perp. $$
For any $w_1, w_2 \in T_{x}M_\theta^a$, 
\begin{equation}\label{iden}
(w_1,w_2) = 1/2(w_1+w_2,w_1+w_2) + 1/2(w_1-w_2, -(w_1-w_2)).
\end{equation}

$(ii)$ Given $x = \Gamma h_1, y = \Gamma h_2 \in M_\theta^a$, define a set $\Delta_{(x,y)} = \{(xg,yg)  \mid g \in G \} \subset M_\theta^a \times M_\theta^a$ for $(xg,yg) = (\Gamma h_1 g, \Gamma h_2 g)$ and $\Delta_{(x,y)}^\perp = \{(xg,yg^{-1}) \mid g \in G \}$ that contains $(x,y)$. For $\psi_{\alpha,\theta}^{(r)}(x) = (x,\Phi^r_{\alpha,\theta}(x))$, its tangent space in $M_\theta^a \times M_\theta^a$ is decomposed 
$$T_{(x,\Phi^r_{\alpha,\theta}(x))}(M_\theta^a \times M_\theta^a) = T_{(x,\Phi^r_{\alpha,\theta}(x))}\Delta_{(x,\Phi^r_{\alpha,\theta}(x))} \oplus (T_{(x,\Phi^r_{\alpha,\theta}(x))}\Delta_{(x,\Phi^r_{\alpha,\theta}(x))})^\perp. $$
Then tangent space of diagonal is $T_{(x,\Phi^r_{\alpha,\theta}(x))}\Delta_{(x,\Phi^r_{\alpha,\theta}(x))} = \{(v,d_x\Phi^r_{\alpha,\theta}(v) ) \mid v \in T_x M_\theta^a\}$ and its normal space is identified as 
$$(T_{(x,\Phi^r_{\alpha,\theta}(x))}\Delta_{(x,\Phi^r_{\alpha,\theta}(x))})^\perp = T_{(x,\Phi^r_{\alpha,\theta}(x))}\Delta_{(x,\Phi^r_{\alpha,\theta}(x))}^\perp.$$
By identification in \eqref{iden}, for $w_1= v$ and $w_2 = -d_x\Phi^r_{\alpha,\theta}(v)$, we write 
\begin{multline}\label{ident}
(T_{(x,\Phi^r_{\alpha,\theta}(x))}\Delta_{(x,\Phi^r_{\alpha,\theta}(x))})^\perp = \\
\{(1/2(v-d_x\Phi^r_{\alpha,\theta}(v)),-1/2(v-d_x\Phi^r_{\alpha,\theta}(v)) \mid v \in T_x M_\theta^a\}.
\end{multline}

$(iii)$ Now define orthogonal projection $ \pi : M_\theta^a \times M_\theta^a \rightarrow  M_\theta^a \times M_\theta^a$ along the  direction of diagonal. That is, for $(x,y) \in M_\theta^a \times M_\theta^a$, there exists $(x',y')$ such that
$\pi(x,y) = (x', y') \in \Delta_{(x,y)} \cap \Delta_{(x,x)}^\perp  .$ Then,
\begin{align}\label{ident2}
T_{\pi(x,y)} \Delta_{\pi(x,y)} = T_{(x,y)}\Delta_{(x,y)}, \quad T_{\pi(x,y)} \Delta_{\pi(x,y)}^\perp = T_{(x,y)}\Delta_{(x,y)}^\perp.
\end{align}

Define a map $F^{(r)}: M_\theta^a \rightarrow M_\theta^a \times M_\theta^a$ given by
$F^{(r)} = \pi \circ \psi_{\alpha,\theta}^{(r)}$. In the local coordinate, by identification (\ref{ident}) and (\ref{ident2}),
\begin{equation}\label{Jac}
d_xF^{(r)}(v) = (1/2(v-d_x\Phi^r_{\alpha,\theta}(v)),-1/2(v-d_x\Phi^r_{\alpha,\theta}(v)), \quad v \in T_x M_\theta^a.
\end{equation}
By (\ref{commute}) and definition of $F^{(r)}$, we have $F^{(r)}(xc) = F^{(r)}(x)c$ for $c \in C_G$. Then for all $r \in \Z$, $F^{(r)}$ induces a quotient map $F_C^{(r)}:  M_\theta^a / C_G \rightarrow M_\theta^a \times M_\theta^a / C_G$. From \eqref{Jac}, the range of differential $DF_C^{(r)}$ is determined by $I-D\Phi^r_{\alpha,\theta}$. 

In the next lemma, we verify the range of differential map $DF_C^{(r)}$.

\begin{lemma}\label{522} For all $r \in \Z \backslash \{0\}$, range of $I-D\Phi^r_{\alpha,\theta}$ on $\II/C_\II(X_{\alpha}) $ coincides with $ \ran(\ad_{X_{\alpha}})$ and Jacobian of $F_C^{(r)}$ is non-zero constant. 
\end{lemma}
\begin{proof}
Recall that $\Phi^r_{\alpha,\theta}$ is $r$-th return map on $M_\theta^a$. We find differential in the direction of each $Y^j_i$ for fixed $i$ and $j$. For $x \in M$, set a curve 
$\gamma^x_{i,j}(t) = x\exp(tY_i^{(j)})\exp(rX_{\alpha}). $
Note that 
\begin{align*}
\exp({tY_i^{(j)}})\exp(rX_{\alpha}) & = \exp(rX_{\alpha})\exp(-rX_{\alpha})\exp({tY_i^{(j)}})\exp(rX_{\alpha})\\
& = \exp(rX_{\alpha}) \exp(e^{-r (\ad_{X_{\alpha}})}({tY_i^{(j)}}))
\end{align*}
and
$$\dfrac{d}{dt}\gamma^x_{i,j}(t)\mid_{t=0} = e^{-r (\ad_{X_{\alpha}})}({Y_i^{(j)}}).$$
By definition,
$ \frac{\partial \Phi^r_{\alpha, \theta}}{\partial s_{i}^{(j)}}(x) = \dfrac{d}{dt}(\gamma^x_{i,j}(t))\mid_{t=0}$ and we have $I-D\Phi^r_{\alpha,\theta} = I-\sum_{(i,j) \in J}\frac{\partial \Phi^r_{\alpha, \theta}}{\partial s_{i}^{(j)}}$. Then, 
\begin{equation}\label{root2}
(I-D\Phi^r_{\alpha,\theta})(\sum_{(i,j) \in J} s_{i}^{(j)}Y_{i}^{(j)}) =  [r(\ad_{X_{\alpha}})(\sum_{k=0}^\infty \frac{(-1)^k}{(k+1)!}(\ad_{X_{\alpha}})^k)](\sum_{(i,j) \in J} s_{i}^{(j)}Y_{i}^{(j)}).\end{equation}
Therefore, range of $I-D\Phi^r_{\alpha,\theta}$ is contained in $\ran(\ad_{X_{\alpha}})$.

Conversely, $\frac{1-e^{-\ad_{X_{\alpha}}}}{\ad_{X_{\alpha}}} = \sum_{k=0}^\infty \frac{(-1)^k}{(k+1)!}(\ad_{X_{\alpha}})^k$  is invertible and 
\begin{equation}\label{root3}
(I-D\Phi^r_{\alpha,\theta})\left((\frac{1-e^{-\ad_{X_{\alpha}}}}{\ad_{X_{\alpha}}})^{-1}(\sum_{(i,j) \in J} s_{i}^{(j)}Y_{i}^{(j)})\right) = r(\ad_{X_{\alpha}})(\sum_{(i,j) \in J} s_{i}^{(j)}Y_{i}^{(j)}).
\end{equation}
Therefore, we conclude that range of $I-D\Phi^r_{\alpha,\theta}$ is $ \ran(\ad_{X_{\alpha}})$.

If  $\sum_{(i,j) \in J} s_{i}^{(j)}Y_{i}^{(j)} \in C_\II(X_{\alpha})$, then $(I-D\Phi^r_{\alpha,\theta})(\sum_{(i,j) \in J} s_{i}^{(j)}Y_{i}^{(j)}) = 0$ and kernel of $I-D\Phi^r_{\alpha,\theta}$ is $C_\II(X_{\alpha})$. I.e, $I-D\Phi^r_{\alpha,\theta}$ is bijective on $\II/C_\II(X_{\alpha}) $. Thus, by (\ref{root2}) Jacobian of $I-D\Phi^r_{\alpha,\theta}$ is non-zero constant and it concludes the statement. 
\end{proof}

\textbf{Setting (continued).} $(iv)$ Set submanifold $\mathcal{S} \subset M_\theta^a \times M_\theta^a$ that consists of diagonal $\Delta$ and coordinates of generators in normal (transverese) directions. Denote its quotient $\mathcal{S}_C = \mathcal{S}/C_G \subset M_\theta^a \times M_\theta^a / C_G$.
Then, following Lemma \ref{522}, we obtain transversality of $F_C^{(r)}$ to $\mathcal{S}_C$. For every $p \in (F_C^{(r)})^{-1}(\mathcal{S}_C)$, 
the transversality holds on tangent space:
\begin{equation}\label{tra3}
T_{F_C^{(r)}(p)}\mathcal{S}_C + DF_C^{(r)}(T_p M_\theta^a/C_G)  = T_{F_C^{(r)}(p)}(M_\theta^a \times M_\theta^a/C_G).
\end{equation}

$(v)$ Denote Lebesgue measure $\LL^{a+1} (=  vol_{M})$ on nilmanifold $M$ and conditional measure $\LL_\theta^{a} (= vol_{M_\theta^a})$ on transverse manifold  $M_\theta^a$. On quotient space $M_\theta^a/C_G$, we write measure $\LL_\theta^{c} (= vol_{M_\theta^a/C_G})$. Similarly, we set conditional measure $\mu_\theta^{a}(= vol_{M_\theta^a \times M_\theta^a})$ on product manifold and $\mu_\theta^{c}(= vol_{M_\theta^a \times M_\theta^a/C_G})$ on its quotient space.

Denote image of $F^{(r)}$ by $M_{\theta,r}^a:= F^{(r)}(M_\theta^a) \subset M_\theta^a \times M_\theta^a$ and  $M_{\theta,r,C}^a:= F_C^{(r)}(M_\theta^a/C_G)$. We write its conditional Lebesgue measure $\mu_{\theta,r}^{a}:= \mu_\theta^{a}|_{M_{\theta,r}^a}$ and $\mu_{\theta,r}^{c}:= \mu_\theta^{c}|_{M_{\theta,r,c}^a}$ respectively.

 For any open set $U_{\mathcal{S}_C} \subset M_\theta^a \times M_\theta^a / C_G$, we write push-forward measure $(F_C^{(r)})_* \LL_\theta^{c}$ 
\begin{align*}
(F_C^{(r)})_* \LL_\theta^{c}(U_{\mathcal{S}_C} \cap M_{\theta,r,C}^a) & = \LL_\theta^{c}((F_C^{(r)})^{-1}(U_{\mathcal{S}_C} \cap  M_{\theta,r,C}^a))\\
& = \int_{U_{\mathcal{S}_C}} \sum_{x \in (F_C^{(r)})^{-1}(\{z\}), z \in U_{\mathcal{S}_C}} \frac{1}{Jac (F_C^{(r)}(x)) } dvol_{M_{\theta,r,c}^a}(z).
\end{align*}
By compactness of $M_\theta^a$ (or $M_\theta^a/C_G$), the above expression is finite. By Lemma \ref{522},  Jacobian of $F_C^{(r)}$ is constant and $(F_C^{(r)})_* \LL_\theta^{c} = \mu_{\theta,r}^{c}$ is Lebesgue. 

 By invariance of action of centralizer, for any neighborhood $U_{\mathcal{S}} \in M_\theta^a \times M_\theta^a$ with $U_{\mathcal{S}_C} = U_{\mathcal{S}}/C_G$,
\begin{equation}\label{pushmeasure}
\mu_{\theta,r}^{c}(U_{\mathcal{S}_C} \cap M_{\theta,r,C}^a) = \mu_{\theta,r}^{a}(U_{\mathcal{S}}\cap M_{\theta,r}^a)
\end{equation}
and by definition of conditional measure, 
\begin{equation}\label{pushm}
\mu_{\theta,r}^{a}(U_{\mathcal{S}}\cap M_{\theta,r}^a) = \mu_{\theta}^{a}(U_{\mathcal{S}}).
\end{equation}

 Let $d$ be a distance function in $M_\theta^a \times M_\theta^a$ and we abuse notation $d$ for induced distance on $M_\theta^a \times M_\theta^a/C_G$. Set $U_{\delta} = \{ z \in M_\theta^a \times M_\theta^a  \mid d(z,\mathcal{S})  < \delta \}$ be a $\delta$-tubular neighborhood of $\mathcal{S}$ and $U_{\delta,C} = \{ z \in M_\theta^a \times M_\theta^a/C_G  \mid d(z,\mathcal{S}_C)  < \delta \}$ be its quotient.
 \medskip

 Define \emph{almost-periodic set} (set of $r$-th close return) on the diagonal 
\begin{equation}\label{APset}
AP^{r}(\U_\delta) := \{ x \in M_\theta^a \mid d(F^{(r)}(x),\mathcal{S}) <  \delta \}.
\end{equation}
Since $F^{(r)}$ commutes with $C_G$, $AP^{r}(\U_\delta)/C_G = \{ x \in M_\theta^a/C_G \mid d(F_C^{(r)}(x),\mathcal{S}_C) <  \delta \}$ and $ \LL_\theta^{a}(AP^{r}(\U_\delta))  = \LL_\theta^{c}(AP^{r}(\U_\delta)/C_G) $.  
\medskip

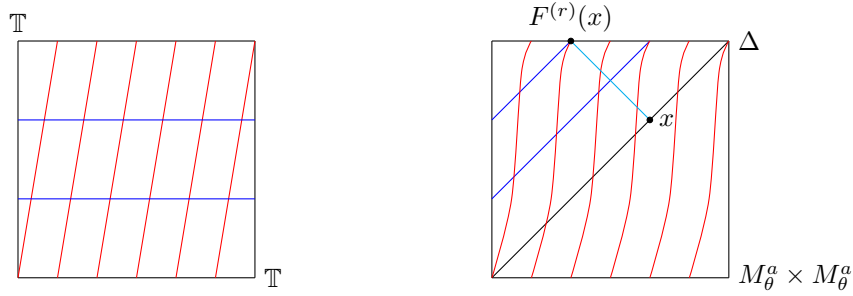
\begin{figure}\centering
\begin{tikzpicture}[scale=1.05]
      \draw[-] (0,0) -- (3,0) node[right] {$\T$};
      \draw[-] (0,0) -- (0,3) node[above] {$\T$};
            \draw[-] (0,3) -- (3,3) node[above] {};
             \draw[-] (3,0) -- (3,3) node[above] {};
                   \draw[-][red] (0,0) -- (0.5,3) node[above] {};      
            \draw[-][red] (0.5,0) -- (1,3) node[above] {};      
            \draw[-][red] (1,0) -- (1.5,3) node[above] {};      
                \draw[-][red] (1.5,0) -- (2, 3) node[above] {};      
                             \draw[-][red] (2,0) -- (2.5,3) node[above] {};      
                          \draw[-][red] (2.5,0) -- (3,3) node[above] {};      
                           \draw[-][blue] (0,2) -- (3,2) node[above] {};      
                          \draw[-][blue] (0,1) -- (3,1) node[above] {};      
             

      \draw[-] (6,0) -- (9,0) node[right] {$M_\theta^a \times M_\theta^a$};
            \draw[-] (6,0) -- (6,3) node[left, above] {};
         \draw[-] (6,3) -- (9,3) node[right] {};
        \draw[-] (9,0) -- (9,3) node[right] {$\Delta $};
      \draw[-] (6,0) -- (9,3) node[above] {};      
      \draw[-][blue] (6,1) -- (8,3) node[above] {};      
      \draw[-][blue] (6,2) -- (7,3) node[above] {};      
    \draw [red, xshift=0cm] plot [smooth, ] coordinates { (6,0)(6.25,1)(6.375,2.6)(6.5,3)};
        \draw [red, xshift=0cm] plot [smooth, ] coordinates { (6.5,0)(6.75,1)(6.875,2.6)(7.0,3)};
                \draw [red, xshift=0cm] plot [smooth, ] coordinates { (7.0,0)(7.25,1)(7.375,2.6)(7.5,3)};
                                \draw [red, xshift=0cm] plot [smooth, ] coordinates { (7.5,0)(7.75,1)(7.875,2.6)(8,3)};
                \draw [red, xshift=0cm] plot [smooth, ] coordinates { (8.0,0)(8.25,1)(8.375,2.6)(8.5,3)};
                                \draw [red, xshift=0cm] plot [smooth, ] coordinates { (8.5,0)(8.75,1)(8.875,2.6)(9,3)};
 node[anchor=south, above, right]
                         \draw[-][cyan] (7,3) -- (8,2) node[above,right] {};
                                   \draw (7,3) circle[radius=1pt]; \fill(7,3) circle[radius=1pt]
 node[anchor=south, above]{$F^{(r)}(x)$};
 \draw (8,2) circle[radius=1pt]; \fill(8,2) circle[radius=1pt]
 node[anchor=south, below, right]{$x$};

    \end{tikzpicture}
    \caption{Illustration of displacement $F^{(r)}$ in product $M_\theta^a \times M_\theta^a$ and comparison with uniform expanding map.} \label{fig:M1}
\end{figure}

 The following  volume estimate of \emph{almost-periodic set} holds.
\begin{lemma}\label{attract} Let $U_{\delta,C}$ be any tubular neighborhood of $\mathcal{S}_C$ in $M_\theta^a \times M_\theta^a/C_G$.
  For all $r \in \Z \backslash \{0\}$, the conditional measure $vol_{M_\theta^a}$ of $AP^{r}(\U_\delta)$ is given as follows:
$$ \LL_\theta^{a}(AP^{r}(\U_\delta))  = \mu_{\theta,r}^{c}(U_{\delta,C} \cap M_{\theta,r,C}^a).$$
\end{lemma}
\begin{proof}
By previous setting $(v)$, it suffices to prove $\LL_\theta^{c}(AP^{r}(\U_\delta)/C_G) = \mu_{\theta,r}^{c}(U_{\delta} \cap M_{\theta,r,C}^a)$. Note that $(F_C^{(r)})^{-1}(AP^{r}(\U_\delta)/C_G) = \{ x \in M_\theta^a / C_G \mid d(z,\mathcal{S}_C) <  \delta \}$ if $z = F_C^{(r)}(x)$ for some $x \in M_\theta^a/C_G$, otherwise it is an empty set. 

Then, $(F_C^{(r)})^{-1}(AP^{r}(\U_\delta)/C_G) = (U_{\delta,C} \cap M_{\theta,r,C}^a) $. Thus, by definition of push-forward measure, the equality holds.
\end{proof}

 Recall that $F_C^{(r)}: M_\theta^a/C_G \rightarrow M_{\theta,r,C}^a$ has non-zero constant Jacobian if $r \neq 0$ by Lemma \ref{522} and it is a local diffeomorphism. Thus, by transversality of $F_C^{(r)}$, in a small tubular neighborhood $U$, $F_C^{(r)}$ is covering. 
\begin{lemma}
For any $z \in U\cap M_{\theta,r,C}^a$, there exist finite number of pre-images of $F_C^{(r)}$. 
\end{lemma}
\proof
If we suppose that $(F_C^{(r)})^{-1}(z)$ contains infinitely many different points, then since the manifold $M_\theta^a$ is compact (and $M_\theta^a/C_G$ is compact), there exists a sequence of pairwise different points $x_i \in (F_C^{(r)})^{-1}(z)$, which converges to $x_0$. We have $(F_C^{(r)})(x_0) = z$ and by inverse function theorem, the point $x_0$ has  a neighborhood $U'$ in which $F_C^{(r)}$ is a homeomorphism. In particular, $U'\backslash \{x_0\} \cap (F_C^{(r)})^{-1}(z) = \emptyset$, which is a contradiction.
\qed\\

Set $N_r(z) = \#\{x \in M_\theta^a/C_G  \mid F_C^{(r)}(x) = z\}$ the number of pre-images of $F_C^{(r)}$. The number $N_r(z)$ is independent of choice of $z \in U\cap M_{\theta,r,C}^a$ since Jacobian is constant and degree of map is invariant (see \cite[\S3]{DAS}). 
\medskip

 Now we introduce the volume estimate of $\delta$-neighborhood $U_{\delta,C}$.

\begin{proposition}\label{attract2} The following volume estimate holds:  for any $r \neq 0$, there exists $C:= C(M_\theta^a) >0$ such that
$$\mu_{\theta,r}^{c}(U_{\delta,C}\cap M_{\theta,r,C}^a) < C\delta.$$
\end{proposition}
\begin{proof}
Let $U \subset M_\theta^a \times M_\theta^a/C_G$ be a tubular neighborhood of $\mathcal{S}_C$ that contains $U_{\delta,C}$ with the following condition:
\begin{equation}\label{condition42}
vol_{M_\theta^a \times M_\theta^a/C_G}(U_{\delta,C}) = \delta vol_{M_\theta^a \times M_\theta^a/C_G}(U).
\end{equation}
If $U\cap M_{\theta,r,C}^a = \emptyset$, then there is nothing to prove since $U_{\delta,C}\cap M_{\theta,r,C}^a = \emptyset$. Assume $z \in U\cap M_{\theta,r,C}^a$ and let $\{J_k\}_{k \geq 1}$ be connected components of $(F_C^{(r)})^{-1}(U \cap M_{\theta,r,C}^a)$. We firstly claim that $F_C^{(r)}|_{J_k}$ is injective.

Given $z \in U\cap M_{\theta,r,C}^a$, assume that there exist $x_1 \neq x_2 \in J_k$ for some $k$ such that $z = F_C^{(r)}|_{J_k}(x_1) = F_C^{(r)}|_{J_k}(x_2)$. Let $\gamma : [0,1] \rightarrow J_k$ be a path that connects $\gamma(0) = x_1$ and $\gamma(1) = x_2$.   Set the lift of path $\tilde \gamma = F_C^{(r)}|_{J_k} \circ \gamma : [0,1] \rightarrow U$. Then $\tilde \gamma(0) = \tilde \gamma(1) = z$ and $\tilde \gamma$ is a loop in $U$. Since $U$ is simply connected, $\tilde \gamma$ is contractible and there exists a homotopy of path $g_s: [0,1] \rightarrow U$ such that $g_0 = \tilde \gamma$ is homotopic to a constant loop $g_1 = c$ by fixing two end points $F_C^{(r)}|_{J_k}(x_1) = F_C^{(r)}|_{J_k}(x_2) = z$ for $s \in [0,1]$ . 

Note that $F_C^{(r)}|^{-1}_{J_k} \circ g_s$ is a lift of homotopy $g_s$, and lift of $g_0$ is $\gamma = F_C^{(r)}|^{-1}_{J_k}(\tilde \gamma)$ with fixed end points $x_1$ and $x_2$. By continuity of homotopy, $g_s$ also keeps the same end points $x_1$ and $x_2$ fixed for all $s \in [0,1]$.
  Since $g_1$ is constant loop and its lift should be a single point, $\gamma$ is homotopic to a constant. Since end points of $\gamma$ is fixed, it has to be a constant  but it leads a contradiction. Therefore, we have $x_1= x_2$.

By injectivity of $F_C^{(r)}|_{J_k}$, we obtain
$$(F_C^{(r)})^{-1}(U) = (F_C^{(r)})^{-1}(U \cap M_{\theta,r,C}^a) = \bigcup_{k=1}^{N_r} J_k.$$ 
Furthermore, we obtain the following equality:
\begin{equation}\label{jac43}
vol_{M_\theta^a/C_G}(J_k) = \frac{vol_{M_{\theta,r,C}^a}(U  \cap M_{\theta,r,C}^a)}{Jac(F_C^{(r)}|_{J_k})} = \frac{vol_{M_{\theta,r,C}^a}(U  \cap M_{\theta,r,C}^a)}{Jac(F_C^{(r)})}.
\end{equation}

Since volume of $M_\theta^a/C_G$ is a finite, 
\begin{align}\label{voljac}
N_r\Big(\frac{ vol_{M_\theta^a \times M_\theta^a/C_G}(U)}{Jac(F_C^{(r)})}\Big) &= N_r\Big(\frac{ vol_{M_{\theta,r,C}^a}(U \cap M_{\theta,r,C}^a)}{Jac(F_C^{(r)})}\Big) \\
&= \sum^{N_r}_{k=1} vol_{M_\theta^a/C_G}(J_k) < \infty.
\end{align}

Assume that $(F_C^{(r)})^{-1}(U_{\delta,C}) = \bigcup_{k=1}^{N_r} (F_C^{(r)}|_{J_k})^{-1}(U_{\delta,C})$. Then by \eqref{jac43} and definition of conditional measure,
\begin{align*}
vol_{M_\theta^a/C_G}((F_C^{(r)})^{-1}(U_{\delta,C})) &= \sum^{N_r}_{k=1} vol_{M_\theta^a/C_G}({(F_C^{(r)}|_{J_k})^{-1}}(U_{\delta,C})) \\
&= N_r \Big(\frac{vol_{M_{\theta,r,C}^a}(U_{\delta,C} \cap M_{\theta,r,C}^a)}{{Jac(F_C^{(r)})}}\Big) \\
& = N_r\Big(\frac{ vol_{M_\theta^a \times M_\theta^a/C_G}(U_{\delta,C})}{Jac(F_C^{(r)})}\Big).
\end{align*}

By previous last equality with condition \eqref{condition42}, 
\begin{equation}\label{condition45}
vol_{M_\theta^a/C_G}((F_C^{(r)})^{-1}(U_{\delta,C})) = N_r\Big(\frac{ \delta vol_{M_\theta^a \times M_\theta^a/C_G}(U)}{Jac(F_C^{(r)})}\Big).
\end{equation}

Therefore, combining \eqref{voljac} and \eqref{condition45}, there exists $C>0$ such that 
\begin{align*}
\mu_{\theta,r}^{c}(U_{\delta,C}) = (F_C^{(r)})_* vol_{M_\theta^a/C_G}(U_{\delta,C}) = vol_{M_\theta^a/C_G}((F_C^{(r)})^{-1}(U_{\delta,C}))  < C\delta.
\end{align*}
\end{proof}

\begin{definition}\label{I(Y)}
For any basis $Y = \{Y_1,\cdots Y_a \}$ of codimension 1 ideal $\II$ of $\n$, let $I:=I(Y)$ be the supremum of all constant $I' \in (0,\frac{1}{2})$ such that for any $x \in M$ the map 
\begin{equation}\label{maap}
\phi_x^Y : (s_1,\cdots,s_a) \mapsto x \exp(\sum_{i=1}^{a}s_i{Y_i}) \in M
\end{equation}
is local embedding (injective) on the domain
$$\{\textbf{s} \in \R^a \mid |s_i| < I' \text{ for all }i = 1,\cdots,a\}.$$
\end{definition}

For any $x, x' \in M$, set local distance $d_*$ (measured locally in the Lie algebra) on transverse section $M_\theta^a$ along $Y_i$ direction by $d_{Y_i}({x,x'}) = |s_i|$  if there is $\textbf{s}: = (s_1,\cdots, s_a) \in [-I/2,I/2]^a $ such that 
$$x' = x\exp(\sum_{i=1}^as_iY_i),$$
otherwise  $d_{Y_i}({x,x'}) =I$.

Recall the projection map $pr_1 : M \rightarrow \T^{n+1}$ onto the base torus. On transverse manifold, for all $\theta \in \T^1$, let $pr_\theta : M_\theta^a \rightarrow \T^n$ be the restriction to $M_\theta^a$. Then, by applying formula \eqref{retN} to projection to base torus,
\begin{equation}\label{dist-proj}
d_{Y_i}(pr_\theta(\Phi_{\alpha,\theta}^r(x)),pr_\theta(x)) =  r\alpha_i, \  1\leq i \leq n.
\end{equation}

For any $L \geq 1$, $r \in \Z$, $x \in M_\theta^a $ and given scaling factor $\rho = (\rho_1,\cdots, \rho_a) \in [0,1)^a$, we define
\begin{equation}
\begin{aligned}\label{77}
\epsilon_{r,L} &:= \max_{1\leq i \leq n}\min\{I, L^{\rho_i} d_{Y_i}(\Phi^r_{\alpha, \theta}(x),x) \};\\
\delta_{r,L}(x) &:= \max_{n+1 \leq i \leq a}\min\{I, L^{\rho_i}d_{Y_i}(\Phi^r_{\alpha, \theta}(x),x) \}.
\end{aligned}
\end{equation}
We note the distance $d_{Y_i}(\Phi^r_{\alpha, \theta}(x),x)$ on the generator level does not depend on choice of $x$. 
For this reason, we split the cases $\epsilon_{r,L}$ and $\delta_{r,L}(x)$ for step $\geq 2$.

The condition $\epsilon_{r,L} < \epsilon < I$ and $\delta' < \delta_{r,L}(x) < \delta < I$ are equivalent to saying 
\begin{equation}\label{xprime}
\Phi_{\alpha,\theta}^r(x) = x \exp(\sum_{i=1}^as_i{Y_i})
\end{equation}
for some vectors $\textbf{s}: = (s_1,\cdots, s_a) \in [-I/2,I/2]^a $ such that 
\begin{align*}
& |s_i| < \epsilon L^{-\rho_i}, \text{ for all } i \in \{1,\cdots, n\}; \\
& |s_i| < \delta L^{-\rho_i}, \text{ for all } i \in \{n+1,\cdots, a\}; \\
& |s_j| > \delta' L^{-\rho_j}, \text{ for some } j \in \{n+1,\cdots, a\}. 
\end{align*}

For every $r \in \Z \backslash \{0\}$ and $j \geq 0$, let ${AP^r_{j,L}} \subset  M$ be sets defined as follows
\begin{equation}\label{78}
 {AP_{j,L}^{r}} = \begin{cases} \emptyset & \text{ if } \epsilon_{r,L}>\frac{I}{2};\\
(\delta_{r,L})^{-1}\left((2^{-(j+1)}I, 2^{-j}I]\right) & \text{ otherwise.}
\end{cases}
\end{equation}

In the next lemma, Lebesgue measure of the set of almost-periodic points ${AP_{j,L}^{r}}$ on  $M$ is 
estimated by the volume of $\delta$-neighborhood $\U_\delta$.

\begin{lemma}\label{5.4} 
For all $r \in \Z \backslash \{0\}$, $j \in \N$, $L \geq 1$, 
the $(a+1)$ dimensional Lebesgue measure of the set $AP^r_{j,L}$ can be estimated as follows: there exists $C>0$ such that
$$ \LL^{a+1}(AP^r_{j,L}) \leq   \frac{C I^{a-n}}{2^{j(a-n)}} L^{-\sum_{i=n+1}^a\rho_i}.$$
\end{lemma}
\begin{proof}
Without loss of generality, we assume that $AP^r_{j,L} \neq \emptyset$. 

By Tonelli's theorem,
\begin{equation}\label{ttt}
 \LL^{a+1}(AP^r_{j,L}) = \int_{0}^1 \LL_\theta^{a}(AP^r_{j,L} \cap M_\theta^a)d\theta.
\end{equation}

Recall the definition $AP^{r}(\U_\delta)$ in \eqref{APset}. Choose $\delta = \frac{I^{a-n}}{2^{j(a-n)}} L^{-\sum_{i=n+1}^a\rho_i}$ and set $ \U_{\delta}^{L,j}:= \U_{\delta}$.
Then we claim that $ {AP_{j,L}^{r}} \cap M_\theta^a \subset AP^{r}(\U_{\delta}^{L,j})$.

 By identification of $x \in M_\theta^a$  to $ (x,x)$ in diagonal $\Delta \subset M_\theta^a \times M_\theta^a$, local distance $d_{Y_i}(\Phi^r_{\alpha, \theta}(x),x)$ is identified by the distance function $d$ in the product $M_\theta^a \times M_\theta^a$. 
 Thus, $x \in AP^r_{j,L} \cap M_\theta^a$ implies that $d(F^{(r)}(x),\mathcal{S}) < \delta$. That is, $AP^r_{j,L} \cap M_\theta^a \subset AP^r(\U_\delta^{L,j})$. 
By Lemma \ref{attract}, the volume estimate follows
\begin{align*}
\LL_\theta^{a}(AP^r_{j,L} \cap M_\theta^a) &\leq \LL_\theta^{a}(AP^{r}(\U_\delta^{L,j})) = \mu_{\theta,r}^{c}(\U_{\delta,C}^{L,j} \cap M_{\theta,r,C}^a).
\end{align*}
Finally, by Proposition \ref{attract2},
$$\mu_{\theta,r}^{c}(\U_{\delta,C}^{L,j} \cap M_{\theta,r,C}^a) \leq   \frac{CI^{a-n}}{2^{j(a-n)}} L^{-\sum_{i=n+1}^a\rho_i}.$$
Thus the proof follows from formula (\ref{ttt}).
\end{proof}

\subsection{Expected width bounds.}\label{width bounds} We prove a bound on the average width of a orbit on nilmanifold with respect to scaled basis. This section follows in the same way of \cite[\S 5.2]{FF14}. For completion of the proof, we repeat the similar arguments in nilmanifolds under transverse conditions. 

For $L\geq 1$ and $r \in \Z \backslash \{0\}$, let us consider the function
\begin{equation}
h_{r,L} = \sum_{j=1}^{\infty} \min\{2^{j(a-n)},(\frac{2}{\epsilon_{r,L}})^n \} \chi_{{AP_{j,L}^r}}.
\end{equation}
Define cut-off function $J_{r,L} \in \N$ by the formula:
\begin{equation}\label{cutoff}
J_{r,L} := \max\{j\in \N \mid 2^{j(a-n)} \leq (\frac{2}{\epsilon_{r,L}})^n \}.
\end{equation}
The function $h_{r,L}$ is
\begin{equation}\label{cutt1}
h_{r,L} = \sum_{j=1}^{J_{r,L}}2^{j(a-n)}\chi_{{AP_{j,L}^r}} + \sum_{j>J_{r,L}}(\frac{2}{\epsilon_{r,L}})^n\chi_{{AP_{j,L}^r}}.
\end{equation}

For every $L\geq 1$, let $\F_{\alpha}^{(L)}$ be the rescaled strongly adapted basis
\begin{align}
\F_{\alpha}^{(L)} & = (X_{\alpha}^{(L)},Y_1^{(L)},\cdots,Y_a^{(L)}) = (LX_{\alpha},L^{-\rho_1}Y_1,\cdots,L^{-\rho_a}Y_a).
\end{align}

For $(x,T) \in M \times \R$, let $w_{\F_{\alpha}^{(L)}}(x,T)$ denote the \emph{average width} of the (scaled) orbit segment 
$$\gamma^T_{X_{\alpha}^{(L)}}(x) : = \{ \phi_{X_{\alpha}^{(L)}}^t(x) \mid 0 \leq t \leq T \}.$$

We prove a bound for the average width of the orbit arc in terms of the following function
\begin{equation}\label{HTL}
H_{L}^T := 1 + \sum_{|r|=1}^{[TL]}h_{r,L}.
\end{equation}

\begin{definition}\label{def;cases} For $t  \in [0,T]$, we define a set of points $\Omega(t) \subset \{t\}\times \R^a$ as follows:

Case 1. 
If $\phi^t_{X_{\alpha}^{(L)}}(x) \notin \bigcup_{|r|=1}^{[TL]}\bigcup_{j>0} AP^r_{j,L}$, let $\Omega(t)$ be the set of all points $(t,s_1,\cdots,s_a)$ such that 
$$|s_i| < I/4, \ \  i \in \{1,\cdots,a\}.$$

If $\phi^t_{X_{\alpha}^{(L)}}(x) \in \bigcup_{|r|=1}^{[TL]}\bigcup_{j>0} AP^r_{j,L}$, then we consider two subcases.

Case 2-1. if $\phi^t_{X_{\alpha}^{(L)}}(x) \in \bigcup_{|r|=1}^{[TL]}\bigcup_{j>J_{r,L}} AP^r_{j,L}$, let $\Omega(t)$ be the set of all points $(t,\textbf{s})$ such that 
\begin{align*}
&|s_i| < \frac{1}{4}\min_{1\leq |r| \leq [TL]}\min_{j>J_{r,L}}\{\epsilon_{r,L}:\phi^t_{X_{\alpha}^{(L)}}(x) \in AP^r_{j,L} \}, \text{ for } i \in \{1,\cdots,n\};\\
&|s_i| < \frac{I}{4} \quad \text{for} \  i \in \{n+1,\cdots,a\}.
\end{align*} 

Case 2-2. if $\phi^t_{X_{\alpha}^{(L)}}(x) \in \bigcup_{|r|=1}^{[TL]}\bigcup_{j\leq J_{r,L}} AP^r_{j,L}\backslash \bigcup_{|r|=1}^{[TL]}\bigcup_{j>J_{r,L}} AP^r_{j,L}$, let $l$ be the largest integer such that 

$$\phi^t_{X_{\alpha}^{(L)}}(x) \in \bigcup_{|r|=1}^{[TL]}\bigcup_{ l \leq j\leq J_{r,L}} AP^r_{j,L}\backslash \bigcup_{|r|=1}^{[TL]}\bigcup_{j>J_{r,L}} AP^r_{j,L}, $$ 
\\
and let $\Omega(t)$ be the set of all points $(t, \textbf{s})$ such that 
\begin{align*}
&|s_i| < \frac{I}{4}, \quad \text{ for } i \in \{1,\cdots,n\};\\
&|s_i| < \frac{I}{4}\frac{1}{2^{l+1}}, \quad \text{for} \  i \in \{n+1,\cdots,a\}.
\end{align*}

We set 
$$\Omega := \bigcup_{t \in [0,T]} \Omega(t) \subset [0,T] \times [-I/4,I/4]^a  \subset [0,T] \times \R^a.$$
\end{definition}

\begin{lemma}\label{lem;85} The restriction to $\Omega$ of the map 
\begin{equation}\label{85}
 (t,\textbf{s})\in \Omega \mapsto x\exp(tX_{\alpha}^{(L)})\exp(\sum_{i=1}^as_i Y_i^{(L)}) 
 \end{equation}
  is injective.
\end{lemma}

\begin{proof}
For every $t \in [0,T]$, we define a set $\Omega(t) \subset \{t\}\times \R^a$ as follows: 

$$\Omega := \bigcup_{t\in[0,T]}\Omega(t) \subset  \R^a.$$
Then we set
\begin{equation}\label{86}
 \phi_{X_{\alpha}^{(L)}}^t(x) \exp(\sum_{i=1}^as_i{Y_i^{(L)}}) = \phi_{X_{\alpha}^{(L)}}^{t'}(x) \exp(\sum_{i=1}^as'_i{Y_i^{(L)}}).
\end{equation}

Let us assume $t' \geq t$. By considering the projection on the base torus, we have the following identity:
\begin{multline}
(t,s_1,\cdots,s_n) \text{ mod } \Z^{n+1}   = pr_1(\phi_{X_{\alpha}^{(L)}}^t(x)) \\
= pr_1(\phi_{X_{\alpha}^{(L)}}^{t'}(x)) = (t',s'_1,\cdots,s'_n) \text{ mod } \Z^{n+1},
\end{multline}
which implies $t \equiv t'$ modulo $\Z$. As $\phi_{X_{\alpha}}^{tL} = \phi^t_{X_{\alpha}^{(L)}}$, the number $r_0 = t' - t$ is a non negative integer satisfying $r_0 \leq TL$; hence $r_0 \leq [TL]$.
\medskip

If $r_0=0$, then $t' = t$ and $s'_i = s_i$. Then injectivity is obtained by definition of I. 
Assume that $r_0\neq 0$. Let $p,q \in M_\theta^a$ and then we have 
$$p: = \phi^t_{X_{\alpha}^{(L)}}(x), \quad q: = \phi^{t'}_{X_{\alpha}^{(L)}}(x) \Lra  q = \Phi_{\alpha,\theta}^{r_0}(p).$$
 From identity (\ref{86}) we have 
\begin{align}\label{88}
\begin{split}
q & = p \exp(\sum_{i=1}^as_i{Y_i^{(L)}})\exp(-\sum_{i=1}^as'_i{Y_i^{(L)}}) \\
& = p\exp( \sum_{i=1}^a(s_i'-s_i + P_i(s_i,s_i'))L^{-\rho_i}Y_i  )
\end{split}
\end{align}
where $P_i$ is polynomial expression following from Baker-Cambell-Hausdorff formula.\medskip

Note that $P_i = 0$ if $i = 1,\cdots, n$ and $|P_i| \leq \sum_{l = 1}^\infty 1/2|s_ls_l'|^l$ for $i>n$. Since $|s_i|, |s_i'| \leq \frac{I}{4} \ll 1$,
 $$q =  p\exp( \sum_{i=1}^a(s_i'-s_i + \epsilon_i)L^{-\rho_i}Y_i  ), \ \text{ for some } \epsilon_i \in [0,I_\epsilon)$$
where $I_\epsilon = \sum_{l=1} (\frac{I}{4})^l  = \frac{I}{4-I} < I/3.$ Thus for all $i \in \{1,\cdots,a\}$,

$$L^{\rho_i} d_{Y_i}(p, \Phi_{\alpha,\theta}^{r_0}(p))  =  L^{\rho_i}|(s_i'-s_i + \epsilon_i)L^{-\rho_i}| \leq |s_i'|+|s_i| + |\epsilon_i| $$

and
$$\epsilon_{r_0,L} = \max_{1\leq i \leq n}L^{\rho_i} d_{Y_i}(p, \Phi_{\alpha,\theta}^{r_0}(p))  \leq \max_{1\leq i \leq n} |s_i| +|s_i'|+|\epsilon_i|  \leq \frac{5}{6}I.$$
For the same reason, from formula (\ref{88}) we also obtain that
$$\delta_{r_0,L}(p) = \delta_{-r_0,L}(q) < I/2.$$
By defining $j_0 \in \N$ as the unique non-negative integer such that 
$$\frac{I}{2^{j_0+1}} \leq \delta_{r_0,L}(p) \leq \frac{I}{2^{j_0}} $$
and by the Definition \ref{I(Y)}, we have $p \in AP^{r_0}_{j_0,L}$ and $q \in AP^{-r_0}_{j_0,L}$.

If $j_0 > J_{r_0,L} = J_{-r_0,L},$ then $p,q \in \bigcup_{|r|=1}^{[TL]}\bigcup_{j \geq J_{r,L}}$. It follows that the sets $\Omega(t)$ and $\Omega(t')$ are both defined on case 2-1. 
Hence,
$$\epsilon_{r_0,L} \leq \max_{1\leq i \leq n} |s_i| +|s_i'|+|\epsilon_i|  \leq \frac{5}{6}\epsilon_{r_0,L}$$
which is a contradiction.

If the map in formula (\ref{85}) fails injective at points $(t,s)$ and $(t',s')$ with $t \geq t'$, then there are integers $r_0 \in [1,TL], j_0 \in [1,J(|r_0|)]$ and $\theta \in \T^1$ such that the points $p$ and $q$ satisfy 
$$q = \Phi_{\alpha,\theta}^{r_0}(p), \quad p,q \notin \bigcup_{|r|=1}^{[TL]}\bigcup_{ j > J_{r,L}} AP^r_{j,L}.$$

In this case, the sets $\Omega(t)$ and $\Omega(t')$ are both defined according to case (2-2). Let $l_1$ and $l_2$ as the largest integers such that
$$ p \in \bigcup_{|r|=1}^{[TL]}\bigcup_{l_1 \leq j \leq J_{r,L}} AP^r_{j,L} \quad \text{and} \quad
 q \in \bigcup_{|r|=1}^{[TL]}\bigcup_{l_2 \leq j \leq J_{r,L}} AP^r_{j,L}.$$
 
On case (2-2), we have 
$$|s_i| < \frac{I}{4} \frac{I}{2^{l_1+1}}, \quad |s_i'| < \frac{I}{4} \frac{I}{2^{l_2+1}}, \text{ for all } i\in \{n+1,\cdots,a\},$$
which also leads contradiction because $l_1,l_2 > j_0$ deduce the contradiction
$$\frac{I}{2^{j_0+1}}\leq \delta_{r_0,L}(p) \leq \max_{ i \geq n+1} |s_i| +|s_i'|+|\epsilon_i|  < \frac{I}{2^{l_1+1}} + \frac{I}{2^{l_2+1}} \leq \frac{I}{2} \frac{1}{2^{j_0+1}}.$$
Hence, the injectivity is proved.
\end{proof}

We simply reprove the bound of averaged width (see \cite[Lemma 5.5]{FF14}) in the general settings (under transversality conditions) by combining with Lemma \ref{lem;85}.
\begin{lemma}\label{HH} For all $x \in M$ and for all $T, L \geq 1$ we have
$$\frac{1}{w_{\F_{\alpha}^{(L)}}(x,T)} \leq \left(\frac{2}{I} \right)^a \frac{1}{T}\int_{0}^{T}H^{T}_{L}\circ \phi^{t}_{X_{\alpha}^{(L)}}(x)dt.$$
\end{lemma}
\proof
The width function $w_{\Omega}$ of the set $\Omega$ is given by the following:
\begin{equation} \label{awidth}
w_\Omega(t) = \begin{cases} 
(\frac{I}{2})^a & \text{case 1 } \\
\\
(\frac{I}{2})^a(\frac{\min\{\epsilon_{r,L}\}}{2})^n & \text{case 2-1 }\\
\\
(\frac{I}{2})^a2^{-(a-n)(l+1)} & \text{case 2-2},  \\
\end{cases} 
\end{equation} 
and it implies that 
\begin{equation}
\frac{1}{w_\Omega(t)} \leq
\begin{cases}
(\frac{2}{I})^{a-n} & \text{case 1 } \\
\\
(\frac{2}{I})^{a-n} \displaystyle\sum_{|r|=1}^{[TL]}\sum_{j>J_{r,L}}\frac{2^{n}\chi_{AP_{j,L}^r}(\phi_{X_{\alpha}^{(L)}}^t(x))}{(\epsilon_{r,L})^n} & \text{case 2-1 }\\
\\
(\frac{2}{I})^a \displaystyle\sum_{|r|=1}^{[TL]}\sum_{j>J_{r,L}}2^{(j+1)(a-n)}\chi_{AP_{j,L}^r}(\phi_{X_{\alpha}^{(L)}}^t(x))  & \text{case 2-2}.  \\
\end{cases} 
\end{equation}

By the definition of the function $H_{L}^T$ in formula (\ref{HTL}), we have
\begin{equation}
\frac{1}{w_\Omega(t)} \leq \left(\frac{2}{I}\right)^a H_{L}^T\circ \phi^{t}_{X_{\alpha}^{(L)}}(x), \text{ for all } t \in [0,T].
\end{equation}
From the definition (\ref{width}) of the average width of the orbit segment $\{x\exp{(tX_{\alpha}^{(L)})} \mid 0 \leq t \leq T\}$, we have the estimate 

$$ \frac{1}{w_{\F_{\alpha}^{(L)}}(x,T)} \leq \frac{1}{T}\int_{0}^{T}\frac{dt}{w_{\Omega}(t)} \leq \left(\frac{2}{I}\right)^a\frac{1}{T}\int_{0}^{T}H_{L}^T\circ \phi^{t}_{X_{\alpha}^{(L)}}(x)dt. $$
\qed

\begin{lemma}\label{5.6} For all $r \in \Z \backslash \{0\}$ and for all $L \geq 1$, the following estimate holds:
$$ \left|\int_{M} h_{r,L}(x)dx\right| \leq  CI^{a-n}(1+J_{r,L})L^{-\sum_{i=n+1}^a\rho_i}. $$

\end{lemma}
\proof It follows from the  Lemma \ref{5.4} that for $r \neq 0$ and for all $j \geq 0$, the Lebesgue measure of the set $AP^r_{j,L}$ satisfies the following bound:
\begin{equation}\label{APP}
 \LL^{a+1}(AP^r_{j,L}) \leq \frac{C I^{a-n}}{2^{j(a-n)}} L^{-\sum_{i=n+1}^a\rho_i}.
\end{equation}
From the formula (\ref{cutt1}), it follows that 
\begin{align*}
\int_{M} h_{r,L}(x)dx  \leq 1& + \sum_{j=1}^{J_{r,L}}2^{j(a-n)} \LL^{a+1}({AP_{j,L}^r})\\
& + \sum_{j>J_{r,L}} \frac{2^n\LL^{a+1}({AP_{j,L}^r})}{(\epsilon_{r,L})^n}.
\end{align*}
By estimate in the formula (\ref{APP}), we immediately have that
$$\sum_{j=1}^{J_{r,L}} 2^{j(a-n)}\LL ^{a+1}(AP_{j,L}^r)\leq  C I^{a-n}J_{r,L} L^{-\sum_{i=n+1}^a\rho_i}.$$
By the definition of the cut-off in formula (\ref{cutoff}) we have the bound
$$ \frac{2^{n-(J_{r,L}+1)(a-n)}}{(\epsilon_{r,L})^n} \leq 1,$$
and by an estimate on a geometric sum
\begin{align*}
\sum_{j>J_{r,L}} \frac{2^n\LL^{a+1}({AP_{j,L}^r})}{(\epsilon_{r,L})^n} &\leq \frac{2^{n-(J_{r,L}+1)(a-n)}}{(\epsilon_{r,L})^n}CI^{a-n}L^{-\sum_{i=n+1}^a\rho_i} \\
&\leq CI^{a-n}L^{-\sum_{i=n+1}^a\rho_i}.
\end{align*}
\qed

\subsection{Diophantine estimates}\label{section;5.3} In this section we review the concept of simultaneous Diophantine condition. The bounds on the expected average width is estimated under Diophantine conditions. 

\begin{definition}\label{sad}
For any basis $\bar{Y}:= \{ \bar{Y}_1, \cdots, \bar{Y}_n\} \subset \R^n$, let $\bar{I}:= \bar{I}(\bar{Y})$ be the supremum of all constants $\bar{I'}>0$ such that the map 
$$(s_1,\cdots,s_n) \rightarrow \exp(\sum_{i=1}^ns_i\bar{Y}_i) \in \T^n$$
is a local embedding on the domain
$$\{\s \in \R^n \mid |s_i| < \bar{I}' \text{ for all } i=1,\cdots,n \}.$$
\end{definition}
For any $\theta \in \R^n$, let $[\theta] \in \T^n$ its projection onto the torus $\T^n := \R^n/\Z^n$ and let 
$$|\theta|_1 = |s_1|, \cdots, |\theta|_i = |s_i|, \cdots, |\theta|_n = |s_n|,$$
if there is $\s:=(s_1,\cdots, s_n) \in [-\bar{I}/2, \bar{I}/2]^n$ such that 
$$[\theta] = \exp(\sum_{i=1}^ns_i\bar{Y}_i) \in \T^n;$$
otherwise we set $|\theta|_1 = \cdots = |\theta|_n = \bar{I}$.\\

Here is Simultaneous Diophantine condition used in \cite[Def. 5.8]{FF14}.

\begin{definition}\label{simultdio} A vector $\alpha \in \R^n\backslash\Q^n$ is simultaneously Diophantine of exponent $\nu \geq 1$, say $\alpha \in DC_{n,\nu}$ if there exists a constant $c(\alpha) >0$ such that, for all $r\in \N \backslash \{0\} $,
$$\min_i\norm{r \alpha_i} = d(r\alpha, \Z^n) = \norm{r\alpha} \geq \frac{c(\alpha)}{r^\frac{\nu}{n}}.$$  
\end{definition}

\begin{definition}\label{dio1} Let $\sigma = (\sigma_1,\cdots,\sigma_n) \in (0,1)^n$ be such that $\sigma_1+\cdots+\sigma_n = 1$. For any $\alpha = (\alpha_1,\cdots,\alpha_n) \in \R^n$, for any $N \in \N$ and every $\delta>0$, let 
$$R_\alpha(N,\delta) = \{r\in [-N,N] \cap \Z \mid |r\alpha|_1\leq \delta^{\sigma_1},\cdots, |r\alpha|_n\leq \delta^{\sigma_n}\}. $$

For every $\nu >1 $, let $D_n(\bar{Y},\sigma, \nu) \subset (\R\backslash \Q)^n$ be the subset defined as follows: the vector $\alpha \in D_n(\bar{Y},\sigma, \nu)$ if and only if there exists a constant $C(\bar{Y},\sigma, \alpha) >0$ such that for all $N \in \N$ for all $\delta >0$, 
\begin{equation}\label{rdio}
\# R_\alpha (N,\delta) \leq C(\bar{Y}, \sigma, \alpha)\max\{N^{1-\frac{1}{\nu}}, N\delta\}. 
\end{equation}
\end{definition}

The Diophantine condition implies a standard simultaneous Diophantine condition. We quote following Lemmas proved in \cite{FF14}.

\begin{lemma}\cite[Lemma 5.9]{FF14}\label{lem59} Let $\alpha \in D_n$. For all $r \in \Z\backslash \{0\}$, we have 
$$\max\{|r\alpha|_1,\cdots, |r\alpha|_n\} \geq \min\{\frac{\bar{I}^2}{4}, \frac{1}{[1+C(\bar{Y},\sigma,\alpha)]^{2\nu}} \}\frac{1}{|r|^\nu}.$$
\end{lemma}

For any vector $\sigma = (\sigma_1, \cdots, \sigma_n) \in (0,1)^n$ such that $\sigma_1+\cdots+\sigma_n = 1$, let 
\[m(\sigma) = \min\{\sigma_1,\cdots, \sigma_n\} \quad \text{and} \quad M(\sigma) =  \max\{\sigma_1,\cdots, \sigma_n\}.\]

\begin{lemma}\cite[Lemma 5.12]{FF14}\label{512} 
For all bases $\bar{Y} \subset \R^n$, for all $\sigma = (\sigma_1, \cdots, \sigma_n) \in (0,1)^n$ such that $\sigma_1+\cdots+\sigma_n = 1$ and for all $\nu \geq 1$, the inclusion 
$$DC_{n,\nu} \subset D_n(\bar{Y},\sigma,\nu)$$
holds under the assumption 
$$\mu \leq \min\{\nu, [\frac{M(\sigma)}{\nu}+1-\frac{1}{n}]^{-1}, [\frac{1}{\nu} + (1-\frac{2}{n})(1-\frac{m(\sigma)}{M(\sigma)}) ]^{-1}\}.$$
The set $D_n(\bar{Y},\sigma,\nu)$ has full measure if 
\begin{equation}
\frac{1}{\nu} < \min\{ [M(\sigma)n]^{-1},1-(1-\frac{2}{n})(1-\frac{m(\sigma)}{M(\sigma)}) \}.
\end{equation}
\end{lemma}

In dimension one, the vector space has unique basis up to scaling. The following result is immediate.

\begin{lemma}\cite[Lemma 5.13]{FF14}
For all $\nu \geq 1$ the following identity holds:
$$DC_{1,\nu} = D_1(\nu).$$
\end{lemma}

Let $\F_\alpha := (X_\alpha, Y)$ be a basis and let $\bar{Y} = \{\bar{Y}_1,\cdots , \bar{Y}_{n} \} \in \R$   denote the projection of the basis of codimension 1 ideal $\II$  onto the Abelianized Lie algebra $\bar\n := \n/[\n,\n] \approx \R^{n}$. For $\rho = (\rho_1,\cdots,\rho_a) \in [0,1)^a$, we write a vector of scaling exponents  
$$\bar\rho = (\rho_1,\cdots,\rho_n), \quad |\bar{\rho}| = \rho_1+\cdots + \rho_n.$$

Let $\alpha_1 = (\alpha^{(1)}_1, \cdots, \alpha^{(1)}_{n} ) \in D_n(\bar Y, \bar{\rho}/|\bar{\rho}|,\nu)$. For brevity,
let $C(\bar Y, \bar{\rho}/|\bar{\rho}|,\alpha_1)$ denote the constant appeared in \eqref{rdio} for $\alpha_1 \in D_n(\bar Y, \bar{\rho}/|\bar{\rho}|,\nu)$. Let
\begin{equation}\label{100}
C(\alpha_1) = 1+C(\bar Y, \bar{\rho}/|\bar{\rho}|,\alpha_1).
\end{equation}

We prove the upper bound on the cut-off function in the formula (\ref{cutoff}). Let $I = I(Y)$ and $\bar{I} = \bar I( \bar Y)$ be the positive constant introduced in the Definition \ref{I(Y)} and \ref{sad}. We observe that $I \leq \bar{I} $ since the basis $\bar Y$ is the projection of the basis $Y \subset \n'$ and the canonical projection commutes with exponential map. Then the following logarithmic upper bound holds.

\begin{lemma}\cite[Lemma 5.14]{FF14}\label{5.14} For every $\rho \in [0,1)^a$, for every $\nu \leq 1/|\bar\rho|$ and for every $\alpha \in D_n(\bar Y, \bar{\rho}/|\bar{\rho}|,\alpha)$, there exists a constant $K>0$ such that for all $T \geq 1$ and for all $r \in \Z \backslash \{0\}$, the following bound holds: 
$$J_{r,L} \leq K\{1+\log^+[I(Y)^{-1}]+\log C(\alpha_1)\}(1+\log|r|). $$
\end{lemma}
\begin{proof}
 By Lemma \ref{lem59} and by the definition of $\epsilon_{r,L}$ in formula (\ref{77}), it follows that, for all $T>0, L\geq 1$ and for all $r \in \Z \backslash\{0\}$, we have 
$$\epsilon_{r,L} \geq \max_{1\leq i \leq n}\min\{I, |r\alpha_1|_i\} \geq \min\{I, \frac{\bar{I}^2}{4}, \frac{1}{[1+C(\alpha_1)]^{2\nu}} \}\frac{1}{|r|^\nu}. $$
It follows by the bound above and by the definition of the cut-off function (\ref{cutoff}), 
\[J_{r,L} \leq \frac{n}{a-n}(3\log2+3\log^{+}(1/I)+2\nu \log[1+C(\alpha_1)]+\nu\log|r| ).\]
\end{proof}

Assume that there exists $\nu \in 1/|\bar{\rho}|$ such that $\alpha_1 \in D_n(\bar Y,\bar{\rho}/|\bar{\rho}|,\nu)$. For brevity, we introduce the following notation:
\begin{equation}\label{101}
\HH(Y,\rho,\alpha) = 1+ I(Y)^{a-n}C(\alpha_1)\{1+\log^+[I(Y)^{-1}]+\log{C(\alpha_1)} \}.
\end{equation}

\begin{theorem}\cite[Theorem 5.15]{FF14}\label{515} For every $\rho \in [0,1)^a$, for every $\nu \leq 1 / |\bar{\rho}|$ such that $\alpha_1 = \alpha_i^{(1)}   \in D_n(\bar{Y},\bar{\rho},\nu) $ there exists a constant $K'>0$ such that for all $T>0$ and for all $L \geq 1$, the following bound holds:
$$\left|\int_{M}H_{L}^T(x)dx\right| \leq K'\HH(Y,\rho,\alpha)(1+T)(1+\log^+ T+\log L)L^{1-\sum_{i=1}^a\rho_i}. $$
\end{theorem}
\proof

By the definition of $H_{L}^T$ in the formula (\ref{HTL}), the statement follows from the Lemma \ref{5.6} and Lemma \ref{5.14}. In fact, for all $r \in \Z \backslash \{0\}$ and $j \geq 0$, by definition (\ref{78}) the set $AP_{j,L}^{r}$ is nonempty only if $\epsilon_{r,L}<\frac{I}{2}$. Since $\nu \leq 1/|\bar{\rho}|$, it follows from the definition of the Diophantine class $D_n$
$$\# \{r \in [-TL,TL] \cap \Z\backslash \{0\} \mid AP_{j,L}^{r} \neq \emptyset   \} \leq C(\bar Y, \sigma, \alpha_1)(1+T)L^{1-|\bar{\rho}|}. $$

Hence, the statement follows from the Lemma \ref{5.6} and \ref{5.14}.
\qed

\begin{remark}
Main idea of the proof above follows from the Lemma  \ref{5.6} which is based on the estimate to the upper bound of the measure of almost periodic set $AP^r_{j,L}$. In the Lemma \ref{5.4}, this bound is independent of choice of transverse section $M_\theta^a$.
\end{remark}

\subsection{Width estimates along orbit segments}\label{Width estimates}
In this section, we introduce the definition of \emph{good points}, which is crucial in controlling the average width estimate. In \S \ref{sec;irreducible} this idea will be used often to handle bound of ergodic averages for almost all points on $M$.

\begin{definition}\label{good} For any increasing sequence $(T_i)$ of positive real numbers, let $h_i \in [1,2]$ denote the ratio $\log{T_i}/[\log T_i]$ for every $T_i \geq 1$. Set $N_i = [\log T_i]$ and $T_{j,i} = e^{jh_i}$ for integer $j \in [0,N_i]$. 

 Let $\zeta >0$ and $w>0$. A point $x \in M$ is $(w,T_i,\zeta)$-\emph{good} for the basis $\F_\alpha$ if setting $y_i = \phi_{X_{\alpha}}^{T_i}(x)$, then for all $i \in \N$ and for all $0 \leq j \leq N_i$
$$w_{\F_\alpha^{(T_{j,i})}}(x,1) \geq w/T_i^\zeta,  \quad w_{\F_\alpha^{(T_{j,i})}}(y_i,1) \geq w/T_i^\zeta.$$
\end{definition}

\begin{lemma}\cite[Lemma 5.18]{FF14}\label{L_i} Let $\zeta >0$ be fixed and let $ (T_i)$ be an increasing sequence of positive real numbers satisfying the condition 
\begin{equation}
\Sigma((T_i),\zeta) := \sum_{i\in \N}(\log T_i)^2(T_i)^{-\zeta}  < \infty.
\end{equation}
Let $\rho \in [0,1)$ with $\sum \rho_i =1$. Then the Lebesgue measure of the complement of the set $\G(w, (T_i), \zeta)$ of $(w,(T_i),\zeta)-$good points is bounded above. That is, $\exists K >0$ such that
$$\text{meas}( \G(w, (T_i), \zeta)^c) \leq K \Sigma((T_i),\zeta)[1/I(Y)]^a \HH(Y,\rho,\alpha)w.$$

\end{lemma}
\proof For all $i \in \N$ and for all $j = 0, 
\cdots, N_i$, let 
$$ \mathfrak{S}_{j,i} = \{z \in M : w_{\F_\alpha^{(T_{j,i})}}(z,1) < T_i^\zeta/w \}.$$
By definition we have
\begin{equation}\label{103}
 \G(w,(T_i),\zeta)^c = \bigcup_{i\in \N} \bigcup_{j=0}^{N_i}(\mathfrak{S}_{j,i} \cup \phi_{X_{\alpha}}^{-T_i}(\mathfrak{S}_{j,i})).
\end{equation}

By Lemma \ref{HH} for all $z \in  \mathfrak{S}_{j,i}$ we have
$$(I/2)^aT_i^\zeta/w < \int_{0}^1 H^1_{T_{j,i}}\circ \phi_{X_{\alpha}}^\tau(z)d\tau  = \frac{1}{T_{j,i}}\int_{0}^{T_{j,i}}H^1_{T_{j,i}}\circ \phi_{X_{\alpha}}^ \tau(z)d\tau.$$

It follows that 
$$\mathfrak{S}_{j,i} \subset \mathfrak{S}(j,i) := \left\{ z \in M : \sup_{J>0} \frac{1}{J}\int_{0}^JH^1_{T_{j,i}}\circ \phi_{X_{\alpha}}^\tau(z)d\tau > (I/2)^a T_i^\zeta/w  \right\}. $$

By the maximal ergodic theorem, the Lebesgue measure meas[$\mathfrak{S}_{j,i}$] of the set $\mathfrak{S}(j,i) $ satisfies the inequality
$$meas[\mathfrak{S}_{j,i}] \leq (2/I)^a (w/T_i^\zeta) \int_{M} H^1_{T_{j,i}}zdz.  $$
Let $\HH = \HH(Y,\rho,\nu)$ denote the constant defined in the formula (\ref{101}). By Theorem \ref{515}, since by hypothesis $\nu \leq 1/|\bar{\rho}|$ and $\alpha \in D_n(\bar{\rho}/|\bar{\rho}|,\nu)$, there exists a constant $K'(a,n,\nu) >0  $ such that the following bound holds:
$$\left|\int_{M} H^1_{T_{j,i}}(z)dz \right| \leq K' \HH(1+\log T_{j,i}). $$

Hence, by the definition of the $T_{j,i}$, we have 
\begin{equation}\label{104}
N_i \leq \log T_i \leq N_i +1, \ \ \log T_{j,i} \leq 2j. 
\end{equation}
 Thus, for some constant $K'' $, we have 
$$meas[\mathfrak{S}_{j,i}]  \leq K''(2/I)^a \HH w(1+j)T_i^{-\zeta}.$$

By (\ref{104}), for some constant $K''' >0$,
$$meas(\bigcup_{j=0}^{N_i} \mathfrak{S}_{j,i} \cup \phi_{X_{\alpha}}^{-T_i}(\mathfrak{S}_{j,i})) \leq K''' (2/I)^a\HH w(\log T_i)^2 T_i^{-\zeta}.  $$
By sub-additivity of the Lebesgue measure, we derive the bound 
$$meas(\bigcup_{i\in \N}\bigcup_{j=0}^{N_i} \mathfrak{S}_{j,i} \cup \phi_{X_{\alpha}}^{-T_i}(\mathfrak{S}_{j,i})) \leq K''' \Sigma((T_i),\zeta)\HH w.  $$
By formula (\ref{103}), the above estimate concludes the proof.
\qed

\section{Bounds on ergodic average}\label{sec;sec6}
We shall introduce assumptions on coadjoint orbits $\OO \subset \n^*$.

\begin{definition} A linear form $\Lambda \in \OO$ is \emph{integral} if the coefficients $\Lambda(\eta_i^{(m)}), (i,m) \in J $ are integer multiples of $2\pi$. Denote $\widehat M$ the set of coadjoint orbits of $\OO$ of integral linear forms $\Lambda$.
\end{definition}
There exist coadjoint orbits $\OO \subset \n^*$ that correspond to unitary representations which do not factor through the quotient $N/\exp{\n_k}$, $\n_k \subset Z(\n)$. Such coadjoint orbits and unitary representation are called \emph{maximal}. (See \cite[Lemma 2.3]{FF07})

\begin{definition}
Given a coadjoint orbit $\OO \in \widehat M_0$ and a linear functional $\Lambda \in \OO$, let us denote $\F_{\alpha,\Lambda}$ the completed basis $\F_{\alpha,\Lambda} = (X_\alpha,Y_\Lambda).$ For all $t \in \R$, we write scaled basis $\F_{\alpha,\Lambda}(t)$ by
$$\F_{\alpha,\Lambda}(t) = (X_\alpha(t),Y_\Lambda(t)) = A^\rho_t(X_\alpha,Y_\Lambda).$$
\end{definition}

Let $\widehat M_0$ be subset of all coadjoint orbits of forms $\Lambda$ such that $\Lambda(\eta_i^{(m)}) \neq 0$ for $m = k$. This space has maximal rank and $\Lambda(\eta_i^{(m)}) \neq 0, \forall (i,m) \in J $. For any $\OO \in \widehat M_0$, let $H_\OO$ denote the primary subspace of $L^2(M)$ which is a direct sum of sub-representations equivalent to $\Ind_{N'}^N(\Lambda)$.  For adapted basis $\F$, set 
$$W^r(H_\OO,\F) = H_\OO \cap W^r(M,\F).$$


\subsection{Coboundary estimates for rescaled basis}
 Recall definition of \emph{degree} of $Y_{i}^{(m)}$ and
$$d_{i}^{(m)} = 
\begin{cases}
k-m, & \text{ for all } 1 \leq m \leq k-1    \\
0 & \text{ $m = k$}.
\end{cases}
$$
For any linear functional $\Lambda$, the degree of the representation $\pi_{\Lambda}$ only depends on its coadjoint orbits.
We denote scaling vector $\rho \in ({\R^+})^J$  such that 
$$\sum_{(i,j) \in J} \rho_{i}^{(j)} = 1 \quad \text{and}\quad \rho_{i}^{(j)} = 0$$
for any $Y_{i}^{(j)}$ with $deg(Y) = 0$.

Assume that the number of basis of $\n$ with degree $k-m$ is $n_m$. Define
\begin{align}
S_{\n}(k) &:= (n_1-1)(k-1) + n_2 (k-2) + .... + n_{k-1}; \label{sn}\\
\delta(\rho) &:= \min_{\substack{1 \leq m \leq k-1 \\  1 \leq i,j \leq n_m }}\{\rho_{i}^{(m)}-\rho_{j}^{(m+1)}, \  \rho_{i}^{(m)}-\rho_{i}^{(m+1)} \}. \label{delta}
\end{align}

We have $\delta(\rho) \leq \lambda(\rho)$.
This inequality is strict unless one has homogeneous scaling
\begin{align*}
& \rho^{(j)}_i = \frac{d_j}{S_{\n}} \text{ for } j \leq k-1.
\end{align*} 

\begin{lemma}\label{47} 
There exists a constant $C>0$ such that, for all $r \in \R^+$ and for any function $f \in W^r(H_\OO) $, we have 

$$\sum_{(m,i) \in J}|[X_{\alpha}(t),Y^{(m)}_i(t)]f|_{r,\F_{\alpha,\Lambda}(t)} \leq Ce^{t(1-\delta(\rho))}|f|_{r+1,\F_{\alpha,\Lambda}(t)}. $$
\end{lemma}
\proof 
For all $(m,i) \in J$, we have 

$$[X_{\alpha}(t),Y_i^{(m)}(t)] = \sum_{l\geq 1} c_l^{(m+1)}e^{t(1-\rho_{i}^{(m)}+\rho_{l}^{(m+1)} )}Y_l^{(m+1)}(t).
$$
We note that $c^{(j)}_l = 0$ for $j = k$ and for some $l$, which is determined by commutation relation. 
Setting $C = \max_{(i,j)\in J^+ }\{|c_i^{(j)}|\} $, 
$$\sum_{(m,i) \in J}|[X_{\alpha}(t),Y^{(m)}_i(t)]f|_{r,\F_{\alpha,\Lambda}(t)} \leq Ce^{t(1-\delta(\rho))}\sum_{(m,i) \in J^+}|Y^{(m)}_i(t)f|_{r,\F_{\alpha,\Lambda}(t)}. $$
\qed

For $x \in M$, let $\gamma_x$ be the \textit{Birkhoff average operator} 
$$\gamma_x^T(f) = \frac{1}{T}\int_{0}^Tf\circ\phi_{X_{\alpha}}^t(x)dt.$$
Consider the decomposition of the restriction of the linear functional $\gamma_x$ to $W^r_0(H_\OO,\F_{\alpha,\Lambda}(t))$ as an orthogonal sum $\gamma_x = D(t)+R(t) \in W^{-r}_0(H_\OO,\F_{\alpha,\Lambda}(t))$ of $X_{\alpha}$-invariant distribution $D(t)$ and an orthogonal complement $R(t)$.

\begin{theorem}\label{48}Let $r> 2(k+1)(a/2+1)+1/2 $. For $g \in W^r(H_\OO,\F_{\alpha,\Lambda}(t))$ and for all $t \geq 0$, there exists a constant $C_r^{(1)}$ such that 
{
\begin{multline}
|R(t)(g)| \leq C_r^{(1)}e^{(1-\delta(\rho)-(1-\lambda))t}\max\{1,\delta_\OO^{-1}\} \\ 
\times T^{-1}\left(w_{\F_{\alpha,\Lambda}(t)}(x,1)^{-\frac{1}{2}}+w_{\F_{\alpha,\Lambda}(t)}(\phi_{X_{\alpha}}^T(x),1)^{-\frac{1}{2}} \right)|{g}|_{r,\F_{\alpha,\Lambda}(t)}.
\end{multline}}
\end{theorem}
\proof
Fix $t \geq 0$ and set $ D = D(t), R = R(t)$ for convenience. Let $g \in W_{\alpha,\Lambda}^r(H_\OO,\F(t))$. We write $g = g_D + g_R$, where $g_R$ is the kernel of $X_{\alpha}$-invariant distributions and $g_D$ is orthogonal to $g_R$ in $W_{\alpha,\Lambda}^r(H_\OO,\F(t))$. Then, $g_R$ is a coboundary and $R(g_D)=0$. Let $f = G^{X_\alpha(t)}_{X_\alpha,\Lambda}$. From $|D(g_R)|=0$, 
\begin{align}\label{RG1}
|R(g)| & = |R(g_D + g_R)| = |R(g_R)| = |\gamma_x(g_R) - D(g_R)| = |\gamma_x(g_R)|. 
\end{align}
By the Gottschalk-Hedlund argument,
\begin{align}\label{RG2}
\begin{split}
|\gamma_x(g_R)| & = \left|\frac{1}{T}\int_0^Tg\circ\phi_{X_{\alpha}}^s(x)ds\right|\\
& = \frac{1}{T}\left|f\circ\phi^T_{X_{\alpha}}(x) - f(x)\right|\\
& \leq \frac{1}{T}(|f(x)| + |f\circ\phi^T_{X_{\alpha}}(x)|).
\end{split}
\end{align}

By Theorem \ref{sobolev} and Lemma \ref{47},  for any $\tau > a/2+1$, there exists a positive constant $C_r$ such that for any $z \in M$
\begin{equation}\label{RG3}
|f(z)| \leq \frac{C_r}{w_{\F_{\alpha,\Lambda}(t)}(z,1)^\frac{1}{2}}\left(Ce^{(1-\delta(\rho))t}|f|_{\tau,\F_{\alpha,\Lambda}(t)}+|g|_{\tau-1,\F_{\alpha,\Lambda}(t)} \right).
\end{equation}
By Theorem \ref{reestimate}, if $r >  2(k+1)\tau +1/2$, then 
{
$$|f|_{\tau,\F_{\alpha,\Lambda}(t)} \leq C_{r,k,\tau} e^{-(1-\lambda)t} \max\{1, \delta_\OO^{-1}\}|g_R|_{r,\F_{\alpha,\Lambda}(t)}.$$} 
By orthogonality, we have $|g_R|_{r,\F_{\alpha,\Lambda}(t)} \leq |g|_{r,\F_{\alpha,\Lambda}(t)}$.

\qed

\begin{corollary} \label{6bound} 
For every $r> 2(k+1)(a/2+1)+1/2 $, there is a constant $C_r^{(2)}$ such that the following holds for every $\OO \in \widehat I_0$ and every $x \in M$. Then,
$$|R|_{-r, \F_{\alpha,\Lambda}}  \leq C_r^{(2)} [1/I(Y_\Lambda)]^{a/2}\max\{1,\delta_\OO^{-1} \}T^{-1}.$$
\end{corollary} 
\proof
For all $x \in M$, we have 
$$w_{\F_{\alpha,\Lambda}}(x,1) \geq \left(\frac{I(Y_\Lambda)}{2}\right)^a.$$

It follows from Theorem \ref{48} applied to the orthogonal decomposition of $\gamma_x = D(0)+R(0)$.
\qed

\subsection{Bounds on ergodic averages in an irreducible subrepresentation.} \label{sec;irreducible}
In this section, we derive the bounds on ergodic averages of nilflows for function in a single irreducible sub-representation.

For brevity, let us set
\begin{equation}\label{106}
C_r(\OO) = (1+\delta_\OO^{-1}).
\end{equation}


\begin{proposition}\label{T_i} Let $r> 2(k+1)(a/2+1)+1/2 $. Let $(T_i)$ be an increasing sequence of positive real numbers  $ \geq 1$ and let $0 < w < I(Y)^a$. Let $\zeta>0$. There exists a constant $C_r(\rho)$ such that for every $\G(w, (T_i), \zeta)$-good points $x \in M$ and all $f \in W^r(H_\OO,\F)$,  we have
\begin{equation}\label{estimatet_i}
\left|\frac{1}{T_i}\int_{0}^{T_i}f\circ \phi_{X_{\alpha}}^t(x)dt \right| \leq C_{r}(\rho)C_r(\OO)w^{-1/2}{T_i}^{-\delta(\rho)+\zeta+\lambda/2} |{f}|_{r,\F_\alpha}.
\end{equation}

\end{proposition}
\proof
By group action of scaling \eqref{eqn;reno}, a sequence of frame is chosen $\F(t_j) = A_{\rho}^{t_j}\F $ with other scaling factors $\rho_i t_j$ on elements of Lie algebras $Y_i$. Then, as $j$ increases from 0 to $N$, the scaling parameter $t_j$ becomes larger, while the scaled length of the arc becomes shorter approaching to 1. Let $\phi_{X_j}^s(x)$ denote the flow of the scaled vector field $e^{t_j}X = X(t_j)$.

For each $j = 0,\cdots, N$, let $\gamma = D_j + R_j$ be the orthogonal decomposition of $\gamma$ in the Hilbert space $W^{-r}(H_\pi,\F(t_j))$ into $X_{\alpha}$-invariant distribution $D_j$ and an orthogonal complement $R_j$. For convenience, we denote by $|\cdot|_{r,j}$ and $\norm{\cdot}_{r,j}$ respectively, the transversal Sobolev norm $|\cdot|_{r,\F(t_j)}$ and Lyapunov Sobolev norm $\norm{\cdot}_{r,\F(t_j)}$ relative to the rescaled basis $\F(t_j)$. 

Let us set $N_i = [\log T_i]$ and $t_{j,i} := T_{j,i} = \log T_i^{j/N_i}$ for integer $j \in [0,N_i]$. We observe $N_i < \log T_i < N_i + 1$.
For simplicity, we will omit index $i \in \N$ and set $T = T_i$, $N = N_i$ for a while within the proof and lemmas of this subsection.

Our goal is to estimate $|\gamma|_{-r,\F_\alpha} = |\gamma|_{-r,0}$ (the norm of distribution of unscaled basis). By triangle inequality and Corollary \ref{6bound},
\begin{align}\label{dds1}
\begin{split}
|\gamma|_{-r,0} & \leq  |D_0|_{-r,0} + |R_0|_{-r,0} \\
& \leq |D_0|_{-r,0} + C_r^{(2)} [1/I(Y)]^{a/2}C_r(\Lambda_\OO)T^{-1}.
\end{split}
\end{align}

We now estimate $|D_0|_{-r,0}$. By definition of the Lyapunov norm and its bound (\ref{lya}), for $-s < -r<0$ 
\begin{equation}
|D_0|_{-s,0} \leq C_{r,s}\norm{D_0}_{-r,0}.
\end{equation}

Since $D_j + R_j = D_{j-1}+R_{j-1}$, observe  $D_{j-1} = D_j + R'_j$, where $R'_j$ denotes the orthogonal projection of $R_j$, in the space $W^{-r}(H_\OO,\F(t_{j-1}))$, on the space of invariant distribution. By definition of Lyapunov norm,
\begin{align*} 
\norm{D_{j-1}}_{-r,{j-1}} & \leq \norm{D_j}_{-r,{j-1}} + \norm{R'_j}_{-r,{j-1}} \\
& \leq \norm{D_j}_{-r,{j-1}} + |{R'_j}|_{-r,{j-1}} \\
& \leq \norm{D_j}_{-r,{j-1}} + |{R_j}|_{-r,{j-1}}.
\end{align*}

By Lemma \ref{lem69}, equivalence of norm gives
\begin{equation}\label{DDs}
\norm{D_{j-1}}_{-r,{j-1}} \leq \norm{D_{j}}_{-r,{j-1}} + C|R_j|_{-r,j}.
\end{equation}

By Lemma \ref{lyapp}, for any $X_{\alpha}$-invariant distribution $D$ and for all $t_j \geq t_{j-1}$,
$$\norm{D}_{-r,\F(t_{j-1})} \leq e^{-\lambda(\rho)(t_j-t_{j-1})/2}\norm{D}_{-r,\F(t_j)}. $$
Since $\F(t_j) = A_{\rho}^{t_j-t_{j-1}}\F(t_{j-1}) $ and $t_{j} - t_{j-1} = \log T/N$ implies

$$\norm{D_j}_{-r,j-1} \leq T^{-\lambda(\rho)/2N} \norm{D_j}_{-r,j}. $$

From (\ref{DDs}) we conclude by induction
\begin{align}\label{DD67}
 \norm{D_0}_{-r,0} &\leq T^{-\lambda(\rho)/2}\left( \norm{D_N}_{-r,N}+ C\sum_{l=0}^{N-1}T^{(l+1)\lambda(\rho)/2N}|R_{N-l}|_{-r,N-l} \right).
\end{align}
By Lemma \ref{79} and \ref{710},
{
\begin{align} \norm{D_0}_{-r,0} & \leq  C^{1}_r(\rho)C_r(\OO)w^{-1/2} T^{1-\delta(\rho)+\zeta/2-(1-\lambda)-\lambda(\rho)/2}.
\end{align}
}
From (\ref{dds1}) and the above,  we conclude that there exists a constant $C_r(\rho)$ such that 
$$|\gamma|_{-r,\F} \leq C_r(\rho)C_r(\OO)w^{-1/2}T^{-\delta(\rho)+\zeta/2+\lambda/2}.$$
\qed

Here we introduce the proof of supplementary lemmas.

\begin{lemma}\label{79} For any $r>a/2$, there exists a constant $C_r>0$ such that for all good points $x \in \G(w, (T_i), \zeta)$, we have 
$$\norm{D_N}_{-r,N} \leq C_rT^{{\zeta}/{2}}/w^{1/2}. $$
\end{lemma}
\begin{proof}
By definition of norm, 
$$\norm{D_N}_{-r,N} \leq |D_N|_{-r,N} \leq |\gamma|_{-r,N}.$$ 
It suffices to find the bound of orbit segment with respect to rescaled bases.

For all $i \in \N$, set $t_j = t_{j,i}$. By Definition \ref{good}, for $x \in \G(w, (T_i), \zeta)$ and $y_i = \phi_{X_{\alpha}}^{T_i}(x)$ 
\begin{equation}\label{gp}
\frac{1}{w_{\F_\alpha^{(t_j)}}(x,1)} \leq T_i^\zeta/w \ \text{ and } \ \frac{1}{w_{\F_\alpha^{(t_j)}}(y_i,1)} \leq T_i^\zeta/w.
\end{equation}

Note that the orbit segment $(\phi^t_{{X_{\alpha}}}(x))_{0\leq t \leq T} $ coincides with the orbit segment $(\phi^\tau_{{X_{\alpha}}(t_N)}(x))_{0\leq \tau \leq 1} $ of length 1 since ${X_{\alpha}}(t_N) = {X_{\alpha}}(\log T) = T{X_{\alpha}}$. Then by Theorem \ref{sobolev}, 
 $$|\gamma|_{-r,N} \leq  C_r w_{\F_\alpha^{(t_N)}}(x,1)^{-1/2}.$$ 
Therefore, by the inequality (\ref{gp}), 
$$w_{\F_\alpha^{(t_N)}}(x,1)^{-1/2} \leq T^{{\zeta}/{2}}/w^{1/2}.$$
\end{proof}

\begin{lemma}\label{710} For every $r> 2(k+1)(a/2+1)+1/2 $, there is a constant $C_r(\rho)>0$ such that for every good point $x \in \G(w, (T_i), \zeta)$, we have
\begin{equation}
\sum_{l=0}^{N-1}T^{(l+1)\rho_Y/2N}|R_{N-l}|_{-r,N-l} \\
\leq C^{(1)}_r(\rho)C_r(\OO)w^{-1/2}T^{1-\delta(\rho)-(1-\lambda)+\zeta/2}.
\end{equation}
\end{lemma}
\begin{proof}
  The orbit segment $(\phi^t_{{X_{\alpha}}}(x))_{0\leq t \leq T} $ has length $T^{l/N}$ with respect to the generator ${X_{\alpha}}(t_{N-l}) = {X_{\alpha}}((1-l/N)\log T) = T^{1-l/N}{X_{\alpha}}$. Thus, by Theorem \ref{48} with $e^{(1-\delta(\rho))t_{N-l}} = T^{(1-l/N)(1-\delta(\rho))}$. Then,
\begin{align*} |R_{N-l}|_{-r,N-l} &\leq C^{(1)}_rC_r(\OO)T^{(1-l/N)(1-\delta(\rho)-(1-\lambda))-l/N}) \\
& \times \left(\frac{1}{w_{\F_\alpha^{(t_{N-l})}}(x,1)^\frac{1}{2}} + \frac{1}{w_{\F_\alpha^{(t_{N-l})}}(y,1)^\frac{1}{2}} \right)\\
& \leq 2C^{(1)}_rC_r(\OO) w^{-1/2}T^{(1-l/N){(\lambda-\delta(\rho)})-l/N+\zeta/2}.
\end{align*}
Let $C = 2C^{(1)}_rC_r(\OO) w^{-1/2}$. Since $N_i = [\log T_i]$ and $N_i \leq \log T_i \leq N_i+1$, we have ${T_i}^{1/{(N_i+1)}} \leq e \leq {T_i}^{1/{N_i}}.$ By setting $T = T_i$, 
\begin{align*}
&\sum_{l=0}^{N-1}T^{(l+1)\rho_Y/2N}|R_{N-l}|_{-r,N-l} \\
&\leq CT^{1-\delta(\rho)-(1-\lambda)+\zeta/2} \sum_{l=0}^{N-1}T^{(l+1)\rho_Y/2N} T^{-l/N(1-\delta(\rho)-(1-\lambda))-l/N}\\
&\leq CT^{1-\delta(\rho)-(1-\lambda)+\zeta/2+\rho_Y/2N} \sum_{l=0}^{N-1}T^{-l/N(2-\delta(\rho)-(1-\lambda)-\rho_Y/2)}\\
& \leq e^rCT^{1-\delta(\rho)-(1-\lambda)+\zeta/2}\sum_{l=0}^\infty e^{-l(1+\lambda-\delta(\rho)-\rho_Y/2)}.
\end{align*}
By  \eqref{delta}, we have $1+\lambda-\delta(\rho)-\rho_Y/2 \geq 1-\rho_Y/2 >1/2$, thus geometric series converges.
\end{proof}

\begin{lemma}\label{lem69} There exists a constant $C:= C(r)>0$ such that, for all $j = 0, \cdots, N,$
$$C^{-1}|\cdot|_{-r,j} \leq |\cdot|_{-r,j-1} \leq C|\cdot|_{-r,j}.$$
\end{lemma}
\begin{proof} From (\ref{104}), $t_j-t_{j-1} \leq 2$ and observe $\F(t_j) = A^{t_j - t_{j-1}}\F(t_{j-1})$. Passing from the frame $\F(t_{j-1})$ to $\F(t)$, it can be verified that distortion of the corresponding transversal Sobolev norm is uniformly bounded.
\end{proof}

Let 
$$\widetilde M_0 = \bigcup_{\OO \in \hat{M}_0} \{\Lambda \in \OO \mid \Lambda \text{ integral} \}$$
be collection of maximal integral coadjoint orbits.


\begin{remark}
Let $\sigma = (\sigma_1, \cdots, \sigma_n) \in (0,1)^n$ be such that $\sigma_1+ \cdots + \sigma_n = 1$. For simplicity, we choose $\sigma_i = 1/n$ from now on. (See Definition \ref{simultdio} or Lemma \ref{512}) 
\end{remark}

\begin{theorem} \label{711}
For any $\Lambda \in \widetilde M_0$, let 
 $\nu \in [1, 1 + (k/2-1)\frac{1}{n} ]$.  Then, for any $r > (k+1)(a/2+1)+1/2$, there exists a constant $C(\sigma, \nu)$ satisfying the following. For every $\epsilon > 0$ there exists a constant $K_\epsilon(\sigma, \nu)>0$ such that,  for every $\alpha_1 = (\alpha_1^{(1)},\cdots,\alpha^{(1)}_n)  \in D_n(\sigma, \nu)$ and for every $w \in (0, I(Y)^a]$ there exists a measurable set $\G_{\Lambda}(\sigma,\epsilon,w)$ satisfying the estimate
\begin{equation}\label{meas2}
meas ( \G_{\Lambda}(\sigma,\epsilon,w)^c ) \leq K_\epsilon(\sigma, \nu)\left(\frac{w}{I(Y_\Lambda)^a}\right)\HH(Y_\Lambda,\rho,\alpha).
\end{equation}

For every $x \in \G_{\Lambda}(\sigma,\epsilon,w)$,  for every $f \in W^r(H_\OO,\F)$ and $T \geq 1$ we have
\begin{equation*}
\left|\frac{1}{T}\int_{0}^Tf\circ \phi_{X_{\alpha}}^t(x)dt \right| \leq  \frac{C_r(\sigma, \nu) C_r(\Lambda) }{w^\frac{1}{2}} T^{-(1-\epsilon)\frac{1}{3S_{\n}(k)}} |f|_{r,\F_{\alpha,\Lambda}}.
\end{equation*}
\end{theorem}
\proof 
If the coadjoint orbit $\OO$ is integral and maximal with full rank, then we can see that the optimal exponent will be attained by the following scaling. Let $\rho = (\rho_i^{(m)})$ be the vector given by homogeneous scaling:
\begin{align*}
& \rho^{(j)}_i = \frac{d_j}{S_{\n}} \text{ for } i \leq k.
\end{align*}

Let us set $\zeta = 2\delta(\rho)/3 -{\lambda}/{3}$. Given $\epsilon >0$ and for all $i \in \N$, set $T_i = i^{(1+\epsilon)\zeta^{-1}} $. Then there exists a constant $K_\epsilon(\rho)>0$ such that 
$$\Sigma(w,(T_i),\zeta) \leq \sum_i(\log T_i)^2T_i^{-\zeta}   \leq K_\epsilon{(\rho)}.$$

Let $\G = \G_{\Lambda}(\sigma, \epsilon, w) = \G(w,(T_i),\zeta)$ be the set of $(w,(T_i),\zeta)$-good points for the basis $\F_{\alpha}$. The estimate in the formula (\ref{meas2}) follows from the Lemma \ref{L_i} and definition of good points. By Proposition \ref{T_i}, for all $x \in \G$ and for every $f \in W^r(H_\OO,\F)$, the estimate (\ref{estimatet_i}) holds true. Given $T \in [T_i, T_{i+1}]$,
$$\int_{0}^{T}f\circ \phi_{X_{\alpha}}^t(x)dt = \int_{0}^{T_i}f\circ \phi_{X_{\alpha}}^tdt + \int_{T_i}^{T}f\circ \phi_{X_{\alpha}}^t(x) dt = (I) + (II).$$ 
Let $C  = C_{r}(\rho)C_r(\OO) /w^{1/2}$. The first term is estimated by the formula (\ref{estimatet_i}):
\begin{align*}
(I) & \leq C{T_i}^{1-\delta(\rho)+\zeta/2+\lambda/2}|{f}|_{r,\F_{\alpha,\Lambda}}  = C{T_i}^{1-2\delta(\rho)/3+\lambda/3}|{f}|_{r,\F_{\alpha,\Lambda}}.
\end{align*}

For the second term, let us set $\gamma = (1+\epsilon)\zeta^{-1}$ and observe that $\gamma^{-1} = \zeta(1+\epsilon)^{-1} \geq (1-\epsilon)\zeta.$ We have
\begin{align*}
(II) \leq (T-T_i)\norm{f}_\infty & \leq \beta 2^{\gamma - 1 } T^{1-\gamma^{-1}}\norm{f}_{\infty} \\
& \leq C'(\rho)T^{1-(1-\epsilon)(-2\delta(\rho)/3+\lambda/3)} |{f}|_{r,\F_{\alpha,\Lambda}}.
\end{align*}

By the estimates on the terms (I) and (II), the proof is complete. 

\qed

\begin{remark}
If $\OO$ is integral but not maximal, then the restriction of $\Lambda$ factors through an irreducible representation of the $k-1$ step nilpotent group $N/\exp{n_k'}$. Then, $\n/ \n_k$ is polarizing subalgebra for subrepresentation and it reduces to the case of maximal integral. Since the growth rate is determined by the scaling factors and  the exponent $\lambda$ is determined by the step size and number of elements, the highest exponent is obtained by integral maximal full rank case.
\end{remark}


\subsection{General bounds on ergodic averages.}
Finally, in order to solve cohomological equation on nilmanifold, we glue the solutions constructed in every irreducible sub-representation of $N$. The main idea is to increase extra regularity of the Sobolev norm to obtain the estimates that are uniformly bounded across all irreducible subrepresentation.

\begin{definition} For every $\OO \in \widehat M_0$, we define $|\OO| = \max_{\eta_i \in \n_k}|\Lambda(\eta_i^{(k)})|. $
\end{definition}
Note that $|\OO|$ does not depend on the choice of $\Lambda$ and $|\OO| \neq 0$ by maximality.
We specifically choose an element $\eta_{*}^{(k)}$ whose degree $k$ such that 
$$|\OO| = |\Lambda(\eta_{*}^{(k)})|.$$

\begin{lemma}\label{613} For every $\OO \in \widehat M_0$ and for every $\Lambda \in \OO$, we have
$$I(Y_\Lambda)^{-a}\HH(Y_\Lambda) \leq C(\alpha_1)(1+\log C(\alpha_1))2^{a+1}.  $$
\end{lemma}
\proof
The return time of the flow $X_\alpha$ to any orbit of the codimension one subgroup $N' \subset N$ is 1. Hence, by Definition \ref{I(Y)}, we have $I(Y_\Lambda) = 1/2$ for the basis. By (\ref{100}), we have  and $C(\alpha_1)\geq1$. Then, from the definition of  the constant $H(Y,\rho,\alpha)$,  we obtain
\begin{align*}
I(Y_\Lambda)^{-a}\HH(Y) & \leq I(Y_\Lambda)^{-a}+I(Y)^{-n}C(\alpha_1)\left(1+\log^{+}[I(Y)^{-1}]+\log C(\alpha_1)\right) \\ 
& \leq C(\alpha_1)(1+\log C(\alpha_1)) \left(I(Y_\Lambda)^{-a}+I(Y_\Lambda)^{-n}\log^{+}[I(Y)^{-1}]\right)\\
& \leq 2C(\alpha_1)(1+\log C(\alpha_1))I(Y_\Lambda)^{-a} .
\end{align*}
\qed

\begin{corollary}\label{614} For every $\OO \in \widehat M_0, \Lambda \in \OO, w>0$ and $\epsilon >0$, let 
\begin{equation}\label{614def}
w_{\Lambda} = w|\Lambda(\F)|^{-2a-\epsilon}.
\end{equation}
Then, for every $w>0$ and $\epsilon>0$ the set
$$\G(\sigma,\epsilon,w) = \bigcap_{\Lambda \in \widehat M_0} \G_{\Lambda}(\sigma,\epsilon,w_\Lambda)$$
has measure greater than $1-Cw\epsilon^{-1},$ with $C = 2^{-a+1}K_\epsilon(\sigma,\nu)C(\alpha_1)(1+\log C(\alpha_1))$. Furthermore, if $\epsilon'<\epsilon$ we have $\G(\sigma,\epsilon,w) \subset \G(\sigma,\epsilon',w)$.
\end{corollary}

\proof Recall that $|\Lambda(\F)|$ is  integral multiples of $2\pi$. By Lemma  \ref{613}, inequality (\ref{meas2}) and definition of $w_{\Lambda}$, we have
\begin{align*}
meas (\G_\Lambda(\sigma,\epsilon,w_\Lambda)^c ) &\leq K_\epsilon(\sigma, \nu)(\frac{w_\Lambda}{I(Y)^a})\HH(Y_\Lambda,\rho,\alpha) \\
& \leq C'|\Lambda(\F)|^{-2a-\epsilon}w,   
\end{align*}
where $C' = 2^{a+1}K_\epsilon(\sigma,\nu)C(\alpha_1)(1+\log C(\alpha_1))$. Since the $|\Lambda(\F)| = 2\pi l $ is bounded by $(2l)^{a-1}$, 
\begin{align*} \sum_{\Lambda \in \tilde{M_0}} meas (\G_\Lambda(\sigma,\epsilon,w_\Lambda)^c ) &  \leq 2^{-2a}wC'\sum_{l>0} \sum_{\Lambda \in \tilde{M_0} : |\Lambda| = 2\pi l} l^{-a-\epsilon} \\
& \leq Cw\sum_{l>0} l^{-1-\epsilon} < Cw\epsilon^{-1}.
\end{align*}

The last statement on the monotonicity of the set follows from the analogous statement in Theorem \ref{711}.
\qed
\medskip

In every coadjoint orbit, we will make a particular choice of a linear form to accomplish the estimates of the bound for each irreducible sub-representation in terms of higher norms.

\begin{definition}  For every $\OO \in \widehat M_0$, we define $\Lambda_{\OO}$ as the unique integral linear form $\Lambda \in \OO$ such that 
$$0 \leq  \Lambda(\eta^{(k-1)}_{*}) < |\OO|. $$ 

The existence and uniqueness of $\Lambda_{\OO}$ follows from 
$$\Lambda \circ \Ad(\exp(tX_{\alpha}))(\eta^{(k-1)}_{*}) = \Lambda(\eta^{(k-1)}_{*})+t|\OO|, $$
and the form $\Lambda \circ \Ad(\exp(tX_{\alpha})$ is integral for all integer values of $t \in \R$. 
\end{definition}

\begin{lemma}\label{616} There exists a constant $C(\Gamma)>0$ such that the following holds on the primary subspace $C^{\infty}(H_\OO)$ the following holds:
$$|\Lambda_{\OO}(\F)|\Id \leq C(\Gamma)(1+\Delta_{\F})^{k/2}. $$
\end{lemma}
\proof
Let $x_0 = -\Lambda_{\OO}(\eta_{*}^{(k-1)})/|\OO|$. Then there exists a unique $\Lambda' \in \OO$ such that $\Lambda'(\eta^{(k-1)}_{*}) = 0$ given by $\Lambda' = \Lambda \circ \Ad(e^{x_0X_{\alpha}})$. The element $W \in \II$ is represented in the representation as
multiplication operators by the polynomials  
\begin{equation}\label{rep;lambda}
P(\Lambda, W)(x) = \Lambda(\Ad(e^{xX_{\alpha}})W).
\end{equation}

By the definition of the linear form, the identity $[X_{\alpha}, \eta_{*}^{(k-1)}] = \eta_{*}^{(k)}$ implies 
$$P(\Lambda', \eta_{*}^{(k-1)})(x) = |\OO|x.$$

From \eqref{rep;lambda}, we have
$$\sum_{j} \frac{(-x)^j}{j!}P(\Lambda', \ad(X_{\alpha})^jW) = \Lambda'(W), \quad \text{ for all } W \in \II.$$
Then we obtain
\begin{align*}
\Lambda'(W) & = \sum_{j} \frac{(-x)^j}{j!}P(\Lambda', \ad(X_{\alpha})^jW)  \\
& = \sum_{j} \frac{(-1)^j}{j!}\left(\frac{P(\Lambda', \eta^{(k-1)}_{*})}{|\OO|}\right)^jP(\Lambda', \ad(X_{\alpha})^jW)  \\
& = |\OO|^{1-k}\sum_{j} \frac{(-1)^j}{j!}P(\Lambda', \eta^{(k-1)}_{*})P(\Lambda', \eta^{(k)}_{*})^{k-1-j}P(\Lambda', \ad(X_{\alpha})^jW).  
\end{align*}

For any $\Lambda \in \n^*$ the transversal Laplacian for a basis $\F$ in the representation $\pi_\Lambda$ is the operator of multiplication by the polynomial and derivative operators

$$\Delta_{\Lambda, \F} = \sum_{W \in \F}  \pi^{X_{\alpha}}_\Lambda(W)^2 = \sum_{W \in \F}P(\Lambda,W)^2.$$
Hence,  
$$|P(\Lambda',\eta_{j}^{(m)} )|\leq (1+\Delta_{\Lambda', \F})^{1/2}.$$

By above identity in formula, the constant operators $\Lambda'(\eta_{j}^{(m)})$ are given by polynomial and derivative expressions of degree $k$ in the operators $P(\Lambda',\eta_{j}^{(m)})$ we obtain the estimate
$$|\Lambda'(\F)|\Id \leq C_1(\Gamma)(1+\Delta_{\Lambda', \F})^{k/2}.$$
Since the representation $\pi_{\Lambda'}$ and $\pi_{\Lambda_\OO}$ are unitarily intertwined by the translation operator by $x_0$, and since constant operators commute with translations, we also have
$$|\Lambda'(\F)|\Id \leq C_1(\Gamma)(1+\Delta_{\Lambda_\OO, \F})^{k/2}.$$

Since $x_0$ is bounded by a constant depending only step size $k$, the norms of the linear maps $\Ad(\exp(\pm x_0X_{\alpha}))$ are bounded by a constant depending only on $k$. Therefore, $|\Lambda_\OO(\F)| \leq C_2(k)|\Lambda'(\F)|$ and the statements of the lemma follows.
\qed

\begin{corollary}\label{617}
There exists a constant $C'(\Gamma)$ such that for all $\OO \in \widehat M_0$ and for any sufficiently smooth function $f \in H_{\OO}$,
$$C_r(\Lambda_\OO)|{f}|_{r,\F_{\alpha,\Lambda}} \leq C'(\Gamma)w^{-1/2}|{f}|_{r+l,\F_{\alpha}} $$
where 
$l = (kr+1)k/2 +ak$. 
\end{corollary}
\begin{proof}
From the definition (\ref{106}) we have $C_r(\Lambda_\OO) = (1+|\Lambda_{\OO}(\F)|)^{l_1}$ with $l_1 = kr+1$.
By the formula (\ref{614def}), 
$$C_r(\Lambda_\OO)w_\Lambda^{-1/2} \leq  w^{-1/2}(1+|\Lambda_{\OO}(\F)|)^{l_2}$$
with $l_2 = l_1 + 2a$. By Lemma \ref{616} we have 
$$(1+|\Lambda_{\OO}(\F)|)^{l_2} \leq C'(\Gamma)(1+\Delta_\F)^{l_2k/2}.$$
\end{proof}

\begin{proposition}\label{esttt} Let $r > (k+1)(3a/4 +2)+1/2$. Let $\sigma = (1/n,
\cdots, 1/n) \in (0,1)^n$ be a positive vector. Let us assume that $\nu \in [1,1+ (k/2-1)\frac{1}{n}]$ and let $\alpha \in D_n(\sigma,\nu)$. For every $\epsilon >0 $ and $w >0$, there exists a measurable set $\G(\sigma,\epsilon,w)$ satisfying
$$meas (\G(\sigma,\epsilon,w)^c ) \leq Cw\epsilon^{-1} \text{ with } C = 2^{-a+1}K_\epsilon(\sigma,\nu)C(\alpha_1)(1+\log C(\alpha_1)), $$
such that for every $x \in  \G(\sigma,\epsilon,w)$,  $f \in W^r(M)$ and any $T \geq 1$ we have 
\begin{equation}
\left|\frac{1}{T}\int_{0}^{T}f\circ \phi_{X_{\alpha}}^t(x)dt \right| \leq Cw^{-1/2}T^{-(1-\epsilon)\frac{1}{3S_{\n}(k)}} |{f}|_{r,\F_{\alpha}}.
\end{equation}
\end{proposition}
\proof
Let $\tau := r - ak/2 > (a+2)(k+1)+1/2 $. Let $f \in W^\tau(M,\F)$ and let $f = \sum_{\OO \in \widehat M_0}f_{\OO}$ be its orthogonal decomposition onto the primary subspace $H_{\OO}$. For each $\OO \in \widehat M_0$, the constant $w_{\OO}$ is given and the set 
$$\G(\sigma,\epsilon,w) = \bigcap_{\OO \in \widehat M_0} \G_{\Lambda}(\sigma,\epsilon,w_{\OO})$$ has measure greater than $1-Cw\epsilon^{-1}$ as proved in Corollary \ref{614}.

If $x \in \G(\sigma,\epsilon,w)$, then by Theorem \ref{711} and Corollary \ref{617}, the following estimate holds true for every $\OO \in \widehat M_0$ and all $T \geq 1$:
\begin{equation*}
\left|\frac{1}{T}\int_{0}^{T}f_{\OO}\circ \phi_{X_{\alpha}}^t(x)dt \right| \leq C_r(\sigma,\nu)w^{-1/2}{T}^{-(1-\epsilon)\frac{1}{3S_{\n}(k)}} |f_{\OO}|_{r,\F_\alpha}.
\end{equation*}

For any $\tau >0$ and any $\epsilon'>0$, by Lemma \ref{616} and orthogonal splitting of $H_\OO$ we have
\begin{align*}
\left|\sum_{\OO \in \widehat M_0}|f_{\OO}|_\tau \right|^2 &\leq \sum_{\OO \in \widehat M_0}(1+|\Lambda_{\OO}(\F_\alpha)|)^{-a-\epsilon'} \sum_{\OO \in \widehat M_0}(1+|\Lambda_{\OO}(\F_\alpha)|)^{a+\epsilon'}|f_{\OO}|_{\tau, \F_\alpha}^2\\
& \leq C(a) |f|_{\tau+(a+\epsilon')k/2, \F_\alpha}^2.
\end{align*}
and the theorem follows after renaming the constant.
\qed
\medskip

\textit{Proof of Theorem \ref{main}.} Under same hypothesis of Proposition \ref{esttt}, for $i \in \N$ let $w_i = 1/2^iC$ and $\G_i = \G(\sigma,\epsilon,w_i)$. Set $K_{\epsilon}(x) = 1/{w_i}^{1/2}$ if $x \in \G_i \backslash \G_{i-1}$. By Proposition \ref{esttt}, the set $\G_i$ are increasing and satisfy $meas( \G_i^c) \leq 1/2^i\epsilon$. Hence, the set $\G(\sigma,\epsilon) = \bigcup_{i \in \N}\G_i$ has full measure and the function $K_{\epsilon}$ is in $L^p(M)$ for every $p \in [1,2)$.\qed
\medskip

\textit{Proof of Corollary \ref{1.2}.}  For step-$k$ strictly triangular nilpotent Lie algebra $\n$ has  dimension $\frac{1}{2}k(k+1)$ with 1 dimensional center. If the coadjoint orbit $\OO$ is integral and maximal, then the optimal exponent will be attained by the formula (\ref{sn}). Let $\rho = (\cdots, \rho_i^{(m)},\cdots)$ be the rescaling factor  given :

$$S_{\n}(k) = [(k-1)^2+\sum_{n=1}^{k-2}n(n+1)] = (k-1)(k^2+k-3),$$
and we choose homogeneous scaling
\begin{align*}
& \rho^{(j)}_i = \frac{d_i^{(j)}}{S_{\n}(k)} =  \frac{k-j}{(k-1)(k^2+k-3)} \text{ for } i \leq j.
\end{align*}

Then, we can verify that 
\begin{equation}
\lambda(\rho) = \delta(\rho) = \frac{1}{(k-1)(k^2+k-3)}.
\end{equation}
By inductive argument with rescaling again, the exponent is obtained which proves Corollary \ref{1.2}. 
\qed

\section{Uniform bound of the average width of step 3 case.}\label{7}
In this section we prove Theorem \ref{step3}, uniform bound of effective equidistribution of nilflow on strictly triangular step 3 nilmanifold. On its structure, it is possible to derive uniform bound under Roth-type Diophantine condition due to the \emph{linear divergence} of nearby orbits. This argument is based on counting principles of close return times which substitute the necessity of good point. (See also \cite{F16} for 3-step filiform case.)

\subsection {Average Width Function} 
Let $N$ be a step 3 nilpotent Lie group on 3 generators introduced in (\ref{triaaa}). We denote its Lie algebra $\n $ with its basis $\{X_1,X_2,X_3,Y_1,Y_2,Z\}$ satisfying following commutation relations
\begin{equation}\label{commu}
[X_1,X_2] = Y_1, \ [X_2,X_3] = Y_2, \ [X_1,Y_2] = [Y_1,X_3] = Z.
\end{equation}
As introduced in section 2, $ \{\phi^t_V\}_{t\in \R}$ is a measure preserving flow generated by $V:=X_1+\alpha X_2+\beta X_3$ and $(1,\alpha)$ satisfies standard simultaneous Diophantine condition (\ref{simultdio}).

By definition of the average width (see Definition \ref{def;averagewidth}), for any $t \geq 0$ and for any $(x,T) \in M \times [1,+\infty)$ we will construct an open set $\Omega_{t}(x,T) \subset \R^6$ which contains the segment $\{(s,0,\cdots,0)\mid 0 \leq s \leq T\}$ such that the map 
\begin{multline*}
\phi_x(s,x_2,x_3,y_1,y_2,z) \\
 = \Gamma x\exp(se^tV)\exp(e^{-\frac{1}{3}t}x_2X_2+e^{-\frac{1}{3}t}x_3X_3+e^{-\frac{1}{6}t}y_1Y_1+e^{-\frac{1}{6}t}y_2Y_2+zZ)
\end{multline*}
is injective on $\Omega_{t}(x,T)$. Injectivity fails if and only if there exist vectors 
$$(s,x_2,x_3,y_1,y_2,z) \neq (s',x'_2,x'_3,y'_1,y'_2,z')$$
such that  
\begin{align}\label{equality}
&\Gamma x \exp(s'e^tV)\exp(e^{-\frac{1}{3}t}x'_2X_2+e^{-\frac{1}{3}t}x'_3X_3+e^{-\frac{1}{6}t}y'_1Y_1+e^{-\frac{1}{6}t}y'_2Y_2+z'Z)\\
& = \Gamma x \exp(se^tV)\exp(e^{-\frac{1}{3}t}x_2X_2+e^{-\frac{1}{3}t}x_3X_3+e^{-\frac{1}{6}t}y_1Y_1+e^{-\frac{1}{6}t}y_2Y_2+zZ). \nonumber 
\end{align}

Let us denote $r = s'-s$ and $\tilde {x_i} = {x_i}' - {x_i}, \tilde {y_i} = {y_i}' - {y_i}$ and $\tilde z = z' - z$. Let $c_\Gamma>0$ denote the distance from the identity of the smallest non-zero element of the lattice $\Gamma$. Let us assume that 
\begin{equation}\label{cond;xy}
{|x_i|}, {|x_i'|}, {|y_i|}, {|y_i'|} \leq c_\Gamma / 4
\end{equation}
so that $\tilde x_i, \tilde y_i \in  [-\frac{c_\Gamma}{2}, \frac{c_\Gamma}{2}]$. 


For all $t \geq 0$ and $s \in [0,T]$, let us adopt the notation 
\begin{align*}
&\tilde{x_2}(t,s) = \tilde{x_2}, \quad \tilde{x_3}(t,s) = \tilde{x_3}, \quad  \tilde{y_1}(t,s) = \tilde{y_1}+e^{\frac{5}{6}t}s\tilde{x_2}\\
& \tilde{y_2}(t,s) = \tilde{y_2}+\alpha e^{\frac{5}{6}t}s\tilde{x_3}+1/2e^{-1/2t}(x_2x_3' - x_2'x_3).
\end{align*}
From the identity in formula (\ref{equality}), we derive the identity
\[
\exp(re^t V)\exp(e^{-\frac{1}{3}t}\tilde x_2\bar{X}_2+e^{-\frac{1}{3}t}\tilde x_3\bar X_3 + e^{-\frac{1}{6}t}\tilde y_1\bar Y_1 + e^{-\frac{1}{6}t}\tilde y_2\bar Y_2) \in  x^{-1} \Gamma x.
\]

Projecting the above identity on the base torus, we obtain
\begin{equation}\label{188}
\exp(re^t\bar V)\exp(e^{-\frac{1}{3}t}\tilde x_2\bar{X}_2+e^{-\frac{1}{3}t}\tilde x_3\bar X_3) \in \overline \Gamma.
\end{equation}
which implies that $re^t$ is return time for the projected toral linear flow at most distant from $e^{-t/3}{c_\Gamma}/{2}$.
Let $R_t(x,T)$ denote the set of $r \in [-T,T]$ such that the equation (\ref{188}) on projected torus has a solution $\tilde x_2, \tilde x_3 \in [-\frac{c_\Gamma}{2}, \frac{c_\Gamma}{2}]$. 
Then for every $r \in R_t(x,T)$, the solution $\tilde x_i: = \tilde x_i(r)$ of the identity in formula (\ref{equality}) is unique. Given $r \in R_t(x,T)$, let $\mathcal{S}{(r)}$ be the set of $s \in [0,T]$ such that there exists a solution of identity (\ref{equality}) satisfying \eqref{cond;xy}. 

Recall that $w_{\Omega_t(r)}(s)$ is the (inner) width function along the orbit $\phi_x(s, \cdot)$ for $s \in [0,T]$.
\begin{lemma}\label{lem;width22} The following average-width estimate holds: for every $T >1$, there exists constant $C_\alpha>0$ such that
\begin{equation}
\frac{1}{T}\int_{0}^{T}\frac{ds}{w_{\Omega_t(r)}(s)} \leq \frac{1024}{c_\Gamma^2}\frac{C_\alpha}{e^{\frac{2}{3}t}\norm{(\tilde x_2(r), \tilde x_3(r))}}.
\end{equation}
\end{lemma}
\begin{proof}
We approximate average width by estimating counting close return orbits. By definition $\mathcal{S}{(r)}$ is a union of intervals $I^*$ of length at most 
$$\max\{c_\Gamma |\tilde x_2(r)|^{-1}e^{-5t/6}/2, c_\Gamma |\alpha\tilde x_3(r)|^{-1}e^{-5t/6}/2 \}.$$
  To count the number of such intervals, we will choose certain points where the distance is minimized. As long as $|\tilde x_2(r)| \geq e^{-5t/6}$, for each component $I^*$ of $\mathcal{S}{(r)}$, there exists $s^* \in I^*$ solution of the equation $\tilde{y_1}(t,s) = \tilde{y_1}(r)+e^{\frac{5}{6}t}s\tilde{x_2}(r)$. The same argument holds for $|\tilde x_3(r)|$.
Let $\mathcal{S}^*(r)$ be the set of all such solutions. Its cardinality can be estimated by counting points.
\medskip

\emph{Claim.} There exists constant $C_\alpha>0$ such that
\begin{equation} 
\# \mathcal{S}^*(r) \leq c_\Gamma^{-1} C_\alpha \norm{(\tilde x_2(r), \tilde x_3(r))}e^{\frac{1}{6}t}T.
\end{equation}

\proof 
Note that $s^*$ is a point which minimizes the distance between orbit and its close return, say 
$$\min_{s}\max\{|\tilde{y_1}(t,s)|, |\tilde{y_2}(t,s)| \}.$$ 

 If either distance $|\tilde{y_1}(t,s)|$ or $|\tilde{y_2}(t,s)|$ dominates another, then it reduces to simply finding a solution to single equation. For other case, we assume $|\tilde{y_1}(t,s)| = |\tilde{y_2}(t,s)|$ and find $s$ that satisfies equation. 
We distinguish following two cases but in either case, we can restrict either $\tilde y_1(r) = 0$ or $\tilde y_2(r) = 0$ for convenience. 

By solving equation $|\tilde{y_1}(t,s)| = |\tilde{y_2}(t,s)|$, there exists $s^* \in I^*$ such that
\[
s^* = \begin{cases}
\dfrac{e^{-\frac{5}{6}t}(\tilde{y_2}+1/2e^{-\frac{1}{2}t}(x_2x_3' - x_2'x_3))}{\tilde x_2(r) - \alpha \tilde x_3(r)} \quad \text{if} \  \tilde{y_1}(t,s) = \tilde{y_2}(t,s)\\
\dfrac{e^{-\frac{5}{6}t}(\tilde{y_2}+1/2e^{-\frac{1}{2}t}(x_2x_3' - x_2'x_3))}{\tilde x_2(r) + \alpha \tilde x_3(r)}\quad \text{if} \  \tilde{y_1}(t,s) = -\tilde{y_2}(t,s).
\end{cases}
\]
%

From the bound 
\[|\tilde{y_2}+1/2e^{-\frac{1}{2}t}(x_2x_3' - x_2'x_3)| \leq c_\Gamma/2 + c^2_\Gamma/16 ,\] 
we obtain  
$$\# \mathcal{S}^*(r) \leq \begin{cases} 2c_\Gamma^{-1}|\tilde x_2(r) - \alpha\tilde x_3(r)| e^{\frac{1}{6}t}T  &  \text{ if }  \tilde x_2(r)\tilde x_3(r)<0;\\
2c_\Gamma^{-1}|\tilde x_2(r) + \alpha\tilde x_3(r)| e^{\frac{1}{6}t}T  &  \text{ if } \tilde x_2(r)\tilde x_3(r)>0.
\end{cases} 
 $$
Thus we prove the claim. 
\qed 
\medskip



\begin{figure}
\includegraphics[scale=0.65]{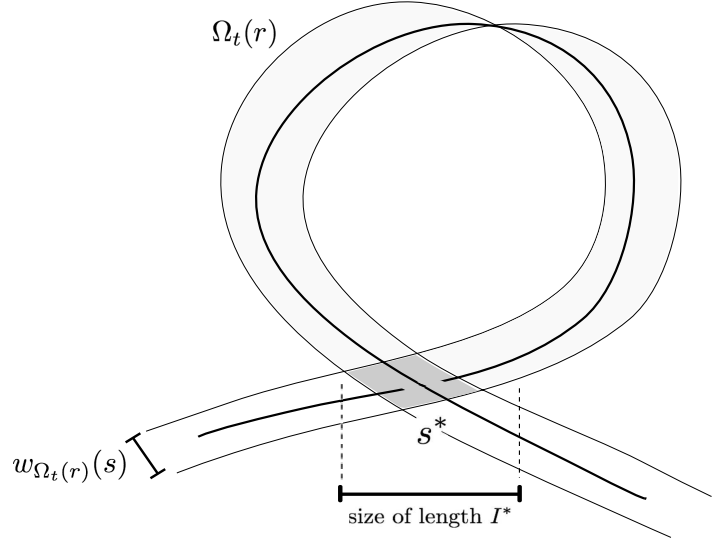}
\caption{Illustration of width function and related quantities}
\centering
\end{figure}

For every $r \in R_t(x,T)$ and every $s \in [0,T]$, we define function
$$\delta_{r}(t,s) = \begin{cases} \frac{1}{16} \norm{(\tilde x_2(r), \tilde x_3(r))} |(s-s^*)e^\frac{5t}{6}| & \text{ for s} \in I^* \text{ with } |s-s^*| \geq e^{-\frac{5}{6}t};\\
\frac{1}{16} \norm{(\tilde x_2(r), \tilde x_3(r))} & \text{for s} \in I^* \text{ with } |s-s^*| \leq e^{-\frac{5}{6}t};\\
\frac{c_\Gamma}{16} & \text{for all s} \in [0,T] \backslash \mathcal{S}(r)
\end{cases}$$
and set
$$\Omega_t(r): = \{(s,x_2,x_3,y_1,y_2,z)  \mid \max \{|x_2|,|x_3|,|y_1|,|y_2|\} < \delta_r(t,s), |z| <  c_\Gamma/16 \}. $$
Now we define \emph{the set of narrow width} by
$$\Omega_t(x,T) := \bigcap_{r\in R_t(x,T)}\Omega_t(r). $$

Under above construction, the map $\phi_x$ is injective on $\Omega_t(x,T)$. The open set $\Omega_t(r) \cap \Omega_t(-r) $ are narrowed near both endpoints of the return time $r$ so that their images in $M$ have no self-intersections under return times $r$ and $-r$.


By the definition of inner width and by construction of the set $\Omega_t(r)$ we have that 
$$w_{\Omega_t(r)}(s) = c_\Gamma \delta_r(t,s)^2, \ \forall s \in [0,T].$$
It follows that for every subinterval $I^* \subset S(r)$ we have (using definition of $\delta_{r}$)
$$\int_{I^*}\frac{ds}{w_{\Omega_t(r)}(s)} \leq \frac {512c_\Gamma^{-1}}{e^\frac{5t}{6}\norm{(\tilde x_2(r), \tilde x_3(r))}^2}. $$

By the upper bound on the length of interval $I^*$ and on the cardinality of the set $\mathcal{S}^*(r)$ we finally derive the conclusion.
\end{proof}
\medskip

Recall from Definition \ref{simultdio}, we choose simultaneously Diophantine number $\alpha \in \R^2\backslash\Q^2$ of exponent $\nu \geq 1$.

\begin{lemma}\label{simul} Given Diophantine condition of exponent $\nu \geq 1$, there exists a constant $C_{\alpha}:= C(\alpha)>0$ such that all solutions of formula (\ref{188}) satisfy the following lower bound
$$\norm{(\tilde x_2, \tilde x_3)}_{\Z^2} \geq C_{\alpha}e^{(\frac{1}{3}-\frac{\nu}{2})t}r^{-\frac{\nu}{2}}.$$
\end{lemma}
\proof 
By projected identity (\ref{188}) on base 3-torus, assume
$$
(re^t, re^t \alpha  +  e^{-t/3} \tilde x_2,re^t \beta  +  e^{-t/3} \tilde x_3)  \in \Z^3.
$$
Then we set $re^t =q \in \Z$ and there exists $(p_1,p_2)\in \Z^2$ such that
$ p_1-q\alpha = e^{-t/3} \tilde x_2,\  p_2-q\beta = e^{-t/3} \tilde x_3$. By the Diophantine condition, there exists a constant $C(\alpha)$ such that
\begin{align*}
\norm{(\tilde x_2, \tilde x_3)}_{\Z^2} & = e^{\frac{1}{3}t}\norm{(p_1-q\alpha,p_2-q\beta)}_{\Z^2}\\
& = e^{\frac{1}{3}t}\norm{(q\alpha,q\beta)}_{\Z^2}\\
& = e^{\frac{1}{3}t}\norm{q(\alpha)}_{\Z^2}\\
& \geq  C_{\alpha}e^{\frac{1}{3}t}q^{-\frac{\nu}{2}}
\end{align*}
which proves the statement.
\qed\\

 For every $n \in \N$, let $R^{(n)}_t(x,T) \subset  R_t(x,T) $ characterized by 
$$\max \{|\tilde x_2(r)|, |\tilde x_3(r)|\} \in (\frac{c_\Gamma}{2^{n+1}}, \frac{c_\Gamma}{2^n}].$$

\begin{lemma}\label{rn}
For all $\epsilon>0$, if the frequency of the projected linear flow satisfies Diophantine condition of exponent $ \nu =  \sqrt{2}+\epsilon$, then there exists $C_\epsilon >0$ such that 
\begin{equation}
\#R^{(n)}_t(x,T) \leq C_{\epsilon }(\bar V)T\frac{c_\Gamma}{2^n}e^{\frac{2}{3}t+\frac{\epsilon t}{2}}.
\end{equation}
\end{lemma}
\begin{proof} Under a Diophantine condition of exponent $\nu \geq 1$, from inequality (\ref{rdio}) and definition of $R^{(n)}_t(x,T)$, we have following:
\begin{equation}\label{rt}
\#R^{(n)}_t(x,T) \leq C_{\nu }(V)\max\{(Te^t)^{1-\frac{1}{\nu}}, Te^t\frac{c_\Gamma}{2^n}e^{-\frac{5t}{6}}\}. 
\end{equation}

It suffices to show $(Te^t)^{1-\frac{1}{\nu}} $ is less than or equal to the desired bound. From Lemma \ref{simul}, 
\begin{align*}
(Te^t)^{1-\frac{1}{\nu}}/ \norm{(\tilde x_2, \tilde x_3)}  & \leq (Te^t)^{1-\frac{1}{\nu}} C_{\alpha}e^{(-\frac{1}{3}+\frac{\nu}{2})t}r^{\frac{\nu}{2}}\\
& \leq T^{1+\frac{\nu}{2}-\frac{1}{\nu}}C_{\alpha}e^{(\frac{2}{3}+\frac{\nu}{2}-\frac{1}{\nu})t}.
\end{align*}

Limiting $\nu \rightarrow \sqrt{2}$,
\begin{equation*} 
(Te^t)^{1-\frac{1}{\nu}}/ \norm{(\tilde x_2, \tilde x_3)} \leq  C_{\alpha}Te^{(\frac{2}{3}+\frac{\epsilon}{2})t}.
\end{equation*}

Approximating $\norm{(x_2,x_3)} \sim 1/ 2^n$,
\begin{equation*} 
(Te^t)^{1-\frac{1}{\nu}} \leq  C_{\epsilon }(\bar V)T\frac{c_\Gamma}{2^n}e^{(\frac{2}{3}+\frac{\epsilon}{2})t}.
\end{equation*}
\end{proof}

By combining counting return time and width estimates, we obtain uniform bound.
\begin{proposition}\label{716} There exists a constant $C_\epsilon(V)>0$ such that 
$$ \frac{1}{T}\int_{0}^{T}\frac{ds}{w_{\Omega_t(x,T)}(s)} \leq C_{\epsilon}(V)e^{\epsilon t}. $$
\end{proposition}
\proof
By Lemma \ref{lem;width22} and \ref{rn},
\begin{align*}
\frac{1}{T}\int_{0}^{T}\frac{ds}{w_{\Omega_t(x,T)}(s)} & \leq \frac{1}{T}\sum_{r \in R_t(x,T)}\int_{0}^{T}\frac{ds}{w_{\Omega_t(r)}(s)}\\
& \leq \frac{1}{T}\sum_{r \in R_t(x,T)}\left(\frac{1024}{c_\Gamma^2}\frac{C_\alpha}{e^{\frac{2}{3}t}\norm{(\tilde x_2(r), \tilde x_3(r))}}\right) \\
& \leq C_{\epsilon}(V) e^{\epsilon t}.
\end{align*}
\qed

Denote \emph{average width} of orbit segment with length 1  
\begin{equation}
w_{\F(t)}(x) := \sup\{w_{\F(t)}(y,1) \mid x \in \{y\exp(tV) \mid t \in [0,1]\}\}. 
\end{equation}

\begin{corollary}\label{roth}  Let $\phi^V_t$ be a nilflow on step-3 strictily triangular $M$ generated by $V$ such that projected flow on $\T^3$ satisfies Diophantine condition of Roth type with $\nu \in [1, \sqrt{2}+\epsilon]$. For every $\epsilon >0$ there exists a constant $C_{\epsilon}(V)>0$ such that
$$w_{\F(t)}(x) \geq C_{\epsilon}(V)^{-1}e^{-\epsilon t},\quad \text{for all } (x,t) \in M \times \R^+.$$
\end{corollary}

\textit{Proof of Theorem \ref{step3}.} By Corollary \ref{roth}, it goes without quoting Good points technique and Lyapunov norm. Improved bound of remainder term $R$ in Theorem \ref{48} can be obtained. 
\begin{equation}
|R(g)|_{-r} \leq C_r(1+\delta_\OO^{-1})^{r-2}T^{-1}.
\end{equation}

We revisit backward iteration scheme introduced in proof of Theorem \ref{T_i}. We have
\begin{align}
\begin{split}
|\gamma|_{-r,0} & \leq  |D_0|_{-r,0} + |R_0|_{-r,0} \\
& \leq |D_0|_{-r,0} + C_r(1+\delta_\OO^{-1})^{r-2}T^{-1}
\end{split}
\end{align}
and
\begin{align}
|D_0|_{-r,0} &\leq |D_N|_{-r,0} + \sum_{j=1}^{N}|R'_{j-1}|_{-r,0}. 
\end{align}

Changing the length to 1 and by uniform width bound from Corollary \ref{roth},
\begin{equation}
|D_N|_{-r,0} \leq C_{r}T^{-1/12}  |D_N|_{-s,\F(t_N)} \leq Cw_{\F(t_N)}(x,1)^{-1/2} \leq C_\epsilon T^\epsilon
\end{equation}

Then, by inductive argument resembling (\ref{DD67}),
\begin{align}
|D_{0}|_{-s,0} & \leq  C_{r,s}T^{-1/12} \left( |D_N|_{-r,\F(t_N)} + \sum_{j=1}^{N}C_jT^{1/12}_j|R_{j-1}|_{-r,\F(t_{j-1})}  \right).
\end{align}
Therefore
$$|\gamma|_{-r,0} \leq C_r(\OO)T^{-1/12+\epsilon}.  $$
\\
Finally, we glue all the functions on irreducible representation $H_{\OO}$, which only increase the regularity accordingly. 
\qed

\section{Application : Mixing of nilautomorphism.}\label{sec;8}
In this section, as a further application of main equidistribution results, we verify 
an explicit bound for the rate of exponential mixing of hyperbolic automorphism relying on renormalization argument.  

Let $\mathfrak{F}_{2,3}= \{X_1, X_2, Y_1, Z_1, Z_2\}$ be step 3 free nilpotent Lie algebra with two generators with commutation relations
$$[X_1,X_2] = Y_1, \quad [X_1,Y_1] = Z_1, \quad [X_2,Y_1] = Z_2.$$ 

The group of automorphism on Lie algebras induces automorphism on the nilmanifold
$$Aut(\n) = \left\{\begin{bmatrix} A &  \\
 & 1 &   \\
 &  & A \\
\end{bmatrix}, \ A \in SL(2,\Z)\right\}$$
and we consider a hyperbolic automorphism $T$ with an eigenvalue $\lambda >1$ with corresponding eigenvector $V = X_1 + \alpha X_2$ satisfying Diophantine condition $(1,\alpha)$ on base torus $\T^2$. By direct computation, the following renormalization holds :
$$T\circ\exp(tV) = \exp(t\lambda V)\circ T.$$

\begin{thm} Let $(\phi_{V}^t)$ be a nilflow on 3-step nilmanifold $M = \mathfrak{F}_{2,3} / \Gamma$ such that the projected toral flow $(\bar \phi_V^t)$ is a linear flow with frequency vector $v:= (1,\alpha)$ in Roth-type Diophantine condition (with exponent $\nu = 1+\epsilon$ for all $\epsilon>0$). For every $s> 12$, there exists a constant $C_s$ such that for every zero-average function $f \in W^s(M)$, for all $(x,T) \in M \times \R$, we have
\begin{equation}\label{81}
\left|\frac{1}{T}\int_0^T f\circ \phi_V^t(x)dt \right| \leq C_sT^{-1/6+\epsilon} \norm{f}_s.
\end{equation}
\end{thm} 
The detailed computation follows in the similar way from the section \ref{7}. The only difference with step 3 filiform case \cite{F16} is that it has an extra element in center which is redundant in actual calculation on width, only raising required regularity of zero-average function. 

The proposition below is firstly proved by A. Gorodnik and R. Spatzier in \cite{AR14}.
%
%
%

\begin{proposition} Hyperbolic nilautomorphism $T$  is exponential mixing.
\end{proposition}
\proof
Let $f, g \in C^1(M)$ be smooth. Define $\langle f,g \rangle = \int_M fg d\mu. $
Since Haar measure is invariant under $\phi_V^t$,
$$\langle f \circ T^n , g\rangle  = \int_0^1 \langle f \circ T^n \circ \phi_V^t, g \circ \phi_V^t\rangle dt.$$

By integration by parts,
\begin{align}
\langle f \circ T^n , g\rangle & = \langle \int_0^1 f \circ T^n \circ \phi_V^t dt ,  g \circ \phi_V^t \rangle \\
& - \int_0^1\langle \int_0^t f \circ T^n \circ \phi_V^s ds ,  Vg \circ \phi_V^t \rangle dt.
\end{align}

Therefore,
\begin{align}
\langle f \circ T^n , g \rangle & = (\norm{g}_\infty + \norm{Vg}_\infty) \int_M\sup_{s \in [0,1]} \left| \int_0^s f\circ T^n \circ \phi_V^{t} dt \right|d\mu.
\end{align}

By renormalizing the flow, 
$$T^n \circ \phi_V^t = \phi_V^{\lambda^nt} \circ T^n$$
and
\begin{align*}
 \int_0^s f\circ T^n \circ \phi_V^{t}(x) dt & = \int_0^s f \circ  \phi_V^{\lambda^nt} \circ T^n(x) dt\\
 & = \frac{1}{\lambda^n}\int_0^{\lambda^ns} f \circ  \phi_V^{t} \circ T^n(x) dt.
\end{align*}

Therefore, by the result of equidistribution (\ref{81}), 
\begin{align}
\langle f \circ T^n , g \rangle & \leq  \lambda^{(-1/6+\epsilon)n}\norm{f}_s (\norm{g}_\infty + \norm{Vg}_\infty) \rightarrow 0.
\end{align}
\qed

\appendix
\section{}
In this appendix, we introduce specific example of nilpotent Lie algebra  which goes beyond our approach introduced in the section \ref{sec;AWE}.

\subsection{Free group type of step 5 with 3 generators.}
In this example, we will show the failure of transversality condition. This only means that we cannot apply our theorem but we do not know whether the conclusion holds or not.

Let $\mathfrak{F}_{n}$ be free nilpotent Lie algebra with $n$ generators and $(\mathfrak{F}_{n})_{k+1}$ be $k+1$th subalgebra in central series, following notation in \eqref{eqn;liealg}.
Denote $\mathfrak{F}_{n,k}:= \mathfrak{F}_{n}/(\mathfrak{F}_{n})_{k+1}$ quotient of free algebra with $n$ generators $\mathfrak{F}_{n}$ and it is finite dimensional. 

\begin{definition}\label{gen}  Let $\n$ be nilpotent Lie algebra satisfying \emph{generalized transversality condition} if there exists basis $(X_\alpha, Y_\Lambda)$
of $\n$ for each irreducible representation $\pi_\Lambda^{X_\alpha}$ such that
\begin{equation}\label{trans}
 \langle \mathfrak{G}_\alpha \rangle \oplus \ran(\ad_{X_{\alpha}})+C_\II(\pi_\Lambda^{X_{\alpha}}) = \n
\end{equation} 
where $C_\II(\pi_\Lambda^{X_{\alpha}}) = \{Y \in \II \mid \Lambda([Y,X_{\alpha}]) = 0\}$. 
\end{definition}
Generalized transversality condition implies existence of completed basis for each irreducible representation $\pi_\Lambda^{X_\alpha}$ of non-zero degree. That is, given adapted basis $\F = (X,Y_1,\cdots, Y_a)$, there exists reduced system $\bar{\F} = (X, Y_1',\cdots, Y'_{a'})$ satisfying transversality condition (\ref{2}) and $\pi_\Lambda^X(Y'_{m}) = 0$ for all $a' \leq m \leq a$.
\medskip

Now we will investigate an example that fails transversality condition as well as that in the sense of representation. 

Let $\F = (X, Y^{(j)}_{i})$ be basis of $\mathfrak{F}_{5,3}$ with generators $\{X_1, X_2, X_3\}$ with the following relations:
$$X_1 \quad X_2 \quad X_3$$
$$Y_1 \quad Y_2 \quad Y_3$$
$$Z_1 \quad Z_2 \quad \cdots \quad Z_8 \quad Z_9$$
with
$$[X_1, X_2] = Y_1, \quad [X_2, X_3] = Y_2, \quad [X_1, X_3] = Y_3$$
$$[X_1, Y_1] = Z_1, \quad [X_1, Y_2] = Z_2, \quad [X_1, Y_3] = Z_3$$
$$[X_2, Y_1] = Z_4, \quad [X_2, Y_2] = Z_5, \quad [X_2, Y_3] = Z_6 $$
$$[X_3, Y_1] = Z_7, \quad [X_3, Y_2] = Z_8, \quad [X_3, Y_3] = Z_9 $$
and rest of elements are generated commutation relations with these. In general, we write elements $Y^{(i)}_j \in \n_{i}\backslash \n_{i+1}$ and $Y^{(5)}_i \in Z(\n)$ for all $i$.
By Jacobi-identity
$$[X_1,[X_2, X_3]] + [X_2,[X_3, X_1]] + [X_3,[X_1, X_2]] = 0 \iff Z_2 - Z_6 + Z_7 = 0.$$

 For fixed $\alpha_i$ and $\beta_i$, let
$$V = X_1 + \alpha_2X_2 + \alpha_3X_3 +\beta_1Y_1+\beta_2Y_2+\beta_3Y_3$$
and set $\II$ ideal of $\mathfrak{F}_{5,3}$ codimension 1, not containing $V$.





\begin{proposition} $\mathfrak{F}_{5,3}$ does not satisfy generalized transversality condition for some irreducible representation.
\end{proposition}
\proof
To find centralizer in Lie algebra, for $a_i, b_i \in \R$, set
$$[V,X] = 0 \iff X = a_1X_1+ a_2X_2 + a_3X_3+b_1Y_1+b_2Y_2+b_3Y_3 + c_1Z_1 + \cdots + c_8Z_8.$$
Then, it contains 
\begin{align*}
& (a_2-\alpha_2a_1)Y_1 + (\alpha_2a_3 - \alpha_3a_2)Y_2 + (a_3 - \alpha_3a_1)Y_3 \\
& + (b_1 -\beta_1 a_1)Z_1 + (b_2 -\beta_2 a_1)Z_2 + (b_3 -\beta_3 a_1)Z_3 + \cdots = 0. 
\end{align*}
By linear independence, all the coefficients vanish and it remains
$$a_1X_1 + a_2X_2 + a_3X_3 = a_1(X_1 + \alpha_2X_2 + \alpha_3X_3)$$
$$b_1Y_1+b_2Y_2+b_3Y_3 = a_1(\beta_1Y_1+\beta_2Y_2+\beta_3Y_3) $$ 

Therefore, there is no non-trivial element in $C_\II(V) \cap  {\n_2 \backslash \n_3}$. Since range of $\ad_V$ has rank 2, this model does not satisfy transversality condition in the Lie algebra level.

Now, we verify \emph{generalized transversality condition} is not satisfied on some irreducible representation. By Schur's lemma, an irreducible representation $\pi_{\Lambda}^{V}$  acts as a constant on the center $Z(\n)$. 

Assume  $\pi_*(W_i) = s_i I \neq 0$ for some $W_i \in Z(\n)$. Then, it is possible to choose  element $L_i \in {\n_2 \backslash \n_3}$ such that 
$$
\begin{cases}
\pi_*([V,L_1]) = (a_1t^2+a_2t+a_3)\\
\pi_*([V,L_2]) = (b_1t^2+b_2t+b_3)\\
\pi_*([V,L_3]) = (c_1t^2+c_2t+c_3)\\
\end{cases}
$$
with $(a_i,b_i,c_i)$ are non-proportional for each $i$, and 
$$\pi_*(\ad^3_V(L_i)) = \pi_*(W_i) \neq 0.$$

However, on given irreducible representation, any linear combination of $L_1, L_2$ and $L_3$ does not give any trivial relation. If $s_1 L_1 + s_2 L_2 + s_3 L_3 \in C_\II(\pi_\Lambda^{V})$, then
\begin{align*}
& \pi_*([V, s_1 L_1 + s_2 L_2 + s_3 L_3]) \\
& = s_1(a_1t^2+a_2t+a_3) + s_2(b_1t^2+b_2t+b_3) + s_3(c_1t^2+c_2t+c_3)\\
& = (s_1a_1+s_2b_1+s_3c_1)t^2 + (s_1a_2+s_2b_2+s_3c_2)t + (s_1a_3+s_2b_3+s_3c_3) = 0.
\end{align*}
The system of equations has trivial solution $(t=0)$ by linear independence of each coefficients. Then, there does not exist any element of $\n_2 \backslash \n_3$ that has degree 0. However, range of $\ad_V$ has rank 2 and generalized transversality condition cannot be satisfied in this example. \qed


\bibliographystyle{plain}

\begin{thebibliography}{1}

\bibitem[AGH63]{AGH63} L. Auslander, L. Green, and F. Hahn, Flows on homogeneous spaces, with the assistance
of L. Markus and W. Massey, and an appendix by L. Greenberg. \textit{Annals
of Mathematics Studies}, No. 53, Princeton University Press, Princeton, N.J., (1963).



\bibitem[BDG15]{BDG15} {J. Bourgain, C. Demeter and L. Guth}, Proof of the main conjecture in Vinogradov's mean value theorem for degrees higher than three, \textit{Annals of Mathematics}, (2016) 633--682 

\bibitem[CF15]{CF15} S. Cosentino and L. Flaminio, Equidistribution for higher-rank Abelian actions on Heisenberg nilmanifolds. \textit{J. Mod. Dyn}. \textbf{9} (2015), 305 - 353. 
\bibitem[CG90]{CG90} L. Corwin and F. P. Greenleaf, \textit{Representations of nilpotent Lie groups and their applications. Part 1: Basic theory and examples}, Cambridge studies in advanced mathematics, vol. 18, \textit{Cambridge University Press, Cambridge}, (1990).

\bibitem[DAS]{DAS} B.A Dubrovin, A.T Fomenko, and S.P Novikov. Modern geometry-methods and applications. Part II: The geometry and topology of manifolds. Vol. 104. \textit{Springer Science \& Business Media}, (2012).

\bibitem[FF03]{FF03} L. Flaminio and G. Forni, Invariant distributions and time averages for horocycle flows, \textit{Duke Math. J.} \textbf{119} (2003), no. 3, 465-526.


\bibitem[FF06]{FF06} \bysame, Equidistribution of nilflows and applications to theta sums. \textit{Ergodic Theory Dynam. Systems} \textbf{26} (2006), no. 2, 409-433. 



\bibitem[FF07]{FF07} \bysame, On the cohomological equation for nilflows \textit{J. Mod. Dyn}. \textbf{1} (2007), no. 1, 37-60.


\bibitem[FF14]{FF14} \bysame, {On effective equidistribution for higher step nilflows. Preprint: arXiv:1407.3640}.



\bibitem[FFT16]{FFT16} L.Flaminio, G. Forni, and J. Tanis. Effective equidistribution of twisted horocycle flows and horocycle maps. \emph{Geom. Funct. Anal}. \textbf{26} (2016), no. 5, 1359 -1448


\bibitem[F16]{F16} G. Forni. Effective Equidistribution of Nilflows and Bounds on Weyl Sums. \textit{Dynamics and Analytic Number Theory} \textbf{437} (2016): 136.




 
\bibitem[GT12]{GT12} B. Green and T. Tao, The quantitative behaviour of polynomial orbits on nilmanifolds, \textit{Annals of Math}. \textbf{175} (2012), 465 - 540.

\bibitem[GS14]{AR14} A. Gorodnik and R. Spatzier, Exponential mixing of nilmanifold automorphisms. \textit{J. Anal. Math}. \textbf{123} (2014), 355-396. 


\bibitem[H73]{H73} J. Humphreys. Introduction to Lie algebras and representation theory. Vol. 9. \textit{Springer Science and Business Media}, 1973.

\bibitem[R72]{Rag72} M. S. Raghunathan, Discrete subgroups of Lie groups, \textit{Springer-Verlag}, New York, Heidelberg,
1972, Ergebnisse der Mathematik und ihrer Grenzgebiete, Band 68.


\bibitem[T15]{T15}T.D Wooley.  Perturbations of Weyl sums. \textit{International Mathematics Research Notices} 2016.9 (2015): 2632-2646.













\end{thebibliography}

\end{document}